\documentclass[12pt]{amsart}
\usepackage[all]{xy} 
\usepackage[usenames,dvipsnames]{pstricks}
\usepackage[hypertex]{hyperref}
\usepackage{amssymb,mathrsfs,euscript,enumitem,amsmath,bbold}
\usepackage[backgroundcolor=yellow,colorinlistoftodos]{todonotes}
\usepackage{lettrine}
\usepackage{fourier}

\catcode`\@=11 %\def\today{January 27, 2023}
\def\@evenfoot{\rule{0pt}{20pt}[\today] \hfill [{\tt \jobname.tex}]}
\def\@oddfoot{\rule{0pt}{20pt}{[\tt \jobname.tex}]\hfill [\today]}
\catcode`\@=13
\newtheorem{theorem}{Theorem}%[section]
\newtheorem{proposition}[theorem]{Proposition}

\theoremstyle{definition}

\newtheorem{definition}[theorem]{Definition}
\newtheorem{example}[theorem]{Example}

\newtheorem{remark}[theorem]{Remark}

\renewcommand\vec[1]{\overrightarrow{#1}}

\def\RMod{\hbox{$R$-{\tt Mod}}}
\def\cO{{\check {\ttO}}}
\def\Cycl{{\it Cycl}}
\def\Tau{\hbox{${\mathcal T}$\hskip -.3em}}
\def\ttA{{\tt A}}
\def\Span#1{{\rm Span}{\{#1\}}}
\def\Setp{{\tt Set}^\circ}
\def\St{\hbox{\it St\/}}
\def\Mikinka#1#2#3{{#1\oplus #2 \oplus #3}}
\def\Alicek#1#2#3{{\det(#1)\ot \det(#2) \ot \det(#3)}}
\def\fdV{{\tt fdVec}}\def\qfdV{\fdV/{{\mathbb F}}}
\def\Flicek#1#2#3{{\Ch(#1)\! \ot\! \Ch(#2)\! \ot\! \Ch(#3)}}
\def\Iso{{\rm Iso}(\ttQ)}
\def\Xarrow#1{{\stackrel {#1} \longrightarrow}}
\def\epi{ \twoheadrightarrow}
\def\komp{\raisebox{-.2em}{\rule{0pt}{0pt}}}
\def\Im{{\rm Im}}
\def\eq{{\rm Eq}}
\def\Ker{{\rm Ker}}
\def\Cok{{\rm Coker}}
\def\Ch{{\EuScript Ch}}
\def\frakC{{\mathfrak C}}
\def\Cha{{\tt Chaos}}
\def\tO{{\widehat {\tt O}}} 
\def\prop{{\mathbb P}}
\def\Ass{{\EuScript Ass}}
\def\Coll#1#2{\{#1(#2)\}_{#2 \in \bbN}}
\def\ttA{{\tt A}}
\def\CP{{$\cooP$-$\oP$}}
\def\Cat{{\tt Cat}}
\def\sfD{{\EuScript D}}
\def\fifu{{\EuScript F}}
\def\redukce#1{\vbox to 1.5em{\vss\hbox{$#1$}}}
\def\inv#1{{#1}^{-1}}
\def\d{\dFin}
\def\bbN{{\mathbb N}}
\newcommand\cev[1]{\overleftarrow{#1}}
\def\dash{{\hbox{\hskip .15em -\hskip .1em}}}
\def\sFSet{{\tt sFSet}}

\def\bic#1{{\left\|#1\, \right\|}}
\def\cooP{\raisebox{.7em}{\rotatebox{180}{$\EuScript P$}}}
\def\Arr{{\rm Arr(\ttC)}}\def\ArrM{{\rm Arr}(M)}
\def\ttC{{\tt C}}
\def\ii{\raisebox{.7em}{\rotatebox{180}{$i$}}\!}
\def\iii{\raisebox{.7em}{\rotatebox{180}{$i\!i$}}\!}
\def\iiii{\raisebox{.7em}{\rotatebox{180}{$i\!i\!i$}}\!}
\def\rmiiii{\raisebox{.7em}{\rotatebox{180}{iii}}}

\def\oL{{\EuScript L}}\def\oR{{\EuScript R}}
\def\oM{{\EuScript M}}\def\oL{{\oP}}
\def\sou#1{{\mathbf s(#1)}}\def\tar#1{{\mathbf t}(#1)}
 
\def\ttV{{\tt V}}

\def\phi{\varphi}
\def\ttO{{\tt O}}
\def\ttQ{{\tt Q}}
\def\ot{\otimes}
\def\Rada#1#2#3{#1_{#2},\dots,#1_{#3}}
\def\fib{\triangleright}\def\cof{\triangleleft}
\def\oP{{\EuScript P}}
\def\id{{\mathbb 1}}
\def\bbk{{\mathbb k}}

\def\dBq{{\tt BBq(\frakC)}}
\def\bez#1#2{\genfrac{}{}{0pt}{3}{\hbox{$#1$}}{\hbox{$#2$}}}
\def\dFin{\genfrac{}{}{0pt}{3}{\raisebox{-.3em}{\Fin}}{\rule{0em}{.5em}\niF}}
\def\dd{
\raisebox{.8em}{
$\genfrac{}{}{0pt}{3}{\raisebox{-.8em}{$\Delta_{\rm
        alg}$}}{\raisebox{-.8em}{\redukce{\atleD}}}$}
}

\def\niF{\raisebox{.7em}{\rotatebox{180}{$\Fin$}}}
\def\Ateb{\raisebox{.5em}{\rotatebox{180}{$\beta$}}}

\def\atleD{\raisebox{.7em}{\rotatebox{180}{$\Delta_{\rm alg}$}}}

\def\redukce#1{\vbox to .3em{\vss\hbox{#1}}}
\def\frakC{{\mathfrak C}}

\def\rada#1#2{{#1,\ldots,#2}}
\def\Rada#1#2#3{#1_{#2},\dots,#1_{#3}}
\def\Fin{{\tt sFSet}}

\def\Bq{{\tt Bq}}
\def\martin#1\endmartin{\noindent{{\bf Martin:}\ {\color{red} #1} 
    \hfill\rule{10mm}{.75mm}} \break}

\subjclass[2010]{???????}

\makeatletter
\providecommand\@dotsep{5}
\def\listtodoname{List of Todos}
\def\listoftodos{\@starttoc{tdo}\listtodoname}
\makeatother

\title[Bivariant operadic categories]
{Bivariant operadic categories}

\thanks{Supported by the Institute of
Mathematics, Czech Academy of Sciences (RVO 67985840),
and Praemium Academi{\ae} of Martin Markl.}

\catcode`\@=11
\email{markl@math.cas.cz}
\@namedef{subjclassname@2010}{%
  \textup{2020} Mathematics Subject Classification}
\catcode`\@=13


\begin{document}
\bibliographystyle{plain}

\parskip3pt plus 1pt minus .5pt
\baselineskip 15.5pt  plus .5pt minus .5pt

\author[M.\ Markl]{Martin Markl}
\address{The Czech Academy of Sciences, Institute of Mathematics, {\v Z}itn{\'a} 25,
         115 67 Prague 1, The Czech Republic}

\begin{abstract}
We develop a self-dual, bivariant extension of the 
concept of an operadic category, 
its associated operads and their algebras. Our new theory covers, besides all
classical subjects, also generalized
traces and bivariant versions of Kapranov's charades. It is, 
moreover, combinatorially rich and aesthetically pleasing.
\end{abstract}

\maketitle

\tableofcontents

\section*{Introduction}

\lettrine{\color{red} O}{peradic} categories, the
related operads and their algebras
were introduced in~\cite{duodel}, cf.~also Section~1 of the freely
available article~\cite{env}.
Any operadic category $\ttO$ is, by definition,
equipped with the cardinality functor $\ttO
\stackrel{|-|}\longrightarrow \Fin$ 
to the skeletal category of finite sets.  Each morphism $h : S\to T$ in
$\ttO$ has $n$ fibers $\Rada F1n \in \ttO$,  $n \in \bbN$, where $[n] := \{\rada
  1n \}$ is the
cardinality of the target $T$ of $h$. We will express this fact 
by writing 
\begin{subequations}
\begin{equation}
\label{Uz 16 dni!}
\Rada
F1n \ \fib  S \stackrel  h\longrightarrow T.
\end{equation} 
Operadic categories support operads. An {\em $\ttO$-operad \/} is a collection 
$\oP = \{\oP(T)\}_{T \in \ttO}$ of objects of some~symmetric monoidal
category $\ttV = (\ttV,\ot,1)$, with the compositions
\begin{equation}
\label{Zitra jedu do Kolina.}
\gamma_h : \oP(T) \ot \oP(F_1) \ot \cdots \ot \oP(F_n)
\longrightarrow \oP(S),
\end{equation}
associated to each morphism $h: S \to T$ with fibers $\Rada F1n$ 
as in~(\ref{Uz 16 dni!}). 

Finally, operads have algebras; 
a {\em $\oP$-algebra} is collection $A = \{A_c \}_{c \in
\pi_0(\ttO)}$ of objects of $\ttV$ indexed by the set $\pi_0(\ttO)$
of connected components of $\ttO$, equipped with structure operations 
\begin{equation}
\label{Dnes zverejneny propozice na Safari.}
a_T: \oP(T) \ot \bigotimes_{c \in \sou T}A_c \longrightarrow A_{\tar T}, \
T \in \ttO,
\end{equation}
\end{subequations}
where $\sou T$, the {\/\em source\/} of $T$, is the list of connected
components of the fibers of the identity morphism $\id_T : T \to T$,
and where $\tar T$, the {\/\em target\/} of $T$, is the connected
component of $\ttO$ to which~$T$ belongs. The operations~(\ref{Zitra
  jedu do Kolina.}) and~(\ref{Dnes zverejneny propozice na Safari.})
are, of course, subjects of appropriate axioms cf.~the first sections
of~\cite{duodel} or~\cite{env}. 
The situation is captured by the triad
\[
\xymatrix{
\boxed{\hbox {algebras}} \ \ar@{=>}[r]
&
\ \boxed{\hbox {\ operads}} \ \ar@{=>}[r]
&
\boxed{\ \hbox {\ operadic categories}}
}
\]
in which ``$A \Longrightarrow B$'' must be read as \ ``$A$ is governed by $B$.''

The above concept, inspired by Batanin's
$n$-operads~\cite{batanin:globular},  
covers, either as `operads' or as their `algebras,' 
the most common operad-like structures, 
such as  the traditional operads, their variants such as cyclic or modular
operads, and also diverse versions of PROPs such as wheeled properads, 
dioperads, and even more exotic objects such as permutads 
and pre-permutads. However,
it is still not
fully satisfactory for the following reasons. First,
the sources and the target of a given $T\in \ttO$ are
objects of different types -- while the sources very crucially use the
fiber structure of $\ttO$, the target does not refer to it at all.
Second, each $T$ has only one target, so the `operads' of the theory
cannot have operations with multiple outputs. 

Our goal is to modify and extend the standard concept of 
an operadic category so that morphisms will
possess, along with the fibers as before, also the {\em
  cofibers\/}, so instead of~(\ref{Uz 16 dni!}) we would have
something as
\[
\Rada F1n \ \fib S \stackrel h\longrightarrow T \cof \ \Rada D1m. 
\]
 The target of an object $T$ will then be the list of 
connected components of the cofibers of the identity $\id_T : T \to T$.   
As expected, we define cofibers by dualization of the properties of
the fibers. We will
also need  suitable compatibilities between fibers and cofibers. 
The emerging self-dual concepts of  {\em di-} and {\em
  bioperadic\/} categories will be the main subjects 
of this article.

Let us point out some salient features of these two new concepts. 
While operadic categories have operads, dioperadic categories also have
cooperads and new types of `bimodules.' However, 
it turns out that to define
algebras for operads, coalgebras for cooperads, and traces for
bimodules, one needs to impose some additional conditions. Bioperadic
categories are defined as dioperadic categories that
satisfy them.  The landscape is sketched in 
Figure~\ref{Jsem druhy den v Haife.}.
\begin{figure}
\[
\xymatrix{
&\boxed{\hbox {algebras, coalgebras, traces}} \ar@{=>}[d]
\\
\boxed{\hbox {operads, cooperads, bimodules}}\ar@{=>}[d]&\ \boxed{\hbox
  {operads, cooperads, bimodules}}\ar@{=>}[d]\ar@{_{(}->}[l] 
\\
\boxed{\hbox {dioperadic categories}}\   &
\ar@{_{(}->}[l] \
\boxed{\hbox {bioperadic categories}}
}
\]
\caption{\label{Jsem druhy den v Haife.}Dioperadic duad (left) and
  bioperadic triad (right).}
\end{figure}
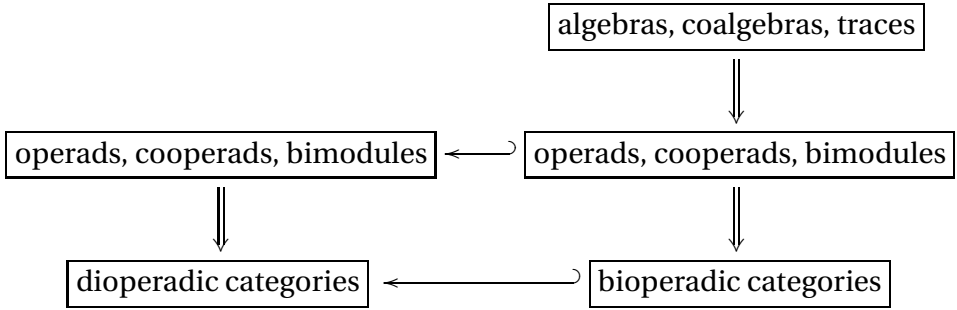
In our setup, traces are objects governed by bimodules in the same way
that algebras
are governed by operads and coalgebras are governed by cooperads. 
The terminology is explained by
Examples~\ref{V Cechach je plno snehu a ja trcim v Izraeli.} 
and~\ref{Vcera jsem byl na schuzi Neuronu.}. By Proposition~\ref{Dnes
  jsem byl s Jarkou na hrbitove.}, 
the bivariant concept presented here does indeed generalize the
theory of unital operadic categories, their operads and algebras,
introduced in~\cite{duodel}.

Abelian categories admit a bioperadic structure,
given by the kernels and the cokernels, having
an interesting, unexpected property. This property is abstracted in
Section~\ref{Jsem pred-predposledni den v Haife.} by the notion  of a 
{\em bicharadic\/} category. The r\^oles of algebras are
played by {\em bicharades}, which are bivariant generalizations 
of the (absolute) Kapranov's charades~\cite[Definition~3.2]{kapranov:langlands}. 
The charadic structure of determinants and Steinberg modules naturally
extends to our bivariant concept.

\vskip .3em
\noindent
{\bf The layout of the paper.}
The article opens with a short section dedicated to a certain property
of fibers in an operadic category, and the induced maps. 
The main body of the article then 
contains two parallel parts. In Sections~\ref{Jsem dnes
  nejak bez energie.} and~\ref{Posloucham ricercar.}
we focus on di- and bioperadic categories, whose morphism have only
one fiber and one cofiber. The general case is postponed to
Sections~\ref{Vsechno je spatne.} and~\ref{Jaroslav Vodrazka nahral ty
  fantasticke veci.}. The correspondence between the individual items
is described~in
\begin{center}
\def\arraystretch{1.3}
\begin{tabular}{|l|l|} 
\hline  \multicolumn{1}{|c|}{{unary  -  Secs.~\ref{Jsem dnes
  nejak bez energie.},~\ref{Posloucham ricercar.}}}
& \multicolumn{1}{c|}{{general - Secs.~\ref{Vsechno je spatne.},~\ref{Jaroslav Vodrazka nahral ty
  fantasticke veci.}}}
\\
\hline \hline 
Definition~\ref{Musim se objednat na endokrinologii za mesic.}  & 
Definition~\ref{Musim se objednat na endokrinologii.}
\\
\hline
Definition~\ref{Musim se objednat na ocni.} &
Definition~\ref{Dnes jsem byl s Jarkou na obede.}
\\
\hline
Definition~\ref{Jarka ma ovoce.} &
Definition~\ref{az do dnu antiky}
\\
\hline
Proposition~\ref{Najdu odvahu se dnes v te strasne zime projet na
    kole?} & Proposition~\ref{zasute refremy kterym uz chybeji slova}
\\ 
\hline
Definition~\ref{Jarka dnes jela do Motola.} &
Definition~\ref{Dnes dam kolecko do opravy.}
\\
\hline
Definition~\ref{Podari se mi ten clanek vubec publikovat?}&
Definition~\ref{Navlekl jsem si gumu do prdlavek.}
\\
\hline
\end{tabular}
\\
\end{center}

We have chosen this arrangement for
expository reasons. While the unary, i.e.\  one fiber\textendash{}one 
cofiber case, exhibits all the salient
features of these novel concepts, it avoids complications caused by
the complex combinatorics of multiple (co)fibers, and the associated numerological
conditions. We assumed that starting directly with the general case would make
the article unreadable; the daring reader may of course 
skip Sections~\ref{Jsem dnes
  nejak bez energie.} and~\ref{Posloucham ricercar.}. 
Between the two parallel parts mentioned above we have inserted
one section on bicharades, which
are unary in nature. In the brief epilogue, as an inspiration for 
future work, we propose possible  modifications of our~setup.

\vskip .3em
\noindent 
{\bf Conventions.} Given a category $\ttC$, we denote by
$\ttC^\dagger$ the opposite category, and by $(-)^\dagger :
\ttC \to \ttC^\dagger$ the corresponding contravariant isomorphism.
If the meaning is clear from the context, we will abuse the
notation and use the same symbols for the maps in $\ttC$ and 
their images in $\ttC^\dagger$.
By $\pi_0(\ttC)$ we denote the set of connected components of $\ttC$
and by $\pi_0(T)$ the connected component to which an object $T \in
\ttC$ belongs.
We will assume that the set of connected components of
categories considered here 
is small with respect to  a sufficiently large surrounding universe.
By a~luff (full read backwards) subcategory we mean a subcategory with
the same objects as the overcategory.

Unless otherwise noted, all operadic categories
considered here are  strict and nonunital. The latter means that we do
not assume the existence of the chosen local terminal objects and the
associated axioms in~\cite[page~1634]{duodel}. 
The simplified axioms for unary, i.e.\ one-fiber,  
operadic categories, sufficient for Sections~\ref{Jsem dnes
  nejak bez energie.}\textendash{}\ref{Jsem pred-predposledni den v Haife.}, 
can be found in~\cite{blob}. 

We use the definition of operads
in which the components of the fibers act,
unlike in~\cite[Definition~1.11]{duodel}, from the right on the 
component of the target. Also, operads in this article act, unlike in~\cite[Definition~1.11]{duodel}, 
on their algebras from
the left. This convention is closer
to the traditional setup.

Finally, we denote by $\bbN$ the set of natural numbers including $0$ and, for
$n \in \bbN$, by $\Sigma_n$ the symmetric group on $n$ elements. 
All algebraic objects will live in a symmetric monoidal
category $\ttV = (\ttV,\ot,1)$.

\vskip .3em
\noindent 
{\bf Acknowledgment.} 
I am indebted to the participants of my seminar held at the Mathematical
Institute in Prague, especially Michael~Batanin and Dominik~Trnka, for many
useful suggestions and comments. Last but not least, I would like to thank
Jaroslav \hbox{Vodr\'a\v zka} for taking me back 50 years in my memories.

\section{Perfect operadic categories}

\lettrine{\color{red} O}{ur} original definition 
of operadic categories~\cite[Section~1]{duodel} 
required the existence of selected local 
terminal objects and the corresponding axioms, i.e.\ the {\/\em unitality\/}. 
In the course of working on~\cite{blob}, it became clear that without 
assuming the existence 
of local terminal objects, the theory takes 
on an unexpected flexibility, including new examples.
We also understood that the unitality is a combination of two conditions
of different characters -- the left and the right~unitality: 

\begin{definition}
\label{Opet cele prepisuji.}
Let $\ttO$ be an operadic category equipped with a family
\begin{equation}
\label{Jsem posledni den v Cuernavace.}
\big\{U_c \in \ttO \ | \ c \in \pi_0(\ttO) \big\}
\end{equation}
of local terminal objects of cardinality $[1]$, such that  $U_c$
belongs to the connected component~$c$ of $\ttO$. 
The category $\ttO$ is {\em
  left unital \/} if the fibers of the identity
automorphism $\id_S:S \to S$ belong to the family~(\ref{Jsem posledni
  den v Cuernavace.}) of the chosen local terminal objects, 
for each $S \in \ttO$.  The category $\ttO$ is {\em right
  unital \/} if the fiber functor restricted to $\ttO/U_c$ is the
domain functor, for each  $c \in \pi_0(\ttO)$.
Finally, $\ttO$ is {\em unital \/} if it
is both left and right unital.
\end{definition}

The property introduced in the following definition is automatically
satisfied by unary operadic categories.

\begin{definition}
\label{Boli mne prave rameno.}
To each operadic category $\ttO$ we associate a substructure $\ttO_F \subset
\ttO$ consisting of the fibers of morphisms of $\ttO$ and of the induced
morphisms between them. We say that $\ttO$ is {\/\em perfect\/} if
$\ttO_F$ is a subcategory of $\ttO$. We then call $\ttO_F$
the {\em fiber subcategory\/} of $\ttO$.
\end{definition}

Definition~\ref{Boli mne prave rameno.} needs an explanation.
Each chain $S \xrightarrow h T \xrightarrow g R$ of morphisms
of~$\ttO$ induces, for $i \in |R|$,
the diagram
\[
\xymatrix@R=-0.3em@C=.8em{F
  \ar[rr]^(.55){h_i}   
&&G \hskip -.5em
\\
\hskip -2em
\hskip 1.9em\raisebox{.7em}{{\rotatebox{270}{$\fib$}}} && 
\hskip -.1em\raisebox{.7em}{{\rotatebox{270}{$\fib$}}} \hskip -.5em
\\
S \ar[rr]^{h}  \ar@/_.9em/[ddddddr]_(.4){gh}   &&
T \ar@/^.9em/[ddddddl]^(.4){g} \hskip -.5em
\\&& 
\\&& 
\\&& 
\\&&
\\ &&
\\
&R &
} 
\] 
in which $F \xrightarrow {h_i} G$ is the induced map between the $i$th
fibers of $gh$ and $g$, respectively. 
The perfectness of $\ttO$ means that the composite of induced
morphisms is again an induced morphism. The~structure~$\ttO_F$ is then
an operadic subcategory of $\ttO$.

\begin{proposition}
\label{Jarka opet nemocna.}
Each right unital operadic category $\ttO$ is perfect, with $\ttO_F =
\ttO$.  
\end{proposition}

The proof is obvious. Later we will see several perfect operadic
categories that are not right unital.

\section{Unary dioperadic categories}
\label{Jsem dnes nejak bez energie.}

\lettrine{\color{red} I} {n} this section we propose a `self-dual closure' of the notion of
a unary operadic category; the non-unary case will be addressed in
Section~\ref{Vsechno je spatne.}.

\subsection{Unary operadic categories}
By definition, the cardinality of
each object of an
unary operadic category $\ttO$ is $[1]$.  Morphisms  thus have 
only one fiber, so~(\ref{Uz 16 dni!}) reduces to
\redukce{$F \, \fib \, S \stackrel h\longrightarrow T$}. The unarity
brings some simplifications to the axioms, which are in this
simplified form presented in~\cite[Section~1]{blob}. 
The fiber structure is in the unary case expressed by a single {\em fiber
  functor\/}  $\fifu :\sfD(\ttO) \to \ttO$,
cf.~\cite[Definition~3]{blob}, 
from the {\em  d\'ecalage\/}
\begin{equation}
\label{Dovolam se ji?}
\sfD(\ttO) :=
\textstyle
\coprod_{c \in \ttO}
\ttO/c
\end{equation}
of the category $\ttO$. The following proposition
describes a particular kind of unary operadic categories.
We will refer to it at several places later in this paper.

\begin{proposition}
\label{Na vikend se pocasi zhorsi.}
Let $\ttO$ be a small category and $r : \ttO \to \ttO$ an endofunctor. 
The rule that defines the fiber of a map $h : S \to T$ in $\ttO$ to be
$r(S)$, that is, $r(S) \fib S \stackrel h\to T$, and the induced map
between its fibers as $h_R := r(h)$  in 
\begin{equation}
\label{Pujdeme s Jarkou na demonstraci?}
\xymatrix@R=-0.3em@C=.8em{r(S)
  \ar[rr]^(.55){h_R:=r(h)}   
&&r(T) \hskip -.5em
\\
\hskip -2em
\hskip 1.9em\raisebox{.7em}{{\rotatebox{270}{$\fib$}}} && 
\hskip -.1em\raisebox{.7em}{{\rotatebox{270}{$\fib$}}} \hskip -.5em
\\
S \ar[rr]^{h}  \ar@/_.9em/[ddddddr]_(.4){gh}   &&
T \ar@/^.9em/[ddddddl]^(.4){g} \hskip -.5em
\\&& 
\\&& 
\\&& 
\\&&
\\ &&
\\
&R &
}
\end{equation}
converts $\ttO$ to a (nonunital) unary operadic category if and only if $r$ is
idempotent, i.e.\  $r^2=r$. 
\end{proposition}

\begin{proof}
The proposition can be proved by checking the axioms
of an operadic category for the fiber functor $\fifu:= rd_\ttO :
\sfD(\ttO) \to \ttO$, with $d_\ttO : \sfD(\ttO) \to \ttO$ the domain
functor, but we offer a more elegant approach.
The d\'ecalage assembles to an endofunctor $\sfD : \Cat \to \Cat$
on the category of small categories, and  the domain functor 
$d_\ttO : \sfD(\ttO) \to \ttO$ gives rise to a
natural transformation
$d: \sfD \rightsquigarrow \id_\Cat$. By~\cite[Corollary~7]{blob}, the
  pair $(\sfD,\mu)$ with $\mu : = \sfD(d) : \sfD^2 \rightsquigarrow
  \sfD$ is a nonunital monad in $\Cat$. Proposition~8
loc.\ cit.\ then states that the algebras for this monad are unary
(nonunital) operadic categories. 

A functor $\fifu : \sfD(\ttO) \to \ttO$ is thus the fiber functor of a unary
operadic category if it is an algebra for the nonunital monad $(\sfD,\mu)$,
i.e.\ if $\fifu\sfD(\fifu) = \fifu \mu$. In the particular case with $\mu =
\sfD(d)$ and  $\fifu= rd_\ttO$ this leads to
\begin{equation}
\label{pisu v Berline}
rd_\ttO\sfD(rd_\ttO) = rd_\ttO \sfD(d_\ttO).
\end{equation}
We notice that $rd_\ttO = d_\ttO \sfD(r)$ since $d$ is a natural
transformation, so the left hand side of~(\ref{pisu v Berline}) equals $d_\ttO\sfD(r)\sfD(rd_\ttO)$
which in turn equals $d_\ttO\sfD(r^2)\sfD(rd_\ttO)$ by the functoriality of
$\sfD$.  Using the naturality of $d$ again, we eventually rewrite the left hand
side to $r^2 d_\ttO D(d_\ttO)$ and 
conclude that~(\ref{pisu v Berline}) is equivalent to
\[
r^2 d_\ttO D(d_\ttO) = rd_\ttO \sfD(d_\ttO).
\]
Since the functor $d_\ttO \sfD(d_A) : \sfD^2(\ttO) \to \ttO$ is 
surjective on both the
objects and the morphisms of $\ttO$, 
the latter equation is satisfied if and only if
$r^2 = r$.
\end{proof}

\begin{example}
\label{Dnes bychom meli jet na vylet do Postupimi.}
Each category $\ttO$ admits the {\em tautological\/} unary operadic
category structure given by choosing, in Proposition~\ref{Na vikend se
  pocasi zhorsi.},  $r$ to be the identity endofunctor
$\id_\ttO : \ttO \to \ttO$. The result is the rule $S  \ \fib
 S  \xrightarrow h T$, for each morphism of $\ttO$. 
\end{example}

\begin{example}
\label{Za chvili na vlak.}
In the presence of a collection~\eqref{Jsem posledni den v Cuernavace.}
we have an important idempotent functor featured
in Proposition~\ref{Na vikend se pocasi zhorsi.} defined by
$r(S) := U_c$,
where~$U_c$ is the chosen object in the connected component $\pi_0(S)$ of
$S \in \ttO$.
The rule $U_c\ \fib S \stackrel h\to T$ thus makes $\ttO$  
a~unary operadic category.
\end{example}

\begin{proposition}
\label{Zitra jedeme na vylet do Postupimi.}
The category $\ttO$ with the operadic structure given by an idempotent
$r:  \ttO \to \ttO$ is left  unital, cf.~ Definition~\ref{Opet cele
  prepisuji.}, 
if and only $r :\ttO \to \ttO$ is the
identity endofunctor. It is right unital if and only if 
the image of $r :\ttO \to \ttO$ consists of the chosen local
terminals  in~(\ref{Jsem posledni den v Cuernavace.}).
Therefore $\ttO$ is unital if and only if it is discrete,
i.e.\ a small set.
\end{proposition}

We leave the proof of Proposition~\ref{Zitra jedeme na vylet do
  Postupimi.} as an exercise.
Thus the only left unital unary operadic categories produced by
Proposition~\ref{Na vikend se pocasi zhorsi.} are the tautological
ones in Example~\ref{Dnes bychom meli jet na vylet do Postupimi.}, 
and the only right unital are those in
Example~\ref{Za chvili na vlak.}.

\begin{example}
\label{Dominik mel skvelou myslenku.}
Each unary operadic category with the structure given by an
idempotent $r$ in Proposition~\ref{Na vikend se pocasi zhorsi.} is
perfect, with $\ttO_F$ the image of $r$. Indeed, the maps between the
fibers are of the form $r(h)$ for a morphisms $h$ of $\ttO$. Since
\[
r(g)r(h) = r^2(g)r^2(h) = r\big(r(g)r(h)\big),
\]
the subclass of morphisms of that form is closed under composition.
Another example of a~perfect unary operadic category is the tautological operadic
category $\Tau(\ttA)$ generated by $\ttA$,
cf.~\cite[Definition~9]{blob}, 
with $\Tau(\ttA)_F$ the d\'ecalage $\sfD(\ttA)$ of the
category~$\ttA$.
 \end{example}

\subsection{Unary dioperadic categories -- definition and examples}
\label{Dnes na vylete v te umelecke vesnicce.}
We will consider triples $\ttQ,\ttQ_L,\ttQ_R$ consisting of a unary
operadic category $\ttQ_L$ with the underlying category $\ttQ$,
together with a unary operadic category~$\ttQ_R$ with the underlying
category the  opposite category $\ttQ^{\dagger}$. We will express this situation
by writing  $\ttQ = (\ttQ_L, \ttQ_R)$.
A morphism $h : S \to T$  of $\ttQ$ 
has thus its {\/\em fiber\/}, say $F$,  in $\ttQ_L$ and its {\/\em
  cofiber\/}, say $D$, which is the fiber of the map $h^\dagger : T^\dagger
\to S^\dagger$ in $\ttQ_R$. We will abbreviate
this~by
\begin{equation}
\label{Pocasi nic moc ale naletal jsem 3 hodiny.}
F\ \fib S \stackrel h \longrightarrow T \cof \ D.
\end{equation}

\begin{definition}
\label{Musim se objednat na endokrinologii za mesic.}
We say that $\ttQ = (\ttQ_L,\ttQ_R)$ as above is {\/\em left dioperadic\/}
if $\ttQ_L$ is perfect. In this case we
denote by $\ttQ_F$ the fiber subcategory of $\ttQ_L$. 
Dually, $\ttQ$ is {\/\em right dioperadic\/} if $\ttQ_R$ is perfect,
with~$\ttQ_C$ denoting the fiber subcategory of $\ttQ_R$.
Finally, $\ttQ$ is {\/\em dioperadic\/} if it is both 
left and right dioperadic.
\end{definition}

\noindent 
{\bf Convention.}
The rest of this and the following section
will concentrate on the unary case, thus all
dioperadic categories and their variants will be unary without
mentioning it.

\begin{example}
\label{Snehulacek se tesi na snih.}
In~Example~1.22 of~\cite{duodel} we discussed a unary
operadic structure on the subcategory  ${\tt Epi(A)}$ of epimorphisms
of an Abelian category $\ttA$, with the kernels of maps as the fibers.  
Operads over ${\tt Epi(A)}$ were the
 (absolute) Kapranov's charades~\cite[Definition~3.2]{kapranov:langlands}.
In the present example we consider the whole $\ttA$, with the fiber-cofiber 
structure  $\ttA = (\ttA_L,\ttA_R)$ given by the scheme
\[
\Ker(a) \ \fib X \stackrel a\longrightarrow Y \cof \ \Cok(a) 
\]
where ${\rm Ker(a)}$ resp.~$\Cok(a)$ is a chosen kernel resp.~cokernel
of a morphism $a$ of $\ttA$. Let us verify that the above definition of fibers 
is functorial. To this end, consider the diagram
\[
\xymatrix@C=2.5em@R=3em{
F'' \komp \   \ar@/_.8em/@{-->}[d]_\alpha 
\ar@{-->}@{^{(}->}[r]^(.6){\kappa''}& X \komp
\ar@{^{(}->}[d]_x\ar@{-->}[rr]^{h_Z}  & &Y   \komp
\ar@{^{(}->}[d]_y
\\
F'  \ar@{-->}[ur]^(.6)s \ar@/_.8em/@{-->}[u]^\beta \ 
\ar@{^{(}->}[r]^(.6){\kappa'}& S
\ar@/_1.5em/[dr]^{gh} \ar[rr]^{h} && T\ar@/^1.5em/[dl]_g
\\
&&Z&
}
\]
in which $X \stackrel x \hookrightarrow S$ resp.~$Y \stackrel y
\hookrightarrow T$ is the kernel of $gh$ resp.~the kernel of $g$. 
Since $g(hx) = (gh)x = 0$ and since $y$ is an equalizer of $g$ and the
null morphism, i.e.\ $y = \eq(g,0)$, there exists a unique
$h_Z : X \to Y$ making the upper rectangle of the diagram commutative.
The functoriality follows from
the uniqueness of this construction.

Let us compare the domain $F'$ of the kernel $\kappa'$ of $h$ to 
the domain $F''$ of the kernel $\kappa''$ of~$h_Z$.
Since $h(x\kappa'') = y(h_Z \kappa'') = 0$ and since $\kappa' = \eq(h,0)$,
there exist a unique
$\alpha : F'' \to F'$ making the left square commutative. 
Since $gh(\kappa') = g(h\kappa') = 0$ and since $x = \eq(gh,0)$, there
exists a~diagonal map $s : F' \to X$ such that the lower right
triangle in the left square commutes. Finally, $y h_Z s =  (hx)s = \kappa' h =
0$ and, thus, $h_Z s = 0$ since $y$, as a kernel, is a
monomorphism. So there exists $\beta : F ' \to F''$ making the upper
left triangle commutative. 
Let us prove that $\alpha$ and $\beta$ are mutual inverses. 
By the standard diagram chase
\[
\kappa' \alpha\beta = x \kappa'' \beta = sx = \kappa' \
\hbox { and } 
\ x\kappa'' \beta\alpha = xs\alpha = \kappa' \alpha = x\kappa''.
\]
Since both $\kappa'$ and $x\kappa''$ are monomorphisms, 
$\alpha\beta = \id_{F'}$ and $\beta\alpha =  \id_{F''}$ as claimed.
The cofiber side can be attended to by dualizing the above steps.

Notice that we did not actually prove that $F' = F''$ as required by
Axiom~(iv) of an operadic category, we only constructed a natural
and canonical isomorphism $F' \cong F''$. This indicates that this example must be
taken {\em cum grano salis\/}. However, in some concrete
cases when $\ttA$ is concrete (pardon the pun) and the kernels are
actual set-theoretic preimages of $0$, we indeed have $F'
= F''$ as~required. 
  
If $\ttA$ is the category of
finite dimensional vector spaces, which is the only case relevant to
our applications, we can take, as the cokernel of a map $h: X \to Y$, 
the space $\big(\Ker(h^*)\big)^*$, i.e.\ the linear dual of the
set-theoretic kernel of the dual map $h^*: Y^* \to X^*$. With this choice, the
canonical isomorphisms became equalities also on the cofiber side.

It is easy to see that $\ttA_L$ is unital, with
the collection of units~\eqref{Jsem posledni den v Cuernavace.} formed by
the null object $0$, so $\ttA_F = \ttA_L$,
cf.~Example~\ref{Dominik mel skvelou myslenku.}. Dually, $\ttA_R =
\ttA_C$, so both $\ttA_L$ and $\ttA_R$ are perfect, and $\ttA =
(\ttA_L,\ttA_R)$ is dioperadic. 
The fiber-cofiber structure of $\ttA$ has however a very specific,
subtle additional property which will be the topics of
Subsection~\ref{Jsem pred-predposledni den v Haife.}. We
postpone further discussion of $\ttA$ to that subsection. 
\end{example}

\begin{example}
\label{Vecer odjizdim na oslavu Silvestra, zitra mam letani.}
Let $\Setp$ be the  category of small pointed sets. Define  
the fiber of a map $S \Xarrow h T$ as the preimage of the base point
of $T$, and its cofiber as the complement 
of $\Im(h)$ in $T$ with the formally  attached
base point $\circ$, that is
\[
F :=   \big\{s \in S \ | \ h(s) = \circ\big\} 
\ \fib S \ \Xarrow h \ T \cof \ (T \setminus \Im(h)) \sqcup\{ \circ\} =: D.
\]
The induced map  $h_R : F  \to G$ between the fibers 
in the diagram of pointed sets
\[
\xymatrix@R=-.3em@C=.8em{
F \ar@{-->}[rr]^(.55){h_R}   
&&G \hskip -.5em
\\
\hskip -2em
\hskip 1.9em\raisebox{.7em}{{\rotatebox{270}{$\fib$}}} && 
\hskip -.1em\raisebox{.7em}{{\rotatebox{270}{$\fib$}}} \hskip -.5em
\\
S \ar[rr]^{h}  \ar@/_.9em/[ddddddddr]_(.4){gh}   &&
T \ar@/^.9em/[ddddddddl]^(.4){g} \hskip -.5em
\\&& 
\\&& 
\\&& 
\\&&
\\ &&
\\
\\
\\
&R &
} 
\] 
is the restriction of $h$ to  $F = \big\{s \in S \ | \ gh(s) = \circ
\big\}$.
The induced map
\[
h^R : 
H =  (T \setminus \Im(g)) \sqcup \{\circ\} \ \longrightarrow
D = (S \setminus \Im(hg)) \sqcup \{\circ\} 
\] 
between the cofibers in the diagram
\[
\xymatrix@R=-.3em@C=.8em{D
  \ar@{<--}[rr]^(.55){h^R}   
&&H \hskip -.5em
\\
\hskip -2em
\hskip 1.9em\raisebox{.7em}{{\rotatebox{270}{$\fib$}}} && 
\hskip -.1em\raisebox{.7em}{{\rotatebox{270}{$\fib$}}} \hskip -.5em
\\
T \ar@{<-}[rr]^{h}  \ar@{<-}@/_.9em/[ddddddddr]_(.4){hg}   &&
S \ar@{<-}@/^.9em/[ddddddddl]^(.4){g} \hskip -.5em
\\&& 
\\&& 
\\&& 
\\&&
\\ &&
\\
\\
\\
&R &
} 
\] 
is given by
\[
h^R(s) := 
\begin{cases}
h(s),\ \hbox { if } h(s) \not\in \Im(hg), \hbox { and}
\\
\circ,\    \hbox { if } h(s) \in \Im(hg).
\end{cases}
\]
The dioperadicity of $\Setp = (\Setp_L,\Setp_R)$ follows, as in
Example~\ref{Snehulacek se tesi na snih.}, from the
(co)unitality of $\Setp_L$ resp.~$\Setp_R$.
\end{example}

\begin{example}
\label{Jarka nemocna jako kazdou sudou sobotu.}
Let $\ttC$ be a category and $\Arr$ the category of arrows in
$\ttC$. Objects of $\Arr$ are morphisms of $\ttC$  and
morphisms  of $\Arr$  are commutative squares in $\ttC$. 
We can make $\Arr$ a~dioperadic category by postulating that the fiber
resp.~the cofiber of a morphism $S \to T$ in $\Arr$ given by the pair
$(F,D)$ of morphisms of $\ttC$
as in:
\begin{equation}
\label{tahadlo}
\xymatrix@R=.7em@C=.2em{
a \ar[rrrr]^F\ar[dd]_S&&&&c\ar[dd]^T
\\
&&&&
\\
b \ar[rrrr]^D&&&&d
}
\end{equation}
is $F$ resp.~$D$ interpreted as objects of $\Arr$, that is
\[
\xymatrix@R=.2em@C=.2em{
a\ar[dd]_F&&a \ar[rrrr]^F\ar[dd]_S&&&&c\ar[dd]^T&&b\ar[dd]^{\ D~.}
\\
& \fib &&&&&& \cof &
\\
c&&b \ar[rrrr]^D&&&&d&&d
}
\]
A simple calculation reveals that the fiber subcategory $\Arr_F$ 
consists of morphisms~\eqref{tahadlo} with $D$ the identity map.  
In other words, $\Arr_F$ is the d\'ecalage $\sfD(\ttC) =
\coprod_{c \in \ttC}
\ttC/c$, with the fiber diagram
\[
\xymatrix@R=.2em@C=.2em{
a\ar[dd]_F&&a \ar[rrrr]^F\ar[ddrr]_S&&&&\ b~.\ar[ddll]^T&%&&c\ar[dd]^{\ \id_c~.}
\\
& \fib &&&&%&& \cof &
\\
b&&&&c&&%&&c
}
\]
Similarly,
$\Arr_C = \coprod_{c \in \ttC} c/\ttC$. We see that here both 
$\ttQ_F \subset \ttQ_L$ and $\ttQ_C \subset \ttQ_R$ are proper subcategories.  
\end{example}

\begin{example}
Take in the previous example as $\ttC$ the chaotic category
$\Cha(\frakC)$ on a set $\frakC$ of `colors.' The associated category of arrows
will be the chaotic
category generated by the symbols $\bez ab$ with $a,b \in \frakC$, 
and the fiber-cofiber diagrams 
\[
\bez a{\rule {0pt}{.8em}c} \fib \bez ab \longrightarrow \bez cd \cof \bez bd.
\]
It is the unary version of the category of $\frakC$-bibouquets in 
Example~\ref{Zalije mi Jarka jeste jednou kyticky?}.
\end{example}

\subsection{Operads, cooperads and bimodules}
\label{Nemam se cim chlubit.}
The main definition of this subsection involves a non-unital
$\ttQ_F$-operad $\oP = \{\oP(T)\}_{T \in \ttQ_F}$ and  a non-counital
$\ttQ_C$-cooperad $\cooP =
\{\cooP(T)\}_{T \in \ttQ_C}$ with the (de)compositions
\begin{equation}
\label{Jarka zase blbne.}
\gamma_h:   \oP(T) \ot \oP(F) \longrightarrow \oP(S)
\ \hbox { and } \ \delta_h : \cooP(T) \longrightarrow \cooP(D) \ot \cooP(S)
\end{equation}
given for any $h : S \to T$ in~(\ref{Pocasi nic moc ale naletal
  jsem 3 hodiny.}) belonging to $\ttQ_F$ in the operad case, resp.~to
$\ttQ_C$ in the cooperad case.
An explicit definition of non-unital operads over unary
operadic categories can be found in~\cite[Section~1]{blob}, axioms of
cooperads are the formal duals of the operad
axioms. While operads and cooperads use only a part of the
structure of $\ttQ$, namely the subcategories $\ttQ_F$ and $\ttQ_C$, 
the structure introduced
below uses the entire dioperad structure.

\begin{definition}
\label{Musim se objednat na ocni.}
Let $\ttQ = (\ttQ_L, \ttQ_R)$ be a dioperadic category, $\oL$
a $\ttQ_F$-operad and $\cooP$ a $\ttQ_C$-cooperad.
A {\/\em {\CP}-bimodule\/} is a collection $\oM = \{\oM(S)\}_{S
  \in \ttQ}$ of objects of $\ttV$ with an action
\begin{subequations}
\begin{equation}
\label{Dva dny za sebou!}
\omega_h : \oM(T) \ot \oL(F) \longrightarrow \cooP(D)  \ot \oM(S)
\end{equation}
defined for each $h$ as in~(\ref{Pocasi nic moc ale naletal jsem 3
  hodiny.}). We moreover require a compatibility between this action
and the (co)operad structures of $\oL$ and $\cooP$. Namely we require, for
each diagram
\begin{equation}
\label{Pujdeme spolu na demonstraci?}
\psscalebox{1.0 1.0} % Change this value to rescale the drawing.
{
\begin{pspicture}(0,1)(5.943182,1.8)
\rput(2.142857,1.2418182){$F$}
\rput(2.6,1.2418182){$\fib$}
\rput(2.9714286,1.2418182){$X$}
\rput(5.0363636,1.2418182){$Y$}
\rput(2.9636364,0.74){$\raisebox{.7em}{{\rotatebox{270}{$\fib$}}}$}
\rput(5,0.74){$\raisebox{.7em}{{\rotatebox{270}{$\fib$}}}$}
\rput(2.1818182,0.6881818){$\raisebox{.3em}{{\rotatebox{90}{$=$}}}$}%{$=$}
\rput(2.142857,0.28181817){$F$}
\rput(2.6,0.28181817){$\fib$}
\rput(2.9714286,0.28181817){$S$}
\rput(5.0363636,0.28181817){$T$}
\psline[linecolor=black, linewidth=0.02, arrowsize=0.05291667cm 2.0,arrowlength=1.4,arrowinset=0.0]{->}(3.1818182,1.2418182)(4.818182,1.2818182)
\psline[linecolor=black, linewidth=0.02, arrowsize=0.05291667cm 2.0,arrowlength=1.4,arrowinset=0.0]{->}(3.1818182,0.24181817)(4.818182,0.24181817)
\rput(5.3636365,0.28181817){$\cof$}
\rput(5.818182,0.28181817){$D$}
\rput(4.0,-2.030909){$Z$}
\pscustom[linecolor=black, linewidth=.04]
{
\newpath
\moveto(7.090909,2.06)
}
\pscustom[linecolor=black, linewidth=.04]
{
\newpath
\moveto(7.818182,1.2418182)
}
\pscustom[linecolor=black, linewidth=0.02]
{
\newpath
\moveto(7.4545455,1.06)
}
\pscustom[linecolor=black, linewidth=0.02]
{
\newpath
\moveto(7.090909,0.6963636)
}
\pscustom[linecolor=black, linewidth=0.02]
{
\newpath
\moveto(7.272727,2.5145454)
}

\pscustom[linecolor=black, linewidth=0.02]
{
\newpath
\moveto(5.0,0.059999943)
\lineto(4.772727,-0.6672727)
\curveto(4.6590905,-1.0309088)(4.4772725,-1.4854544)(4.2727275,-1.7581818)
}
\psline[linecolor=black, linewidth=0.02, arrowsize=0.05291667cm 2.0,arrowlength=1.4,arrowinset=0.0]{->}(4.2727275,-1.7581818)(4.181818,-1.8490909)
\pscustom[linecolor=black, linewidth=0.02]
{
\newpath
\moveto(2.9818182,0.059999943)
\lineto(3.209091,-0.6672727)
\curveto(3.322727,-1.0309088)(3.5045452,-1.4854544)(3.709091,-1.7581818)
}
\psline[linecolor=black, linewidth=0.02, arrowsize=0.05291667cm 2.0,arrowlength=1.4,arrowinset=0.0]{->}(3.709091,-1.7581818)(3.8,-1.8490909)
\rput(3.3818183,-2.5581818){$H$}
\rput(4.6181817,-2.5581818){$R$}
\rput(3.7363637,-2.4036363){$\raisebox{.7em}{{\rotatebox{38}{$\fib$}}}$}
\rput(4.3,-2.94){$\raisebox{.7em}{{\rotatebox{38}{$\fib$}}}$}
\rput(4.2636365,-2.3036363){$\raisebox{.7em}{{\rotatebox{150}{$\fib$}}}$}
\pscustom[linecolor=black, linewidth=0.02]
{
\newpath
\moveto(5.8,0.059999943)
\lineto(5.517172,-0.89584416)
\curveto(5.375757,-1.3737662)(5.1494946,-1.9711689)(4.8949494,-2.3296103)
}
\psline[linecolor=black, linewidth=0.02, arrowsize=0.05291667cm 2.0,arrowlength=1.4,arrowinset=0.0]{->}(4.8949494,-2.3296103)(4.7818184,-2.449091)
\rput(4.0,-3.121818){$H$}
\rput(3.7454545,-2.9){$\raisebox{.7em}{{\rotatebox{135}{$=$}}}$}%{$=$}
\rput[b](3.909091,1.4236363){\scriptsize $h_Z$}
\rput[b](4.0,0.42363638){\scriptsize $h$}
\rput[l](4.818182,-1.1218182){\scriptsize $g$}
\rput[bl](5.5454545,-1.3945454){\scriptsize $g^S$}
\rput(2.9181818,-1.1218182){\scriptsize $gh$}
\end{pspicture}
}
\raisebox{-13em}{\rule{0pt}{0pt}}
\end{equation}
in which $h_Z$ resp.~$g^S$ are the induced maps between the (co)fibers,
and where the equalities $F=F$ resp.~$H=H$ follow from the axioms of
the operadic categories $\ttQ_L$
resp.~$\ttQ_R$, the commutativity~of
\begin{equation}
\label{Uz zase jezdim do Kolina na kole.}
\xymatrix@C=5em{\oM(Z) \ot \oP(Y) \ot \oL(F)   \ar[dd]_{\omega_g \ot
    \id} 
\ar[r]^{\id \ot \gamma_{h_Z}}
& \oM(Z) \ot \oP(X) \ar[d]^{\omega_{gh}}
\\
&\cooP(R) \ot \oM(S) \ar[d]^{\delta_{g^S} \ot \id}
\\
\cooP(H) \ot \oM(T) \ot \oP(F)  \ar[r]^{\id \ot \omega_h}
&\cooP(H) \ot \cooP(D) \ot \oM(S).
}
\end{equation}
\end{subequations}
\end{definition}
Notice that all objects and maps in~(\ref{Uz zase jezdim do Kolina na
  kole.}) are defined, since $F\ \fib X \Xarrow {h_Z} Y$ is a scheme
in $\ttQ_F$ and $D \Xarrow {g^s} R \cof \ H$ a scheme in $\ttQ_C$.

\begin{remark}
\label{Poletim do Izraele?}
\begin{subequations}
Assume that the base monoidal category $\ttV$ is the category of
vector spaces. 
The linear dual 
of the $\ttQ_C$-cooperad $\cooP$ in Definition~\ref{Musim
  se objednat na ocni.} 
is a $\ttQ_C$-operad $\oR$
with the structure operations $\rho_h :\oR(S)\ot \oR(D) \to \oR(T)$
for $h:S \to T$ as in~\eqref{Pocasi nic moc ale naletal jsem 3 hodiny.}.
The action~\eqref{Dva dny za sebou!} induces an action
\begin{equation}
\label{Koupil jsem si dve zpetna zrcatka.}
\varpi_h : \oR(D) \ot \oM(T)\ot \oP(F) \longrightarrow \oM(S)
\end{equation}  
and the commutativity of~(\ref{Uz zase jezdim do Kolina na kole.})
implies the commutativity of
\begin{equation}
\label{Pred 4 lety jsme byli v Parizi.}
\xymatrix@C=5em@R=3.5em{\oR(D) \ot \oR(H) \ot \oM(Z) \ot \oP(Y) \ot \oP(F)
\ar[r]^(.6){\rho_{\! g^S} \ot \id \ot \gamma_{h_Z}} \ar[d]_{\id\ot \varpi_g
\ot \id}
&
\oR(R) \ot \oM(Z) \ot \oP(X) \ar[d]^{\varpi_{gh}} 
\\
\oR(D) \ot  \oM(T) \ot \oP(F)
\ar[r]^{\varpi_h} & \ \oM(S).
}
\end{equation}
For
$\phi \in \oL(F)$, $\psi \in \oL(Y)$, $\zeta \in \oM(Z)$, $\eta
\in \oR(H)$ and  $\delta
\in \oR(D)$ the commutativity of~(\ref{Pred 4 lety jsme byli v
  Parizi.}) means
\begin{equation}
\label{Dnes je koncert ale nechce se mi do deste.}
\varpi_{gh}\left(\rho_{g^S}(\delta,\eta),\zeta,\gamma_{h_Z}(\psi,\phi)\right)
= \varpi_h\left(\delta,\varpi_g(\eta,\zeta,\psi),\phi\right).
\end{equation}
\end{subequations}
\end{remark}

One may wonder why we did not define bimodules using  more
conventional action~\eqref{Koupil jsem si dve zpetna zrcatka.} that
avoids the use of cooperads.
One of the reasons was that the linear dual of an 
operad need not be a cooperad, and that the 
action~\eqref{Koupil jsem si dve zpetna zrcatka.} need not induce an
action~\eqref{Dva dny za sebou!} unless $\oR$ satisfies appropriate
finitarity assumptions. Our approach is therefore more general.
The main reason for out choice was however 
the manifest self-duality of Definition~\ref{Musim se objednat na
  ocni.}.

\begin{example}
\label{V Cechach je plno snehu a ja trcim v Izraeli.}
Let $\ttV$ be, as in Remark~\ref{Poletim do Izraele?},  
the category of vector spaces.
For the terminal one-object, one-morphism dioperadic category $\odot$,
Definition~\ref{Musim se objednat na ocni.} leads to an associative
algebra $R$, a~coassociative coalgebra $C$, a vector
space $M$, and a linear map
\begin{equation}
\label{Pracuji na prubezne zprave.}
\omega : M \ot R \longrightarrow  C \ot M. 
\end{equation}
The compatibility of $\omega$ 
with the associative multiplication $\mu : R \ot R  \to R$ and the coassociative
comultiplication $\delta : C \to C \ot C$ means the
commutativity of the  diagram
\[
\xymatrix@C=5em{M \ot R \ot R   \ar[dd]_{\omega \ot
    \id} 
\ar[r]^{\id \ot \mu}
& M \ot R \ar[d]^{\omega}
\\
&C \ot M \ar[d]^{\delta \ot \id}
\\
C \ot M \ot R  \ar[r]^{\id \ot \omega}
&C \ot C \ot \ M.
}
\]

In the disguise of Remark~\ref{Poletim do Izraele?}
this structure appears as
an $L$-$R$-bimodule $M$ for associative algebras $L$ and
$R$, with $L$ the linear dual of $C$. 
Equation~\eqref{Dnes je koncert ale nechce se mi do deste.} 
in this setup yields
\[
\big[(a'a''),m,(b''b')\big] = \big[a', [a'' \hskip -.3 em ,m,b''],b'\big], \ a',a''
\in L, \ \  b'' \hskip -.2 em   ,b' \in R, \ m \in M,
\]
where $[-,-,-] : L \ot M \ot R \to M$ is the structure operation of the
bimodule $M$. This explains the terminology used in
Definition~\ref{Musim se objednat na ocni.}. 
\end{example}

\section{Unary bioperadic categories}
\label{Posloucham ricercar.}

\lettrine{\color{red} D} {ioperadic} categories in 
Definition~\ref{Musim se objednat na endokrinologii.} 
are structures satisfying the smallest set of conditions that
guarantee the 
existence of operads, cooperads and bimodules. In this
section we analyze when also the notions of the associated algebras,
coalgebras and traces 
make sense. To do this, we resume our quest for a `bivariant' 
definition of algebras by
analyzing the unary version of~\cite[Definition~1.20]{duodel}.

\subsection{Operad algebras revisited}
Let $\ttO$ be a unary operadic category and $\oP$ an $\ttO$-operad
with structure operations as in~\eqref{Jarka zase blbne.}. A
$\oP$-algebra is `classically'  a collection $A = \{A_c \}_{c \in
\pi_0(\ttO)}$ of objects of~$\ttV$ indexed by the set $\pi_0(\ttO)$
of connected components of $\ttO$, equipped with structure operations
that are the unary versions of~\eqref{Dnes zverejneny propozice na
    Safari.}, i.e.\  
\begin{equation}
\label{Za tri dny na Safari.}
a_T : \oP(T) \ot A_{\sou T} \to   A_{\tar T},\ T \in \ttO,
\end{equation}
where $\sou T$ is the
connected component of the unique fiber $U_T$ of the identity $\id_T: T
\to T$, and  $\tar T  := \pi_0(T) \in  \pi_0(\ttO)$ is the connected
component of $T$. The associativity of the actions~(\ref{Za tri dny
  na Safari.}) requires that the diagram
\begin{equation}
\label{Zitra odjizdim na Safari.}
\xymatrix{
\oP(T) \ot \oP(F) \ot A_{\sou F}
\ar[r]^(.55){\id \ot a_F} \ar[d]_{\gamma_h \ot \id}
 & \oP(T) \ot A_{\tar F} \ar@{=}[r]^{\hbox{\textcircled{\scriptsize 3}}}
& \oP(T) \ot A_{\sou T}\ar[d]^(.6){a_T}
\\
\ar@{=}[d]_{\hbox{\textcircled{\scriptsize 1}}}
\oP(S) \ot A_{\sou F} &
\boxed{F  \fib S \stackrel h \longrightarrow T}&  A_{\tar T} \ar@{=}[d]^{\hbox{\textcircled{\scriptsize 2}}}
\\
\oP(S) \ot A_{\sou S} \ar[rr]^{a_S}
 && A_{\tar S} 
}
\end{equation}
commutes for each $F \, \fib \, S \stackrel h\longrightarrow T$.  

Let us explain the 
equalities in~(\ref{Zitra odjizdim na Safari.}). 
Equality \textcircled{\scriptsize
  1} follows from the equality $U_S = U_F$ in the diagram
\[
\xymatrix@R=-.3em@C=.8em{\hskip -1.8em U_F\ \fib F
  \ar[rr]^{\id_F}   
&&F \hskip -.5em
\\
\hskip -1.4em\raisebox{.7em}{\rotatebox{270}{$=$}}\hskip 1.8em\raisebox{.7em}{{\rotatebox{270}{$\fib$}}} && 
\hskip -.2em\raisebox{.7em}{{\rotatebox{270}{$\fib$}}} \hskip -.5em
\\
\hskip -1.8em U_S\ \fib S \ar[rr]^{\id_S}  \ar@/_.9em/[ddddddr]_(.4){h}   &&
S \ar@/^.9em/[ddddddl]^(.4){h} \hskip -.5em
\\&& 
\\&& 
\\&& 
\\&&
\\ &&
\\
& T &
}
\]
which is the particular case of~\cite[equation~(4)]{blob}.
Equality \textcircled{\scriptsize  2} follows from the mere existence
of the map  $h :S \to T $ that implies that $S$ and $T$ belong to the
same component of $\ttO$.
Similarly, \textcircled{\scriptsize 3} follows from the existence of the 
induced map $h_T : F \to U_T$
from the fiber of $h :S \to T $ to the fiber of $\id_T : T \to
T$ in the diagram
\[
\xymatrix@R=-.3em@C=.8em{ F
  \ar[rr]^{h_T}   
&&U_T \hskip -.5em
\\
\hskip -1.7em
\hskip 1.65em\raisebox{.7em}{{\rotatebox{270}{$\fib$}}} && 
\hskip -.5em\raisebox{.7em}{{\rotatebox{270}{$\fib$}}} \hskip -.9em
\\
S \ar[rr]^{h}  \ar@/_.9em/[ddddddr]_(.4){h}   &&
T \ar@/^.9em/[ddddddl]^(.4){\id_T} \hskip -.5em
\\&& 
\\&& 
\\&& 
\\&&
\\ &&
\\
&\ T. &
}
\]

\subsection{Unary bioperadic categories -- definition and examples}
Let us try to formulate a bivariant definition of operad algebras, i.e.\
the one that uses the sources and the targets satisfying 
\begin{equation}
\label{xy}
\sou X = \tar{X^\dagger} \ \hbox { and } \ \tar X = \sou{X^\dagger}, \
\hbox{ for any } \ X \in \ttQ. 
\end{equation} 
If $\ttQ = (\ttQ_L,\ttQ_R)$ is a (left or/and right) dioperadic category, 
the obvious choice is to
define the {\em source} \ $\sou X$ 
resp.~the {\em target} $\tar X$ of $X$ as the connected component of
the fiber $U_X$, resp.~cofiber $C_X$, of the identity $\id_X : X \to X$,
in shorthand 
\begin{equation}
\label{Pojedeme do Telce?}
U_X \ \fib  X \stackrel{\id_X}\longrightarrow X \cof \ C_X \ 
\Longrightarrow \
\sou X := \pi_0(U_X),\ \tar X := \pi_0(C_X).
\end{equation}
Such a choice obviously fulfills the self-duality property~(\ref{xy}).  

Let $\ttQ$ be a dioperadic category. To each morphism $h :S \to T$ in~(\ref{Pocasi nic moc ale naletal jsem
  3 hodiny.}) we 
associate  its {\/\em 
  analysis\/}, which is the scheme
\begin{equation}
\label{Dnes jsme kupovali Jarce vonavcicku.}
\xymatrix@C=-.4em@R=-.2em{
U_F \ar@{=}[dd] &\fib &F  \ar[rrrrrrrrrrrr]^{\id_F} &&&&&&&&&&&& F  &\cof& C_F
\\
&&\raisebox{.7em}{{\rotatebox{270}{$\fib$}}}&&&&&&&&&&&&
\raisebox{.7em}{{\rotatebox{270}{$\fib$}}}
\\
U_S&\fib & \ar[dddddd]_h  S \ar[rrrrrrrrrrrr]^{\id_S}  
&&&&&&&&&&&& S \ar[dddddd]^h     &\cof& C_S
\\
\\
\\
 &&&&&&\hbox{\hskip 1em \boxed{\tt analysis\rule{0em}{.85em}}}&&&&&&
\\
\\
\\
U_T&\fib &   T \ar[rrrrrrrrrrrr]^{\id_T}   &&&&&&&&&&&& T&\cof&\ C_T
\\
&&\raisebox{.3em}{{\rotatebox{90}{\hskip -.2em$\fib$}}}&&&&&&&&&&&&
\raisebox{.3em}{{\rotatebox{90}{\hskip -.2em$\fib$}}}
\\
U_D&\fib&D \ar[rrrrrrrrrrrr]^{\id_D}  &&&&&&&&&&&& D &\cof&\ C_D \ar@{=}[uu]
}
\end{equation} 
in which the two equalities follow from the axioms of operadic
categories $\ttQ_L$ and $\ttQ_R$.

Consider again the associativity diagram~\eqref{Zitra odjizdim na
  Safari.} for $\ttO = \ttQ_F$, but now
with the sources and targets defined in~(\ref{Pojedeme do Telce?}).
Equality~\textcircled{\scriptsize 1} is implied by the equality $U_F =
U_S$ in the upper left corner of~(\ref{Dnes jsme kupovali Jarce vonavcicku.}).  
Equality  \textcircled{\scriptsize 2} however requires
$\pi_0 (C_S)  =
\pi_0(C_T)$ which need not hold in general.
Equality
\textcircled{\scriptsize 3}, i.e.~$A_{\tar F} = A_{\sou T}$, is
moreover of very different nature, since it refers both to the fiber and
cofiber structures of~$\ttQ$. 

Definition~\ref{Jarka ma ovoce.} below formulates conditions 
assuring that diagram~\eqref{Zitra odjizdim na Safari.} and the
similar diagrams for coalgebras and traces make sense. The notation refers to
the analysis~(\ref{Dnes jsme kupovali Jarce vonavcicku.}) of $h: S\to T$.

\begin{definition}
\label{Jarka ma ovoce.}
A {\/\em left bioperadic\/} category is a left dioperadic category $\ttQ$ 
such that,
for an arbitrary morphism $h: S \to T$ in the fiber subcategory $\ttQ_F$, 
\begin{subequations}
\begin{align}
\label{1}
C_F =& \ U_T, \ \hbox { and}
\\
\label{1bis}
C_S &= C_T.
\end{align}
\end{subequations}
Dually, a right dioperadic category $\ttQ$ is {\/\em right
  bioperadic\/} if, for an arbitrary $h: S \to T$  in the cofiber
subcategory $\ttQ_C$, 
\begin{subequations}
\begin{align}
\label{2}
U_D =\ & C_S, \ \hbox { and}
\\
\label{2bis}
U_T& = U_S.
\end{align}
\end{subequations}
Finally, $\ttQ$ is {\em bioperadic\/} if it is both left and right
dioperadic, and if ~(\ref{1}) and~(\ref{2}) are fulfilled for any morphism $h:
S \to T$ of $\ttQ$. 
\end{definition}

Observe that~\eqref{1},  with the sources and 
targets~\eqref{Pojedeme do Telce?}, implies $\tar F = \sou T$,   thus
$A_{\tar F} = A_{\sou T}$ as required 
in \textcircled{\scriptsize 3} of~\eqref{Zitra odjizdim na Safari.}. 
Similarly,~\eqref{1bis} implies  $\tar S = \tar T$.
Dually,~\eqref{2} implies  $\sou D = \tar S$ and
\eqref{2bis} implies $\sou T = \sou S$. 

\begin{remark}
\label{Hraji si se cteckou.}
Let  $F \ \fib S \stackrel h\to T \cof \ D$ be a morphism of a
bioperadic category $\ttQ$.
The equalities $\sou F = \sou S$ resp.~$\tar T = \tar D$ always hold 
by the axioms of the operadic categories $\ttQ_L$
resp.~$\ttQ_R$. Combining them with~\eqref{1bis}--\eqref{2} 
we conclude that, for $h$ in the intersection
$\ttQ_F \cap \ttQ_C$, 
\begin{equation}
\label{Posloucham Ceske a morevske barokni varhany II.}
\tar S = \tar T = \tar D = \sou D\ 
\hbox { and } 
\
\sou S = \sou T = \sou F = \tar F.
\end{equation}
\end{remark}

\begin{proposition} 
\label{Najdu odvahu se dnes v te strasne zime projet na kole?}
The inclusion of sets
\begin{subequations}
\begin{equation}
\label{Za chvili}
\big\{ U_T \ | \ T \in \ttQ_F\big\} \subseteq \big\{ C_T \ | \ T \in
\ttQ_F\big\}
\ \hbox { implying }\
\big\{ \sou T \ | \ T \in \ttQ_F\big\} \subseteq \big\{ \tar T \ | \ T \in \ttQ_F\big\}
\end{equation}
where, as before, $U_T$ resp.~$C_T$ is the fiber resp.~cofiber of the
identity $\id_T : T \to T$,
holds in any left bioperadic category.
Dually, the inclusion 
\begin{equation}
\label{sraz u Podlipneho.}
\big\{ U_T \ | \ T \in \ttQ_C\big\} \supseteq \big\{ C_T \ | \ T \in
\ttQ_C\big\}
\ \hbox { implying }\
\big\{ \sou T \ | \ T \in \ttQ_C\big\} 
\supseteq \big\{ \tar T \ | \ T \in \ttQ_C\big\}
\end{equation}
holds in any right bioperadic category.
If\/ $\ttQ$ is bioperadic, then
\begin{equation}
\label{pinkalove}
\big\{ \sou T \ | \ T \in \ttQ\big\} = \big\{ \tar T \ | \ T \in
\ttQ\big\}
\ \hbox { implying }\ 
\big\{ \sou T \ | \ T \in \ttQ\big\} 
= \big\{ \tar T \ | \ T \in \ttQ\big\}.
\end{equation}
Thus the sets of the sources and the targets of a bioperadic category 
are the same. 
\end{subequations}
\end{proposition}

\begin{proof}
Apply equality \eqref{1} to $h = \id_S$, i.e.\ to the situation 
$U_S \fib S \stackrel{\id_S}\to S \cof C_S$ to obtain
$U_S = C_{U_S}$ for $S \in \ttQ_F$, which implies~(\ref{Za chvili}). 
Similarly, \eqref{2} gives $U_{C_S} = C_S$ for $S \in \ttQ_C$, which 
implies~(\ref{sraz u Podlipneho.}). In a bioperadic category, $U_S =
C_{U_S}$ and  $U_{C_S} = C_S$ hold for any $S \in \ttQ$, which
gives~(\ref{pinkalove}). 
\end{proof}

\begin{example}
\label{Maly Meda}
Any  abelian category~$\ttA$ with the dioperadic structure  $\ttA =
(\ttA_L,\ttA_R)$ in Example~\ref{Snehulacek se tesi na snih.} is
bioperadic. Indeed, the kernels and cokernels of the identities equal
the null object, thus~\eqref{1bis}--\eqref{2} are trivially fulfilled. 
The similar argument holds also 
for the category $\Setp$ with the dioperadic structure in
Example~\ref{Vecer odjizdim na oslavu Silvestra, zitra mam letani.}.  
\end{example}

\begin{example}
Let us come back to the arrow category 
$\Arr$ from Example~\ref{Jarka nemocna jako kazdou sudou sobotu.}. 
For $F \ \fib S \xrightarrow h T\cof \ D$ given by the commutative 
square~\eqref{tahadlo} we compute
\[
C_F =
\raisebox{1.3em}{
$
\xymatrix@R=1.5em{
c \ar[d]^{\id_c}
\\
c
}
$},\
U_T =
\raisebox{1.3em}{
$
\xymatrix@R=1.5em{
c \ar[d]^{\id_c}
\\
c
}
$}
,\
U_D =
\raisebox{1.3em}{
$
\xymatrix@R=1.5em{
b \ar[d]^{\id_b}
\\
b
}
$}\ \hbox { and } \
C_S =
\raisebox{1.3em}{
$
\xymatrix@R=1.5em{
b \ar[d]^{\id_b}
\\
b
}
$},
\]
so both~\eqref{1} and~\eqref{2} is fulfilled by all morphisms. We
moreover have
\[
C_T =
\raisebox{1.3em}{
$
\xymatrix@R=1.5em{
d \ar[d]^{\id_d}
\\
d
}
$}
\ \hbox { and } \
U_S =
\raisebox{1.3em}{
$
\xymatrix@R=1.5em{
a \ar[d]^{\id_a}
\\
a
}
$}.
\]
Since $b=d$ if $h$ belongs to the fiber subcategory,~\eqref{1bis}
is fulfilled;~\eqref{2bis} is fulfilled for the similar reasons.
The category $\Arr$ is therefore bioperadic. 
\end{example}

\subsection{Algebras, coalgebras and traces}
Assume that $\ttQ$ is left bioperadic,
$F \in \ttQ_F$ an object of the fiber subcategory, and $U_F \ \fib
F \Xarrow {\id_F} F \cof \ C_F$. It is clear that $U_F \in \ttQ_F$,
but we 
claim that $C_F \in  \ttQ_F$ too. Indeed, $F$ is a fiber of a map, say
$F \ \fib S \Xarrow {h} T$, so $C_F = U_T$
by~\eqref{1}, while clearly $U_T \in \ttQ_F$. 

Dually, if $\ttQ$ is
right dioperadic and $D \in \ttQ_C$ an object of the cofiber
subcategory, 
then both $U_D$ and $C_D$ in the diagram $U_D \ \fib
D \Xarrow {\id_D} D \cof \ C_D$ belong to $\ttQ_C$.
The structure maps $a_F$ and $b_D$ in the following definition are
therefore well-defined.

\begin{definition}
\label{Jarka dnes jela do Motola.}
Let $\ttQ$ be a left bioperadic category and $\oL$
a $\ttQ_F$-operad. A {\em $\oL$-algebra} is a collection  $A = \{A_c\}_{c
  \in \pi_0(\ttQ_F)}$ of objects of $\ttV$ together with structure operations
\begin{subequations}
\begin{equation*}
a_F : \oL(F) \ot A_{\sou F} \longrightarrow   A_{\tar F},\ F \in \ttQ_F,
\end{equation*}
where the sources and targets are as
in~(\ref{Pojedeme do Telce?}), such that the diagram
\begin{equation}
\label{Dnes predsedam Rade}
\xymatrix{
\oP(T) \ot \oP(F) \ot A_{\sou F}
\ar[r]^(.58){\id \ot a_F} \ar[d]_{\gamma_h \ot \id}
 & \oP(T) \ot A_{\tar F} \ar@{=}[r]^{\eqref{1}}
& \oP(T) \ot A_{\sou T}\ar[d]^(.6){a_T}
\\
\ar@{=}[d]_{\rm axiom\ of\ \ttQ_L\ }
\oP(S) \ot A_{\sou F} &\boxed{F  \fib S \stackrel h \longrightarrow T}
&  A_{\tar T} \ar@{=}[d]^{{\eqref{1bis}}}
\\
\oP(S) \ot A_{\sou S} \ar[rr]^{a_S}
 && A_{\tar S} 
}
\end{equation}
commutes for each $F  \fib S \stackrel h \to T$ in $\ttO_F$.

Dually, suppose that  $\ttQ$ is right bioperadic
 and $\cooP$ a $\ttQ_C$-cooperad. A {\em $\cooP$-coalgebra}  is a collection $B = \{B_c\}_{c
  \in \pi_0(\ttQ_C)}$ together with structure operations
\begin{equation*}
b_D : \cooP(D) \ot B_{\sou D} \longrightarrow   B_{\tar D},\ D \in \ttQ_C,
\end{equation*}
such that the diagram
\begin{equation}
\label{s2}
\xymatrix{
\cooP(D) \ot \cooP(S) \ot B_{\sou S}
\ar[r]^(.58){\id \ot b_S} 
 & \cooP(D) \ot B_{\tar S} \ar@{=}[r]^{\eqref{2}}
& \cooP(D) \ot B_{\sou D}\ar[d]^(.6){b_D}
\\
\ar@{=}[d]_{(\ref{2bis})} \ar[u]_{\delta_h \ot \id}
\cooP(T) \ot B_{\sou S} &\boxed{S \stackrel h\longrightarrow T \cof D}
&  B_{\tar D} \ar@{=}[d]^{\ \rm axiom\ of\ \ttQ_R}
\\
\cooP(T) \ot B_{\sou T} \ar[rr]^{b_T}
 && B_{\tar T} 
}
\end{equation}
commutes for each $S \stackrel h\to T \cof D$ in $\ttQ_C$.
\end{subequations}
\end{definition}

\begin{definition}
\label{Podari se mi ten clanek vubec publikovat?}
Let  $\ttQ$ be a bioperadic category and 
$\oM$ a \CP-bimodule as in Definition~\ref{Musim se objednat na
  ocni.}. An {\/\em $\oM$-trace\/} consists of a $\oL$-algebra 
$A$  and an $\cooP$-coalgebra $B$ as above, together with structure operations
\begin{equation*}
%\label{s3}
c_T : \oM(T) \ot A_{\sou T} \longrightarrow   B_{\tar T},\ T \in \ttQ,
\end{equation*}
such that the diagram
\begin{equation}
\label{Jarce hrabe z chalupy.}
\xymatrix@C=2em{
\oM(T) \ot \oP(F) \ot A_{\sou F} \ar[r]^(.56){\id \ot a_F}
\ar[d]_{\omega_h \ot \id}
&
\oM(T) \ot  A_{\tar F} \ar@{=}[r]^{\eqref{1}}
&\oM(T) \ot  A_{\sou T} \ar[r]^(.63){c_T} & B_{\tar T}
\ar@{=}[dd]^{\ \rm axiom \ of \ \ttQ_R}
\\
\cooP(D) \ot \oM(S) \ot A_{\sou F} \ar@{=}[d]_{\ \rm axiom \ of \ \ttQ_L}
& \boxed{F \fib S \stackrel h\longrightarrow T \cof D}
\\
\cooP(D) \ot \oM(S) \ot A_{\sou S} \ar[r]^(.57){\id \ot c_S}
& \cooP(D) \ot B_{\tar S}  \ar@{=}[r]^{\eqref{2}}
& \cooP(D) \ot B_{\sou D}  \ar[r]^(.63){b_D}
& B_{\tar D}
}
\end{equation}
commutes for each $F \fib S \stackrel h\to T \cof D$ in $\ttQ$.
\end{definition}

Notice that diagrams~(\ref{Dnes predsedam Rade}),(\ref{s2})  
and~\eqref{Jarce hrabe z chalupy.} make sense
because we assumed~\eqref{1bis}--\eqref{2}.
The punchline of our approach thus reads:
\begin{center}
{\em Bioperadic categories are dioperadic categories for which
  algebras over operads, coalgebras \\ over cooperads, and
traces over bimodules could be defined.}  
\end{center}

\begin{example}
In the situation of Example~\ref{V Cechach je plno snehu a ja trcim v
  Izraeli.}, Definition~\ref{Podari se mi ten clanek vubec
  publikovat?} describes structures consisting of an 
associative algebra $R$ acting on a left module $A$, of a coassociative
coalgebra $C$ acting on a left module $B$, and of a vector space $M$  with  an
action $c: M \ot A \to B$. The commutativity of~(\ref{Jarce hrabe z
  chalupy.}) requires that
\begin{equation}
\label{Napadne snih jeste pred zacatkem Zimni skoly?}
c\big(m \ot (ra)\big) =  \sum  c_{(1)} c \big(m_{(2)}\ot a\big), \ m \in M,\ r \in R, \ a \in A,
\end{equation}
with $\sum c_{(1)} \ot m_{(2)} \in C \ot M$ denoting the image of $m \ot r  \in
M \ot R$ under 
the structure map~\eqref{Pracuji na prubezne zprave.}.

Equation~\eqref{Napadne snih jeste pred zacatkem Zimni skoly?} assumes
a particularly nice form when the coalgebra $C$ equals the ground
field~$\bbk$ with the comultiplication given by the canonical isomorphism 
$\bbk \Xarrow \cong \bbk \ot \bbk$ and when $B$ bears the 
trivial $\bbk$-action. Equation~(\ref{Napadne snih jeste pred zacatkem
  Zimni skoly?}) in this case says that the map
\[
\varpi :  M \to {\rm Hom}(A,B),
\]
adjoint to $c: M \ot A \to B$, is a morphism of right $R$-modules.
Here the $R$-module structure of $M$ is given by the 
action~\eqref{Pracuji na prubezne zprave.}
after the identification  \hbox{$\bbk\! \ot\! M \cong M$}, 
and the right $R$-action on
${\rm Hom}(A,B)$ is induced by the left $R$-action on $A$ in the usual
manner. Less simple-minded examples can be found in
Examples~\ref{Vcera jsem byl na schuzi Neuronu.} and~\ref{Necham
  rozmrznout angrest z Jarciny zahradky.} below. 
\end{example}

\section{Bicharades}
\label{Jsem pred-predposledni den v Haife.}

\lettrine{\color{red} I} {n} 
this section we explore a particular property of the
dioperadic category $\ttA$ introduced in Example~\ref{Snehulacek se
  tesi na snih.}. As before we
consider triples   $\ttQ = (\ttQ_L, \ttQ_R)$ of a unary
operadic category $\ttQ_L$ with the underlying category $\ttQ$,
together with a unary operadic category~$\ttQ_R$ whose  underlying
category is the  category $\ttQ^{\dagger}$ opposite to $\ttQ$.

\begin{definition}
\label{Prinutim se dnes k behu?}
A {\em bicharadic category\/} 
is a triple $\ttQ =
(\ttQ_L,\ttQ_R)$ such that,  
in the situation described by diagram~(\ref{Pujdeme spolu na
  demonstraci?}), the fiber $\Omega$  of $g^S$ is naturally isomorphic to 
the cofiber $\mho$ of $h_Z$. We moreover require that each
isomorphisms $s: S' \to S''$ and $t: T' \to T''$ in the commutative
square
\begin{equation}
\label{Dnes jsem byl s Jarkou na elektrine.}
\xymatrix@C=-.3em{
F'\ar@{-->}[d]_{s_t} 
&\fib& \ar[d]_s^\cong S' \ar[rr]^{h'} &\rule {3em}{0em}&\ar[d]^t_\cong  T'& 
\cof & D'\ar@{-->}^{t_s}[d]
\\
F''  &\fib& S''\ar[rr]^{h''} &\rule {3em}{0em}&  T''& 
\cof & D''
}
\end{equation}
induce functorial isomorphisms $s_t : F' \to F''$ and $t_s : D' \to D''$
such that $s_t$ for $t = \id$ resp.~$t_s$ for $s = \id$ equals the
induced map between the fibers in $\ttQ_L$ resp.~$\ttQ_R$.

A  {\/\em strict bicharadic category\/} is a triple $\ttQ =
(\ttQ_L,\ttQ_R)$ such that 
the fiber of $g^S$ in diagram~(\ref{Pujdeme spolu na
  demonstraci?}) equals the  cofiber of $h_Z$ in the same diagram, 
so that we have
\begin{equation*}
%\label{Najedl jsem se tak ze sotva chodim.}
\psscalebox{1.0 1.0} % Change this value to rescale the drawing.
{
\begin{pspicture}(1,1)(5.943182,1.8)
\rput(2.142857,1.2418182){$F$}
\rput(2.6,1.2418182){$\fib$}
\rput(2.9714286,1.2418182){$X$}
\rput(5.0363636,1.2418182){$Y$}
\rput(2.9636364,0.74){$\raisebox{.7em}{{\rotatebox{270}{$\fib$}}}$}
\rput(5.045,0.74){$\raisebox{.7em}{{\rotatebox{270}{$\fib$}}}$}
\rput(5.82,0.74){$\raisebox{.7em}{{\rotatebox{270}{$\fib$}}}$}
\rput(2.1818182,0.6881818){$\raisebox{.3em}{{\rotatebox{90}{$=$}}}$}%{$=$}
\rput(2.142857,0.28181817){$F$}
\rput(2.6,0.28181817){$\fib$}
\rput(2.9714286,0.28181817){$S$}
\rput(5.0363636,0.28181817){$T$}
\psline[linecolor=black, linewidth=0.02, arrowsize=0.05291667cm 2.0,arrowlength=1.4,arrowinset=0.0]{->}(3.1818182,1.2418182)(4.818182,1.2818182)
\psline[linecolor=black, linewidth=0.02, arrowsize=0.05291667cm 2.0,arrowlength=1.4,arrowinset=0.0]{->}(3.1818182,0.24181817)(4.818182,0.24181817)
\rput(5.3636365,0.28181817){$\cof$}
\rput(5.3636365,1.2181817){$\cof$}
\rput(5.818182,0.28181817){$D$}
\rput(5.8,1.2181817){$\Xi$}
\rput(4.0,-2.030909){$Z$}
\pscustom[linecolor=black, linewidth=.04]
{
\newpath
\moveto(7.090909,2.06)
}
\pscustom[linecolor=black, linewidth=.04]
{
\newpath
\moveto(7.818182,1.2418182)
}
\pscustom[linecolor=black, linewidth=0.02]
{
\newpath
\moveto(7.4545455,1.06)
}
\pscustom[linecolor=black, linewidth=0.02]
{
\newpath
\moveto(7.090909,0.6963636)
}
\pscustom[linecolor=black, linewidth=0.02]
{
\newpath
\moveto(7.272727,2.5145454)
}

\pscustom[linecolor=black, linewidth=0.02]
{
\newpath
\moveto(5.0,0.059999943)
\lineto(4.772727,-0.6672727)
\curveto(4.6590905,-1.0309088)(4.4772725,-1.4854544)(4.2727275,-1.7581818)
}
\psline[linecolor=black, linewidth=0.02, arrowsize=0.05291667cm 2.0,arrowlength=1.4,arrowinset=0.0]{->}(4.2727275,-1.7581818)(4.181818,-1.8490909)
\pscustom[linecolor=black, linewidth=0.02]
{
\newpath
\moveto(2.9818182,0.059999943)
\lineto(3.209091,-0.6672727)
\curveto(3.322727,-1.0309088)(3.5045452,-1.4854544)(3.709091,-1.7581818)
}
\psline[linecolor=black, linewidth=0.02, arrowsize=0.05291667cm 2.0,arrowlength=1.4,arrowinset=0.0]{->}(3.709091,-1.7581818)(3.8,-1.8490909)
\rput(3.3818183,-2.5581818){$H$}
\rput(4.6181817,-2.5581818){$R$}
\rput(3.7363637,-2.4036363){$\raisebox{.7em}{{\rotatebox{38}{$\fib$}}}$}
\rput(4.3,-2.94){$\raisebox{.7em}{{\rotatebox{38}{$\fib$}}}$}
\rput(4.2636365,-2.3036363){$\raisebox{.7em}{{\rotatebox{150}{$\fib$}}}$}
\pscustom[linecolor=black, linewidth=0.02]
{
\newpath
\moveto(5.8,0.059999943)
\lineto(5.517172,-0.89584416)
\curveto(5.375757,-1.3737662)(5.1494946,-1.9711689)(4.8949494,-2.3296103)
}
\psline[linecolor=black, linewidth=0.02, arrowsize=0.05291667cm 2.0,arrowlength=1.4,arrowinset=0.0]{->}(4.8949494,-2.3296103)(4.7818184,-2.449091)
\rput(4.0,-3.121818){$H$}
\rput(3.7454545,-2.9){$\raisebox{.7em}{{\rotatebox{135}{$=$}}}$}%{$=$}
\rput[b](3.909091,1.4236363){\scriptsize $h_Z$}
\rput[b](4.0,0.42363638){\scriptsize $h$}
\rput[l](4.818182,-1.1218182){\scriptsize $g$}
\rput[bl](5.5454545,-1.3945454){\scriptsize $g^S$}
\rput(2.9181818,-1.1218182){\scriptsize $gh$}
\end{pspicture}
}
\raisebox{-13em}{\rule{0pt}{0pt}}
\end{equation*}
\end{definition}

The naturality of the isomorphism between the fiber $\Omega$ and
the cofiber $\mho$  required in
Definition~\ref{Prinutim se dnes k behu?} refers to commutative
diagrams  of the form
\[ 
\xymatrix@R=1em{&&T'' \ar[dddd]|(.3)\hole_{g''}  &
\\
S\ar[rru]^{h''} \ar[rrddd]_{g'h' = g''h''}   \ar[rrr]^{h'} &&&\ T'.\ar[ul]_u 
\ar[dddl]_{g'}
\\
\\
\\
&&Z&
}
\]
For such a diagram  denote by
$\mho'$ the cofiber of the induced
map $h_{{\raisebox{,1em}{\scriptsize$Z$}}}'$ between the fibers of $g'h'$ and $g'$, and by $\Omega'$ the fiber of 
the  induced map $g'^S$ between the cofiber of $h'$ and the cofiber 
of~$g'h'$. Let $\mho''$ and $\Omega''$ have the similar meanings, so
that, in the notation parallel to~(\ref{Pujdeme spolu na
  demonstraci?}),
\[
X\ \Xarrow {h'_Z}\ Y' \cof \ \mho',\ 
\Omega' \ \fib D' \ \Xarrow {g'^S}\ R,
\ 
X \ \Xarrow {h''_Z} \ Y'' \cof \ \mho''\ \hbox { and } \ 
\Omega''  \fib D''\ \Xarrow {g''^S}\ R.
\]
The commutative triangle of the induced maps between fibers induces the map $u_Z^X : \mho' \to \mho''$ in 
\[
\xymatrix@R=-.3em@C=.8em{\mho'
  \ar@{-->}[rr]^(.55){u_Z^X}   
&&\mho'' \hskip -.5em
\\
\hskip -2em
\hskip 1.9em\raisebox{.7em}{{\rotatebox{270}{$\fib$}}} && 
\hskip -.1em\raisebox{.7em}{{\rotatebox{270}{$\fib$}}} \hskip -.1em
\\
Y' \ar[rr]^{u_Z}  \ar@{<-}@/_.9em/[ddddddr]_(.4){h'_Z}   &&
Y'' \ar@{<-}@/^.9em/[ddddddl]^(.4){h''_Z} \hskip -.5em
\\&& 
\\&& 
\\&& 
\\&&
\\ &&
\\
&X &
} 
\] 
Similarly, we construct 
a map $u_R^S : \Omega' \to \Omega''$.
The naturality is expressed by the
commutativity of
\[
\xymatrix{\mho'  \ar[r]^{\redukce{\scriptsize $u^X_Z$}}   \ar[d]_{\rho'}^\cong  &\mho''
 \ar[d]^{\rho''}_\cong 
\\
\Omega'  \ar[r]^{u_R^S} & \Omega''
}
\]
where $\rho' : \mho' \to \Omega'$ and  $\rho'' : \mho'' \to \Omega''$
are the isomorphism required in Definition~\ref{Prinutim se dnes k behu?}.

Strict bicharadic categories seem to be rare. For instance, 
the arrow category $\Arr$ of Example~\ref{Jarka nemocna jako
  kazdou sudou sobotu.} is strict bicharadic only if $\ttC$ has one object,
i.e.\ when it is an associative monoid. The following
proposition however shows that some  abelian categories produce
(non-strict) bicharadic categories.

\begin{proposition}
\label{Dosel mi caj, musim pit kavu.}
Assume that $\ttA$ is the abelian category $\RMod$ of modules over a
ring $R$. Then $\ttA$ with the fiber-cofiber structure of
Example~\ref{Snehulacek se tesi na snih.} is bicharadic.  
\end{proposition}

\begin{proof}
Assume that the fibers and cofibers are given by the standard kernels
and cokernels in the category of modules, that is
\[
\Cok(h_Z) = \frac{\Ker(g)}{\Im(h) \cap \Ker(g)} \ \ \hbox { and } \ \
\Ker(g^S) = \Ker\left(
\xymatrix@R=1.5em@1{\displaystyle
\frac T{\Im(h)} \ \ar[r]^(.45){[g]} &  \ \displaystyle\frac Z  {\Im(gh)}}\right)
\]
where the map $[g]$ takes the equivalence class of $t \in T$ to the
equivalence class of $g(t) \in Z$.
Define $\rho : \Cok(h_Z) \to \Ker(g^S)$ as the map induced by the
inclusion $\Ker(g) \hookrightarrow T$.

Let us prove that $\rho$ is injective. Given an equivalence class $[k]
\in \Cok(h_Z)$ of some $k \in \Ker (g)$,
$\rho([k]) = 0$ means $k \in \Im(h)$, so  $k \in \Im(h) \cap \Ker(g)$
thus     $[k] = 0$ in
$\Cok(h_Z)$. 
To prove that $\rho$ is surjective, consider the class $[t] \in  T
\hbox { mod } \Im(h)$ of some $t \in T$. Such $[t]$ belongs to $\Ker( [g])$ if 
$g(t) = gh(s)$ for
some $s \in S$. If this happens, we replace $t$ by $t' : =t - h(s)$
which is the same mod
$\Im(h)$ on one hand, and which belongs to $\Ker (g) $ in
the one hand. We conclude that $[t]= [t'] \in \Im(\rho)$. 

The naturality of $\rho : \Cok(h_Z) \to \Ker(g^S)$ follows from the
universality of the objects featured in the proof.
The existence of functorial
induced maps in~(\ref{Dnes jsem byl s Jarkou na elektrine.}) is obvious.
\end{proof}

A natural morphism $\rho: \Cok(h_Z) \to  \Ker(g^S)$ exists in an arbitrary
Abelian category. It is induced by the composition $\delta: =\Ker(g)
\hookrightarrow T \epi \Cok(h)$ in the diagram
\begin{equation}
\label{Jarka byla zas na kapackach.}
\xymatrix@R=1em@C=2em{
\Ker(gh)\komp
\ar@{^{(}->}[dd] \ar[r] & \ar@{->>}[r]\Ker(g)\komp
\ar[rrdd]^\delta \ar@{^{(}->}[dd] &
\Cok(h_Z) \ar@{-->}[dr]^\rho
\\ && &  \Ker(g^S)\komp\ar@{^{(}->}[d]
\\
S \ar@/_1.8em/[rdd]^{gh}   \ar[r]^h& T \ar[dd]_g\ar@{->>}[rr] &&
\Cok(h)\ar[dd]^{g^S}
\\
\\
& Z  \ar@{->>}[rr] && \Cok(gh)
}
\end{equation}
We were however unable to construct an inverse of $\rho$, though the
full power of the axioms of Abelian
categories~\cite[IX.\S2]{Homology} might provide it.

If $\ttA$ is small, we can follow a suggestion of M.~Batanin, invoke
Mitchell's full imbedding theorem~\cite[Theorem~4.4]{Mitchell} and apply a
fully faithful exact functor $F: \ttA \to \RMod$ on the diagram
in~(\ref{Jarka byla zas na kapackach.}). We will get a similar diagram
in the category of $R$-modules in which $R(\rho)$ is an
isomorphism. Since $F$ is full and faithful, $\rho$ must be an
isomorphism too. However, the only Abelian category we will use will
be the category of finite dimensional vector spaces, so the current
formulation of Proposition~\ref{Dosel mi caj, musim pit kavu.} is
sufficient.

\begin{example}
Let us return to the category $\Setp$ of pointed small sets with the
dioperadic structure of Example~\ref{Vecer odjizdim na oslavu
  Silvestra, zitra mam letani.}. For the diagram
\begin{equation}
\label{Jarka je porad smutna.}
\xymatrix@C=.8em{
\{u,\circ\} \ar[rr]^h \ar[rd]_{gh}  & & \{x,y,\circ\} \ar[ld]^g
\\
&\{v,\circ\}
}
\end{equation}
with $h(u) := y$, $g(x) = g(y) := v$, the map $\rho : \mho \to \Omega$
is the inclusion $\{\circ\} \subsetneqq \{x,\circ\}$ of a proper
subset, not an isomorphism. 

This illustrates the importance of linearity for
bicharades:  $g(x) = g(y)$ does not imply that the (nonexistent)
difference $x-y$ belongs to the `kernel' of $g$. 
Let us compare that to the linearized version of~(\ref{Jarka je porad smutna.})
\[
\xymatrix@C=.8em{
\Span{u} \ar[rr]^{\Span h} \ar[rd]_{\Span{gh}}  & & \Span{x,y}
\ar[ld]^{\Span g}
\\
&\Span{v}
}
\]
with $\Span h (u) := y$, $\Span g (x) =  \Span g (y) :=v$. Then $\mho
= \Span {x-y}$, $\Omega = \Span x$ and $\rho : \mho \to \Omega$ is the
  isomorphism given by $\rho(x-y) := x$.
\end{example}

Bicharadic categories admit the following bivariant version of
Kapranov's charades~\cite[Definition~3.2]{kapranov:langlands}. 

\begin{definition}
\label{Kousek od Golan.}
Let $\ttQ$ be a bicharadic category.
A {\em $\ttQ$-bicharade\/} is a functor $\Ch :  \Iso \to \ttV$ from
the luff subcategory of isomorphisms of $\ttQ$ to the base monoidal category
$\ttV$,  with an action
\[
\chi_h : \Ch(T) \ot \Ch(F) \longrightarrow \Ch(D)  \ot \Ch(S)
\]
defined for each morphism $F\ \fib S \stackrel h \longrightarrow T
\cof \ D$, 
compatible with the isomorphisms~(\ref{Dnes jsem
  byl s Jarkou na elektrine.}) in the sense that the induced diagram
\begin{subequations}
\begin{equation}
\label{Nastartuje mi Greta?}
\xymatrix{
\Ch(T') \ot \Ch(F') \ar[r]^{\chi_{h'}}
\ar[d]_{\Ch(t) \ot \Ch(s_t)}  & \Ch(D')  \ot \Ch(S')\ar[d]^{\Ch(t_s) \ot \Ch(s)} 
\\
\Ch(T'') \ot \Ch(F'') \ar[r]^{\chi_{h''}}  & \Ch(D'')  \ot \Ch(S'')
}
\end{equation}
commutes. We also require the commutativity of
\begin{equation}
\label{Zitra odjizdime na Vanoce do Mercina.}
\xymatrix@C=3.2em{
\Flicek ZYF \ar[r]^{\id \ot \chi_{h_Z}} \ar[d]_{\chi_g \ot \id} 
& \Flicek Z\mho X   \ar@{<->}[r]^{\rm symmetry}_\cong &
\Flicek ZX\mho \ar[dd]^{\chi_{gh} \ot \Ch(\rho)}
\\  
\Flicek HTF \ar[d]_{\id \ot \chi_h}
\\
\Flicek HDS & \ar[l]_{\chi_{g^S} \ot \id}  \Flicek R\Omega S
\ar@{<->}[r]^{\rm symmetry}_\cong  & \Flicek RS\Omega
}
\end{equation}
\end{subequations}
for each diagram~(\ref{Pujdeme spolu na demonstraci?}) and the
isomorphism $\rho: \mho \to \Omega$ assumed in
Definition~\ref{Prinutim se dnes k behu?}. 
\end{definition}

\begin{example}
The terminal category $\odot$ is strict bicharadic. A
$\odot$-bicharade is an object $B \in \ttV$ with a morphism $E: B \ot B
\to B\ot B$ such that 
\[
(\id \ot E)(E \ot \id) = (E \ot \id)(\id \ot \sigma) (E \ot \id)(\id
\ot \sigma) (\id \ot E)
\]
where $\sigma : B \ot B \to B \ot B$ is the symmetry. We have no idea
where to place this object.
\end{example}

\begin{example}
The arrow category $\ArrM$ of an associative mono\"{\i}d  $M$  
considered as a category with one
object is strict bicharadic.
An\/ $\ArrM$-bicharade is a collection $B = \{B(a)\}_{a\in M}$ of objects
of $\ttV$ with structure operations $B(a)\! \ot\! B(b) \to
B(c)\! \ot\!  B(d)$ specified for any
$a,b,c,d \in M$ with  \hbox{$ab =
cd$}. An interested reader can easily
figure out what diagram~(\ref{Zitra odjizdime na Vanoce do Mercina.})
requires in this case.  We have no idea
what kind of object we described.  
\end{example}

\begin{example}
\label{A to je mam na sobe teprve podruhe.}
We will show that the determinant, i.e.\ the top  exterior power, is a
bicharade over the abelian category  $\fdV$   
of  finite-dimensional vector spaces with
the fiber-cofiber  structure of  Example~\ref{Snehulacek se tesi na
  snih.}, cf.~Proposition~\ref{Dosel mi caj, musim pit kavu.}. 
First of all, it is clear that any
isomorphism $\varpi :S' \xrightarrow \cong  S''$ in $\fdV$ induces a natural
isomorphism
\begin{subequations}
\begin{equation}
\label{Muj posledni den v Haife 2023.}
\det(\varpi) : \det(S') \stackrel \cong \longrightarrow  \det(S'')
\end{equation}
of determinants. We will also use the fact that any 
short exact sequence $S'' \hookrightarrow S \epi S'$ in $\fdV$ 
induces a natural isomorphism
\begin{equation}
\label{Za chvili sraz s Antonem.}
\det(S') \ot \det(S'') \stackrel \cong\longrightarrow \det(S),
\end{equation}  
cf.~\cite[pages 122--123]{kapranov:langlands}. In particular, we have
natural a isomorphism
\begin{equation}
\label{Za chvili volam jestli pojedeme.}
\det(S' \oplus S'') \cong \det(S') \ot \det(S'') 
\end{equation}
\end{subequations}
related to $S'' \hookrightarrow S'' \oplus S' \epi S'$.

An arbitrary diagram $\Ker(h) \hookrightarrow S \Xarrow h T \epi \Cok(h)$ 
in $\fdV$ extends to the exact sequence at the
bottom of the diagram
\begin{equation}
\label{Teplacky se mi uz trhaji.}
\xymatrix{
F:=\Ker (h)\ \ \ar@{>->}[r] & \ S \ar@/^1.5em/[rr]^h 
\ar@{->>}[r] &\ \Im(h) \ \
\ar@{>->}[r] &T 
\ar@{->>}[r]& \ D=:\Cok (h)}
\end{equation}
which gives rise to two isomorphism as in~(\ref{Za chvili sraz s
  Antonem.}), namely
\[
{}^\bullet\! h :  \det\big(\Im(h)\big) \ot \det(F)   \stackrel
\cong\longrightarrow \det(S)
\ \hbox { and } \ 
h^\bullet :\det(D) \ot  \det\big(\Im(h)\big)  \stackrel \cong\longrightarrow
\det(T).
\]
We then define 
\begin{equation}
\label{Odpoledne jedeme do Mercina.}
\chi_h :\det(T) \ot \det(F) \longrightarrow \det(D) \ot \det(S)
\end{equation}
via the span of isomorphisms
\[
\xymatrix@C=-1em{
& \det(D) \ot \det\big(\Im(h)\big) \ot \det(F)
\ar[ld]_{h^\bullet \ot \id}^\cong \ar[rd]^{\id \ot {}^\bullet\!h}_\cong
\\
\det(T) \ot \det(F)  \ar[rr]^{\chi_h}_\cong     && \det(D) \ot \det(S).
}
\]

We claim that the isomorphisms~(\ref{Muj posledni den v Haife 2023.})
together with structure operations~(\ref{Odpoledne jedeme do
  Mercina.})   make the collection $\det = \{\det(S)\}_{S \in \fdV}$ a
$\fdV$-bicharade. While the commutativity of~(\ref{Nastartuje mi
  Greta?}) is immediate, the commutativity 
of~(\ref{Zitra odjizdime na Vanoce do Mercina.}) requires some
attention.

By elementary linear algebra,  each diagram $S \xrightarrow h T
\xrightarrow g Z$ is isomorphic, in the category of diagrams, to the `standard'
diagram with 
\begin{equation}
\label{Dnes vyzvednu Jarku z Mercina.}
S = \alpha \oplus \beta \oplus \gamma,\ T = \beta \oplus \gamma \oplus
\phi \oplus \psi \ \hbox { and } \ Z = \gamma \oplus \phi \oplus \varsigma
\end{equation} 
for some $\alpha,\beta, \gamma, \phi, \psi, \varsigma \in \fdV$ such that
\[
h(\alpha) = 0, \
h|_{\beta\oplus \gamma} = \id_{\beta\oplus \gamma},\
g(\beta \oplus \psi) = 0\ \hbox { and } \ 
g|_{\gamma\oplus \phi} = \id_{\gamma\oplus \phi},
\]
cf.~the schematic Figure~\ref{Barak bude vymrzly.}.
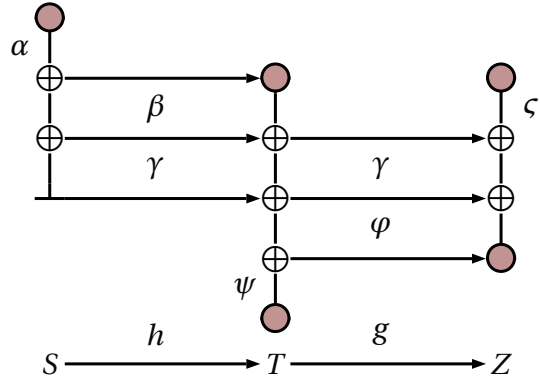
\begin{figure} %Flicek.svg
  \centering
  \psscalebox{1.0 1.0} % Change this value to rescale the drawing.
{
\begin{pspicture}(0,-2.37)(6.995,2.37)
\definecolor{colour0}{rgb}{0.043137256,0.02745098,0.02745098}
\definecolor{colour1}{rgb}{0.7607843,0.5764706,0.5764706}
\rput(0.5,-2.25){\textcolor{colour0}{$S$}}
\rput(3.5,-2.25){\textcolor{colour0}{$T$}}
\rput(6.5,-2.25){\textcolor{colour0}{$Z$}}
\rput(1.9,-1.85){\textcolor{colour0}{$h$}}
\rput(4.9,-1.85){\textcolor{colour0}{$g$}}
\rput(0.1,1.95){\textcolor{colour0}{$\alpha$}}
\rput(1.9,1.15){\textcolor{colour0}{$\beta$}}
\rput(1.9,0.35){\textcolor{colour0}{$\gamma$}}
\rput(4.9,0.35){\textcolor{colour0}{$\gamma$}}
\rput(3.1,-1.25){\textcolor{colour0}{$\psi$}}
\rput(4.9,-0.45){\textcolor{colour0}{$\phi$}}
\rput(6.9,1.15){\textcolor{colour0}{$\varsigma$}}
\psline[linecolor=black, linewidth=0.04, arrowsize=0.05291667cm 2.0,arrowlength=1.4,arrowinset=0.0]{<-}(3.3,1.55)(0.7,1.55)
\psline[linecolor=black, linewidth=0.04, arrowsize=0.05291667cm 2.0,arrowlength=1.4,arrowinset=0.0]{<-}(3.3,0.75)(0.7,0.75)
\psline[linecolor=black, linewidth=0.04, arrowsize=0.05291667cm 2.0,arrowlength=1.4,arrowinset=0.0]{<-}(6.3,0.75)(3.7,0.75)
\psline[linecolor=black, linewidth=0.04, arrowsize=0.05291667cm 2.0,arrowlength=1.4,arrowinset=0.0]{<-}(6.3,-0.05)(3.7,-0.05)
\psline[linecolor=black, linewidth=0.04, arrowsize=0.05291667cm 2.0,arrowlength=1.4,arrowinset=0.0]{<-}(6.3,-0.85)(3.7,-0.85)
\psline[linecolor=black, linewidth=0.04, arrowsize=0.05291667cm 2.0,arrowlength=1.4,arrowinset=0.0]{<-}(3.3,-2.25)(0.7,-2.25)
\psline[linecolor=black, linewidth=0.04, arrowsize=0.05291667cm 2.0,arrowlength=1.4,arrowinset=0.0]{<-}(6.3,-2.25)(3.7,-2.25)
\psline[linecolor=black, linewidth=0.04](6.5,0.95)(6.5,1.55)
\psline[linecolor=black, linewidth=0.04](6.3,1.55)(6.7,1.55)
\psline[linecolor=black, linewidth=0.04](3.3,-1.65)(3.7,-1.65)
\psline[linecolor=black, linewidth=0.04](0.3,2.35)(0.7,2.35)
\psline[linecolor=black, linewidth=0.04](3.5,1.35)(3.5,0.95)
\psline[linecolor=black, linewidth=0.04](3.5,-0.25)(3.5,-0.65)
\rput(3.5,0.75){$\bigoplus$}
\rput(3.5,-0.85){$\bigoplus$}
\rput(6.5,0.75){$\bigoplus$}
\rput(6.5,-0.05){$\bigoplus$}
\rput(0.5,0.75){$\bigoplus$}
\rput(0.5,1.55){$\bigoplus$}
\rput(3.5,-0.05){$\bigoplus$}
\psline[linecolor=black, linewidth=0.04](0.5,2.35)(0.5,1.75)
\psline[linecolor=black, linewidth=0.04](0.5,1.35)(0.5,0.95)
\psline[linecolor=black, linewidth=0.04](0.5,0.55)(0.5,0.15)
\psline[linecolor=black, linewidth=0.04](3.5,0.55)(3.5,0.15)
\psline[linecolor=black, linewidth=0.04](6.5,0.55)(6.5,0.15)
\psline[linecolor=black, linewidth=0.04](6.5,-0.25)(6.5,-0.65)
\psline[linecolor=black, linewidth=0.04](3.5,-1.05)(3.5,-1.65)
\psline[linecolor=black, linewidth=0.04, arrowsize=0.05291667cm 2.0,arrowlength=1.4,arrowinset=0.0]{<-}(3.3,-0.05)(0.7,-0.05)
\psline[linecolor=black, linewidth=0.04](3.3,1.55)(3.7,1.55)
\psline[linecolor=black, linewidth=0.04](3.5,1.55)(3.5,1.35)
\psline[linecolor=black, linewidth=0.04](6.3,-0.85)(6.7,-0.85)
\psline[linecolor=black, linewidth=0.04](6.5,-0.65)(6.5,-0.85)
\psline[linecolor=black, linewidth=0.04](0.3,-0.05)(0.7,-0.05)
\psline[linecolor=black, linewidth=0.04](0.5,0.15)(0.5,-0.05)
\psellipse[linecolor=black, linewidth=0.04, fillstyle=solid,fillcolor=colour1, dimen=outer](3.5,-1.64)(0.2,0.2)
\psellipse[linecolor=black, linewidth=0.04, fillstyle=solid,fillcolor=colour1, dimen=outer](6.5,-0.84)(0.2,0.2)
\psellipse[linecolor=black, linewidth=0.04, fillstyle=solid,fillcolor=colour1, dimen=outer](6.5,1.56)(0.2,0.2)
\psellipse[linecolor=black, linewidth=0.04, fillstyle=solid,fillcolor=colour1, dimen=outer](3.5,1.56)(0.2,0.2)
\psellipse[linecolor=black, linewidth=0.04, fillstyle=solid,fillcolor=colour1, dimen=outer](0.5,2.36)(0.2,0.2)
\end{pspicture}
}
\caption{\label{Barak bude vymrzly.}
The standard form of $S \xrightarrow h T \xrightarrow g Z$.}
\end{figure}
For the vector spaces in~(\ref{Zitra odjizdime na Vanoce do Mercina.})
we have, besides~(\ref{Dnes vyzvednu Jarku z Mercina.}), canonical isomorphisms
\[
X \cong \alpha \oplus \beta,\ Y \cong \gamma \oplus \psi,\
F \cong \alpha,\
\mho \cong \Omega \cong \psi,\
H \cong \varsigma,\
D \cong \phi \oplus \psi \  \hbox { and } \
R \cong \phi\oplus \varsigma.
\]
Using the above isomorphisms we conclude that the spaces at
the  boundary of the diagram 
\[
\xymatrix{
\ar@{<->}[d]^\relax
\Mikinka ZYF \ar@{<->}[r]^\relax & \Mikinka Z\mho 
X  \ar@{<->}[r]^\relax & \Mikinka ZX\mho \ar@{<->}[dd]^\relax
\\
\ar@{<->}[d]^\relax \ar@{<->}[r]^\relax
\Mikinka HTF & 
\ar@{<->}[ru]^\relax \ar@{<->}[rd]^\relax
\ar@{<->}[d]^\relax  \ar@{<->}[u]^\relax
\ar@{<->}[ul]^\relax \ar@{<->}[dl]^\relax
\alpha \oplus \beta \oplus \gamma \oplus \phi \oplus
\psi \oplus \varsigma
\\
\ar@{<->}[r]^\relax
\Mikinka HDS & \ar@{<->}[r]^\relax \Mikinka R\Omega S & \Mikinka RS\Omega 
}
\]
are canonically isomorphic to the space in the center, so that the
triangles and thus also the boundary rectangle commutes. Applying
$\det(-)$ on the boundary rectangle and invoking the iterated~\eqref{Za chvili
  volam jestli pojedeme.} we verify the commutativity of
\[
\xymatrix{
\ar@{<->}[d]^\relax
\Alicek ZYF \ar@{<->}[r]^\relax & \Alicek Z\mho 
X  \ar@{<->}[r]^\relax & \Alicek ZX\mho \ar@{<->}[dd]^\relax
\\
\ar@{<->}[d]^\relax
\Alicek HTF & 
\\
\ar@{<->}[r]^\relax
\Alicek HDS & \ar@{<->}[r]^\relax \Alicek R\Omega S & \Alicek RS\Omega 
}
\]
which is~\eqref{Zitra odjizdime na Vanoce do Mercina.} for the
determinant. 

This proves the commutativity of~\eqref{Zitra odjizdime na Vanoce do
  Mercina.} for diagrams $S \xrightarrow h T \xrightarrow g Z$ in the standard
form. Since the isomorphism
\[
\xymatrix@R=1.5em{S' \ar[r]^{h'} \ar@{<->}[d]_\cong  & T'
  \ar@{<->}[d]_\cong\ar[r]^{g'} & Z' \ar@{<->}[d]^\cong
\\
S'' \ar[r]^{h''} & T'' \ar[r]^{g''} & Z''
}
\]
of diagrams induces compatible isomorphisms of the objects featured 
in~\eqref{Pujdeme spolu na demonstraci?}, and since each diagram $S
\xrightarrow h T \xrightarrow g Z$ is isomorphic to a diagram in the
standard form, this establishes the commutativity 
of~\eqref{Zitra odjizdime na Vanoce do
  Mercina.} for an arbitrary diagram by~(\ref{Nastartuje mi Greta?}). 
We believe there is a smarter proof not using the standard forms, 
but we have not been able to find it.
\end{example}

\begin{example}
In this example we use the material
of~\cite[\S3.3]{kapranov:langlands}. Let $\qfdV$ denote the category
of finite-dimensional vector spaces over a finite field ${\mathbb
  F}$. The Tits building of an \hbox{$n$-dimensional}
vector space $S \in \qfdV$ is the 
simplicial set $B_\bullet(S)$  
whose $m$-simplices are flags of subspaces $S_0 \subset
\cdots \subset S_m \subset S$ such that either $S_0 \not= \{0\}$ or
$S_m \not= S$. It is known that the homology $H_i
(B_\bullet(S);{\mathbb k})$ with coefficients in a field $\mathbb k$
is zero for $i \not= 0,n\!-\!1$. The space $\St(S) : =H_{n-1}
(B_\bullet(S);{\mathbb k})$ is the {\/\em Steinberg module\/} of $S$.

Trusting in~\cite[page~132]{kapranov:langlands}, each short exact sequence
$S' \hookrightarrow S \epi S''$ in  $\qfdV$ induces a natural map
\[
\mu _{S'SS''} : \St(S') \ot \St(S'') \longrightarrow  \St(S) .
\]
Mimicking the methods of Example~\ref{A to je mam na sobe teprve podruhe.} 
we get, for each morphism $h$ as
in~(\ref{Teplacky se mi uz trhaji.}), the span
\[
\xymatrix@C=-1em{
& \St(D) \ot \St\big(\Im(h)\big) \ot \St(F)
\ar[ld]_{h^\bullet \ot \id}^\relax \ar[rd]^{\id \ot {}^\bullet\!h}_\relax
\\
\St(T) \ot \St(F)     && \St(D) \ot \St(S).
}
\]
Since neither ${}^\bullet\!h$ nor $h^\bullet$ are isomorphisms in
general, the above diagrams  interpreted as operations 
\[
\chi_h :\St(T) \ot \St(F) \longrightarrow \St(D) \ot \St(S)
\]
can make  $\St = \{\St(S)\}_{S \in 
\qfdV}$ only a $\qfdV$-bicharade in 
the category of spans in $\fdV$.
\end{example}

\section{General dioperadic categories}
\label{Vsechno je spatne.}

\lettrine{\color{red} W} {\ e} focus again on
a pair $\ttQ = (\ttQ_L,\ttQ_R)$ of two
operadic categories such that $\ttQ$ is the underlying category of~$\ttQ_L$,
and the underlying category of  $\ttQ_R$ is the opposite 
category~$\ttQ^\dagger$. If~both $\ttQ_L$ and $\ttQ_R$ are unary,
we are in the context of Definition~\ref{Musim se
  objednat na endokrinologii za mesic.}. Although the structures
discussed below are
straightforward generalizations of the corresponding ones in 
Section~\ref{Jsem dnes nejak bez
  energie.}, several new important examples will occur.

\subsection{Dioperadic categories -- definition and examples}
Recall from~\cite[page 1623]{duodel} that each  operadic category
comes with a
cardinality functor~$|\dash|$
to the skeletal category  $\sFSet$  of
finite sets.  Objects of  $\sFSet$
are linearly ordered sets $[n] :=\{1,\ldots,n\}$,
$n\in \bbN$, and morphisms are arbitrary maps between these sets.
In the above situation we thus
have two cardinality functors,  $|\dash|_L:\ttQ\to \sFSet$ and $|\dash|_R:\ttQ\to
\sFSet^\dagger$,  given by the operadic
structures of $\ttQ_L$ and $\ttQ_R$, which together form
the  {\em bicardinality \/}  functor
\[
\bic \dash : \ttQ \longrightarrow \sFSet \times
\sFSet^\dagger.
\]
The latter category will be used so often that we
introduce a condensed notation
\[
\d:=
\sFSet \times \sFSet^\dagger .
\] 
The objects of $\dFin$ are pairs $([m],[n])$ of finite ordinals; we will
sometimes use the
shorter notation $(m,n)$ instead of  $([m],[n])$.

Assuming $\bic S = (a,m)$ and $\bic T
= (n,b)$, a morphism $h : S \to T$ of $\ttQ$ 
has  $n$ fibers $F_i$ and $m$ cofibers~$D_j$, $i \in [n]$,
$j \in [m]$, which we express by
\begin{equation}
\label{Dnes bude cely den prset.}
\Rada F1n \ \fib S \stackrel h\longrightarrow T \cof \ \Rada D1m. 
\end{equation}
We  will sometimes denote the $i$th fiber $F_i$ of $h$ in $\ttQ_L$
by $\inv {\redukce{{$\vec h$}}} (i)$   
and the $j$th cofiber $D_j$ of
$h$, i.e.\ the $j$th fiber of $h^\dagger$ in $\ttQ_R$, by 
$\inv{\redukce{$\cev h$}}(j)$. We emphasize that here neither $\vec h$ or $\cev h$
are actual maps, the arrows only  distinguish between fibers
and cofibers.

\begin{definition}
\label{Musim se objednat na endokrinologii.}
We say that  $\ttQ = (\ttQ_L,\ttQ_R)$ is {\/\em left dioperadic}
if $\ttQ_L$ is perfect.
Dually, $\ttQ$ is {\/\em right dioperadic\/} if 
$\ttQ_R$ is perfect. Finally, $\ttQ$
is {\/\em dioperadic\/} if both $\ttQ_L$ and $\ttQ_R$ are
perfect. As in the unary case, we denote by $\ttQ_F$ resp.~$\ttQ_C$
the fiber, resp.~the cofiber subcategory.
\end{definition}

\begin{example}
\label{Grete prokluzuje spojka.}
The archetype of a dioperadic category is $\d$ itself, with
the identity automorphism playing the r\^ole of the bicardinality.  
 A morphism 
$\phi : (a,m) \to (n,b)$ is a pair
$\phi = (\vec \phi, \cev \phi)$ of arbitrary maps 
$\vec \phi : [a] \to [n]$ and $\cev \phi : [b] \to [m]$ of finite sets.
The $i$th fiber of
$\phi$
is  the pair $\big(\inv{\vec \phi}(i),1$\big) and the $j$th cofiber the pair
$\big(1,\inv{\cev \phi}(j)\big)$, for $i \in [n]$, $j \in [m]$. 
Here $\inv{\vec
  \phi}(i)$ is the $i$th fiber of the map $\vec \phi : [a] \to [n]$ in
$\sFSet$ defined, as in~\cite[Section~1]{duodel},
to be the pullback of $\vec \phi$ along the map $[1] \to [n]$ which picks
up $i \in [n]$:
\begin{equation}
\label{Zase tady nekdo neco hrabe!}
\xymatrix@C=1.5em@R=.1em{\inv{\vec \phi}(i)\ar[dd] \ar[rr]^{I_i} && [a] \ar[dd]^{\vec \phi}
\\
&\hbox{\raisebox{1.5em}{\Huge \hskip -2.3em  $\lrcorner$}}
\\
[1] \ar[rr]^{1 \ \longmapsto \ i}     &&\ [n]. 
}
\end{equation}
The fiber $\inv{\vec \phi}(i)$ is unique by the skeletality of
$\sFSet$.  The expression $\inv{\cev
  \phi}(j)$ has the obvious similar meaning. It is easy to verify that 
\begin{equation}
\label{Zitra pojedu na Technion.}
\left(\d\right)_F \cong \left(\d\right)_C \cong \sFSet.
\end{equation}
\end{example}
 
\begin{example}
\label{Zalije mi Jarka jeste jednou kyticky?}
The subtle combinatorics of dioperadic categories is nicely illustrated by
the category $\dBq$  of bibouquets, which is a `colored' version of
the dioperadic category $\d$ and simultaneously a bivariant version of
the operadic category of 
$\frakC$-bouquets in~\cite[Example~1.7]{duodel}. 
Given  a set of `colors' $\frakC$,
a {\em $\frakC$-bibouquet\/} is
a pair of maps
$B=   (\beta,\Ateb) : [n] \times [m] \to  \frakC \times \frakC$, $n,m
\in \bbN$, represented by the `fraction'
\begin{equation}
\label{Byl jsem v Zoo.}
\bez {\Rada u1n}{\rule {0pt}{.8em}{\Rada v1m}},\
u_i:= \beta(i),\ v_j := \Ateb(j), \ \hbox{\  $i \in [n]$, $j \in [m]$},
\end{equation}
of ordered lists of colors. The biarity of $B$ is $(m,n)$. 
The symbol in~(\ref{Byl jsem v Zoo.}) can be visualized as a~directed
corolla with $m$ inputs labelled $\Rada v1m$ and $n$ outputs
labelled $\Rada u1n$, as in
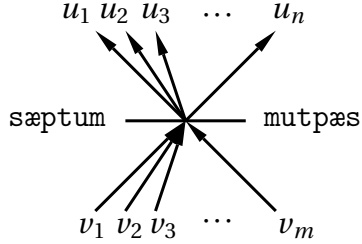
\[
\psscalebox{1.0 1.0} % Change this value to rescale the drawing.
{
\begin{pspicture}(0,-1.8)(2.945,1.8)
\psline[linecolor=black, linewidth=0.04, arrowsize=0.05291667cm 2.0,arrowlength=2.5,arrowinset=0.0]{<-}(0.155,1.2)(1.355,0.0)
\psline[linecolor=black, linewidth=0.04, arrowsize=0.05291667cm 2.0,arrowlength=2.5,arrowinset=0.0]{<-}(0.555,1.2)(1.355,0.0)
\psline[linecolor=black, linewidth=0.04, arrowsize=0.05291667cm 2.0,arrowlength=2.5,arrowinset=0.0]{<-}(0.955,1.2)(1.355,0.0)
\psline[linecolor=black, linewidth=0.04, arrowsize=0.05291667cm 2.0,arrowlength=2.5,arrowinset=0.0]{<-}(2.555,1.2)(1.355,0.0)
\rput[b](-0.1,1.3){$u_1$}
\rput[b](0.4,1.3){$u_2$}
\rput[b](0.955,1.3){$u_3$}
\rput[b](2.725,1.3){$u_n$}
\rput[tl](-.1,-1.3){$v_1$}
\rput[tl](0.4,-1.3){$v_2$}
\rput[t](1.05,-1.3){$v_3$}
\rput[tl](2.555,-1.3){$v_m$}
\psline[linecolor=black, linewidth=0.04, arrowsize=0.05291667cm 2.0,arrowlength=2.5,arrowinset=0.0]{<-}(1.355,0.0)(0.155,-1.2)
\psline[linecolor=black, linewidth=0.04, arrowsize=0.05291667cm 2.0,arrowlength=2.5,arrowinset=0.0]{<-}(1.355,0.0)(0.555,-1.2)
\psline[linecolor=black, linewidth=0.04, arrowsize=0.05291667cm 2.0,arrowlength=2.5,arrowinset=0.0]{<-}(1.355,0.0)(0.955,-1.2)
\psline[linecolor=black, linewidth=0.04, arrowsize=0.05291667cm 2.0,arrowlength=2.5,arrowinset=0.0]{<-}(1.355,0.0)(2.555,-1.2)
\rput[b](1.8,1.3){$\cdots$}
\rput[t](1.8,-1.3){$\cdots$}
\psline[linecolor=black, linewidth=0.04](0.555,0.0)(2.155,0.0)
\rput[r](0.3,0){\tt s\ae ptum}
\rput[l](2.4,0){\tt mutp\ae s}
\end{pspicture}
}
\] 

Let $B'=   (\beta',\Ateb') : [a] \times [m] \to  \frakC \times \frakC$ and
$B''=   (\beta'',\Ateb'') : [n] \times [b] \to  \frakC \times \frakC$
be bibouquets. Morphisms $\Phi : B' \to B''$ are arbitrary maps
$\phi = (\vec \phi,\cev \phi) : (a,m) \to (n,b)$ in $\d$.
The $i$th fiber of $\Phi$ is the bibouquet
\[
\inv{\vec \Phi}(i): = \big(\beta' I_i,1 \mapsto \Ateb(i)\big) :  \inv{\vec \phi}(i) \times
[1] \to  \frakC \times \frakC, \ i \in [n],
\]
where $I_i$ is the upper horizontal morphism in~(\ref{Zase tady nekdo
  neco hrabe!}). The cofibers are defined similarly. For example, 
for $\phi = (\vec \phi,\cev \phi) : (4,3) \to (2,3)$ given by
\[
\vec \phi(1) = \vec \phi(3) :=2,\
\vec \phi(2) = \vec \phi(4) :=1,\
\cev \phi(1) : =3, \ \cev \phi(2) = \cev \phi(3) :=2,
\]
the fiber-cofiber diagram reads
\[
\xymatrix{
\bez {u'_2,u'_4}{\rule {0pt}{.8em}{u''_1}}\ ,
\bez {u'_1,u'_3}{\rule {0pt}{.8em}{u''_2}}\
\fib\
\bez {u'_1,u'_2,u'_3,u'_4}{\rule {0pt}{.8em}{v'_1,v'_2,v'_3}}\
\ar[rr]^\Phi
& & %\stackrel B\longrightarrow  
\bez {u''_1,u''_2}{\rule {0pt}{.8em}{v''_1,v''_2,v''_3}}\
\cof \
\bez {v'_1}{\rule {0pt}{.8em}{\emptyset}}\ ,
\bez {v'_2}{\rule {0pt}{.8em}{v''_2,v''_3}}\ ,
\bez {v'_3}{\rule {0pt}{.8em}{v''_1}} \ .
}
\]

The fiber subcategory  $\dBq_F$  consists of bibouquets of the form
$B=   (\beta,\Ateb) : [n] \times [1] \to  \frakC \times \frakC$, $n
\in \bbN$. Notice that $\Ateb$ just picks an element of $\frakC$, the
`root color' of $B$. The maps of bibouquets $B' \to B''$ in $\dBq_F$
whose root colors
coincide are arbitrary maps of $\dBq$, otherwise there is no map
between $B'$ and $B''$. 
We recognize  $\dBq_F$ as the operadic category $\Bq(\frakC)$ of
$\frakC$-bouquets from~\cite[Example~1.7]{duodel}, with the maps
existing only between `flowers'
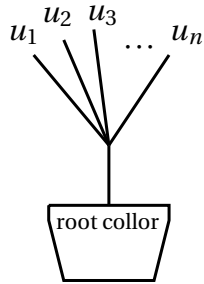
\[
\psscalebox{1.0 1.0} % Change this value to rescale the drawing.
{
\begin{pspicture}(0,-1.81)(2.36,1.81)
\psline[linecolor=black, linewidth=0.04](0.2,1.21)(1.2,0.01)
\psline[linecolor=black, linewidth=0.04](1.2,0.01)(2.0,1.21)
\psline[linecolor=black, linewidth=0.04](1.2,0.01)(1.2,-0.79)
\rput[bl](-0.15,1.35){$u_1$}
\rput[bl](0.3,1.59){$u_2$}
\rput[bl](.9,1.61){$u_3$}
\rput[bl](2.0,1.35){$u_n$}
\psline[linecolor=black, linewidth=0.04](0.4,-0.79)(2.0,-0.79)(2.0,-0.99)(1.8,-1.79)(0.6,-1.79)(0.4,-0.99)(0.4,-0.79)
\rput[bl](0.5,-1.1){\scriptsize root collor}
\psline[linecolor=black, linewidth=0.04](1.2,0.01)(0.6,1.41)
\psline[linecolor=black, linewidth=0.04](1.2,0.01)(1.0,1.55)
\rput[bl](1.4,1.21){$\cdots$}
\end{pspicture}
}
\]
with identical pots.  The cofiber category $\dBq_C$ is
isomorphic to $\Bq(\frakC)$ too. 
Notice that the fiber subcategory  $\dBq_F$ is the proper subcategory of the full
subcategory of  $\dBq$ of objects of 
bicardinalities $(n,1)$, $n \in \bbN$, and that 
$\dBq$ is not the product  $\dBq_F \times \dBq_F^\dagger$. 
\end{example}

\begin{example}
\label{Pobezim po strese?}
Let $\ttO$ be a right unital operadic category with a family of local terminal
objects~\eqref{Jsem posledni den v Cuernavace.}.  Then  $\tO =
(\tO_L,\tO_R)$,  with $\tO_L := \ttO$ and
$\tO_R$ the unary operadic category 
of Example~\ref{Za chvili na vlak.},
is dioperadic. Indeed, $\tO_L$ is right unital, thus perfect by
Proposition~\ref{Jarka opet nemocna.}, with $\tO_F = \ttO$. 
By Example~\ref{Dominik mel skvelou myslenku.}, also $\tO_C$ is
perfect, with $\tO_F$ the discrete category with objects $U_c,\  c
\in \pi_0(\ttO)$. 
\end{example}

\begin{example}[due to M.~Batanin]
\label{Super tautological}
Let $\ttO$ be a perfect operadic category. Consider $\cO =
(\cO_L,\cO_R)$ with $\cO_L := \ttO$ and $\cO_R$ the category opposite
to the underlying category of $\ttO$, with the cardinality defined by
$|T|_R :=  [0]$ for each $T$. 
Then $\cO$ is dioperadic, with $\cO_F = \ttO_F$ and
$\cO_C$ the empty~category.
\end{example}

\begin{example}
The d\'ecalage $\sfD(\ttC)$ of a category $\ttC$ is a unital operadic category. 
The construction of Example~\ref{Pobezim po strese?} leads to the
dioperadic category $\widehat{\sfD}(\ttA)$ with
$\widehat{\sfD}(\ttA)_F = \sfD(\ttC)$ and $\widehat{\sfD}(\ttA)_C
\cong \ttC_{\rm disc}$, the discrete category with objects of $\ttC$.
\end{example}

\subsection{Operads, cooperads and bimodules}
This subsection contains generalizations 
of the notions introduced in Subsection~\ref{Nemam se cim chlubit.}.
The only, but essential, conceptual difference is that the 
translation of Definition~\ref{Musim se objednat na ocni.} 
to the form in Example~\ref{Poletim do Izraele?}
does not make much sense if objects of cardinalities different from
$[1]$ occur, cf.~Remark~\ref{zase} below.

\begin{definition}
\label{Je tam jen 11 stupnu!}
\begin{subequations}
A {\em $\ttQ$-operad} over a left dioperadic category $\ttQ$ is a collection $\oP =
\{\oP(T)\}_{T \in \ttQ_F}$ of objects of a symmetric monoidal category
$\ttV$,
with structure operations
\begin{equation}
\label{Vecer pujdu na Brezanku}
\gamma_h:   \oP(T) \ot \oP(F_1) \ot \cdots \ot \oP(F_n)
\longrightarrow \oP(S), 
\end{equation}
specified for an arbitrary morphism $\Rada F1n \ \fib S \xrightarrow h T$ in
$\ttQ_F$, and satisfying the
associativity axiom in item~(i) of~\cite[Definition~1.11]{duodel}.

Dually, a {\em $\ttQ$-cooperad \/} over a right dioperadic category $\ttQ$ is a
collection $\cooP =
\{\cooP(T)\}_{T \in \ttQ_C}$ of objects of $\ttV$ with structure
operations 
\begin{equation}
\label{a tam jako abstinent budu jen tak koukat.}
\delta_h : \cooP(T) \longrightarrow \cooP(D_1) \ot
\cdots \ot \cooP(D_m) \ot \cooP(S)
\end{equation}
that are given for any morphism 
$S \xrightarrow h T \cof \ \Rada D1m$ 
in $\ttQ_C$, and that satisfy the
obvious dual form of the associativity diagram in item~(i) 
of~\cite[Definition~1.11]{duodel}.
\end{subequations}
\end{definition}

\begin{example}[due to M.~Batanin]
\label{Dva dny za sebou.}
The `dioperadic envelope' $\ttO \mapsto \cO$ constructed in
Example~\ref{Super tautological} defines a full and faithful embedding of the
category of perfect operadic categories to the category of dioperadic
categories. The category of $\ttO_F$-operads is isomorphic to
the category of $\cO$-operads.
\end{example}

Inspired by the shorthand used in~\cite[Definition~1.1]{duodel} we
introduce,
for $h :S \to T$ as in~\eqref{Dnes bude cely den prset.}  
and a $\ttQ$-operad $\oP$ over a left dioperadic
category $\ttQ$, the notation
\begin{subequations}
\begin{equation}
\label{Tuto sobotu jedu}
\oP(h) := \oP(F_1) \ot \cdots \ot \oP(F_n).
\end{equation}
If $\ttQ$ is right dioperadic and $\cooP$ a $\ttQ$-cooperad, we
similarly denote
\begin{equation}
\label{za Andulkou do Berlina.}
\cooP(h) := \cooP(D_1) \ot \cdots \ot \cooP(D_m).
\end{equation} 
\end{subequations}
With this shorthand, the operations in~(\ref{Vecer pujdu na Brezanku})
and~(\ref{a tam jako abstinent budu jen tak koukat.}) assume a nice
concise form
\[
\gamma_h:   \oP(T) \ot \oP(h) \longrightarrow \oP(S)
\ \hbox { and } \
\delta_h :  \cooP(T) \longrightarrow \cooP(h) \ot \cooP(S).
\]

\begin{definition}
\label{Dnes jsem byl s Jarkou na obede.}
Let $\ttQ$ be a dioperadic category, $\oP$
a $\ttQ_F$-operad and $\cooP$ a $\ttQ_C$-cooperad.
A {\/\em {\CP}-bimodule \/} is a collection $\oM = \{\oM(S)\}_{S
  \in \ttV}$ of objects of $\ttV$ with structure morphisms
\begin{equation}
\label{Byly Dusicky.}
\omega_h :  \oM(T)  \ot \oL(F_1) \ot \cdots \ot \oL(F_n) 
\longrightarrow
\cooP(D_1) \ot \cdots \ot \cooP(D_m) \ot  \oM(S)
\end{equation}
given for each $\Rada F1n \ \fib S \xrightarrow h T \cof \ \Rada D1m$. 
 
We moreover require a compatibility between this action
and the (co)operad structures of $\oL$ and $\cooP$. 
Namely, let
\[
\xymatrix{S\ar[rr]^h \ar[dr]_{gh}   &&T \ar[ld]^g
\\
&Z
}
\]
be a commutative diagram in $\ttQ$. 
Both $g$ and $gh$ have $a := |Z|_L$ fibers; denote by
$\Rada h1a$ the maps between them induced by $h$. 
Likewise $h$ and $gh$ have $b := |S|_R$ cofibers;
denote by $\rada{g^1}{g^b}$ the maps between them
induced by $g$. We demand the commutativity of the~diagram
\[
\xymatrix{
\oM(Z) \ot \oP(g) \ot \oP(h_1) \ot \cdots \ot \oP(h_{a})
\ar[r]^(.62){\id \ot \gamma_{h_1} \ot \cdots \ot \gamma_{h_{a}}}
\ar[dd]_{\omega_g \ot \id}
&
\oM(Z) \ot \oP(gh)\ar[d]^{\omega_{gh}}
\\
& \cooP(gh) \ot \oM(S) \ar[d]^{\delta_{g^1} \ot \cdots \ot
  \delta_{g^{b}} \ot \id}
\\
\cooP(g) \ot \oM(T) \ot \oP(h) \ar[r]^(.42){\id \ot \omega_h}
&
\cooP(g^1) \ot \cdots \ot \cooP(g^{b}) \ot \cooP(h) \ot \oM(S)
}
\]
which uses the shorthand~(\ref{Tuto sobotu jedu})--(\ref{za
  Andulkou do Berlina.}) and  implicitly assumes the equality of the set
of the (co)fibers of
a map and the set of the (co)fibers of the induced map between these
(co)fibers required by Axiom (iv) of the
definition of an operadic category in~\cite[page~1634]{duodel}.
The structure morphism~(\ref{Byly Dusicky.}) is symbolized by the
`flow diagram' in Figure~\ref{Porad mne pali ty ruce.}.
\end{definition}

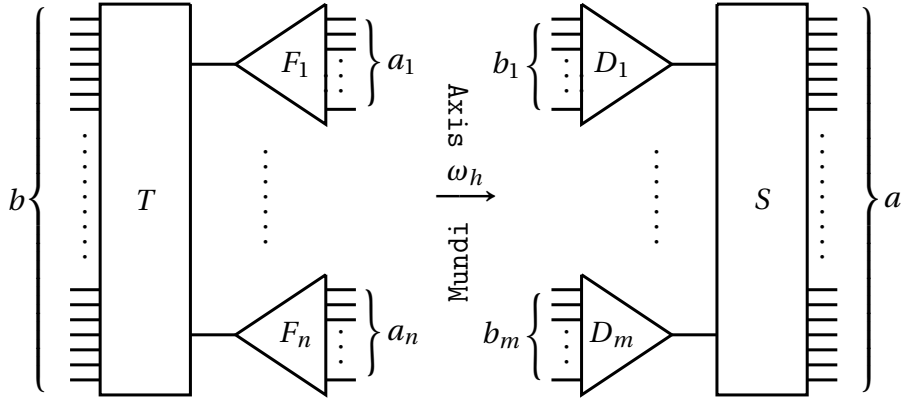
\begin{figure}
\centering
\psscalebox{1.0 1.0} % Change this value to rescale the drawing.
{
\begin{pspicture}(0,-2.6286852)(11.7,2.6286852)
\psline[linecolor=black, linewidth=0.04](7.54,2.4086852)(7.14,2.4086852)
\psline[linecolor=black, linewidth=0.04](7.54,2.2086854)(7.14,2.2086854)
\psline[linecolor=black, linewidth=0.04](7.54,2.0086854)(7.14,2.0086854)
\psline[linecolor=black, linewidth=0.04](7.54,1.2086853)(7.14,1.2086853)
\psline[linecolor=black, linewidth=0.04](7.54,-1.1913147)(7.14,-1.1913147)
\psline[linecolor=black, linewidth=0.04](7.54,-1.3913147)(7.14,-1.3913147)
\psline[linecolor=black, linewidth=0.04](7.54,-1.5913147)(7.14,-1.5913147)
\psline[linecolor=black, linewidth=0.04](7.54,-2.3913147)(7.14,-2.3913147)
\psline[linecolor=black, linewidth=0.04](9.34,2.6086853)(9.34,-2.5913148)(10.54,-2.5913148)(10.54,2.6086853)(9.34,2.6086853)
\psline[linecolor=black, linewidth=0.04](10.54,2.4086852)(10.94,2.4086852)
\psline[linecolor=black, linewidth=0.04](10.54,2.2086854)(10.94,2.2086854)
\psline[linecolor=black, linewidth=0.04](10.54,2.0086854)(10.94,2.0086854)
\psline[linecolor=black, linewidth=0.04](10.54,1.8086853)(10.94,1.8086853)
\psline[linecolor=black, linewidth=0.04](10.54,1.6086853)(10.94,1.6086853)
\psline[linecolor=black, linewidth=0.04](10.54,1.4086853)(10.94,1.4086853)
\psline[linecolor=black, linewidth=0.04](10.54,-1.1913147)(10.94,-1.1913147)
\psline[linecolor=black, linewidth=0.04](10.54,-1.3913147)(10.94,-1.3913147)
\psline[linecolor=black, linewidth=0.04](10.54,-1.5913147)(10.94,-1.5913147)
\psline[linecolor=black, linewidth=0.04](10.54,-1.7913147)(10.94,-1.7913147)
\psline[linecolor=black, linewidth=0.04](10.54,-1.9913146)(10.94,-1.9913146)
\psline[linecolor=black, linewidth=0.04](10.54,-2.1913147)(10.94,-2.1913147)
\psline[linecolor=black, linewidth=0.04](10.54,-2.3913147)(10.94,-2.3913147)
\psline[linecolor=black, linewidth=0.04](10.54,1.2086853)(10.94,1.2086853)
\rput(7.94,1.8086853){$D_1$}
\rput(7.94,-1.7913147){$D_m$}
\rput(9.94,0.008685303){$S$}
\rput[r](6.74,1.8086853){$b_1$}
\rput(7.54,1.6086853){$\cdot$}
\rput(7.54,1.4086853){$\cdot$}
\rput[l](11.54,0.008685303){$a$}
\rput(10.74,0.8086853){$\cdot$}
\rput(10.74,0.6086853){$\cdot$}
\rput(10.74,0.4086853){$\cdot$}
\rput(7.34,1.8086853){$\cdot$}
\rput(7.34,1.6086853){$\cdot$}
\rput(7.34,1.4086853){$\cdot$}
\rput(10.74,0.20868531){$\cdot$}
\rput(10.74,0.008685303){$\cdot$}
\rput(10.74,-0.1913147){$\cdot$}
\rput(7.34,-1.7913147){$\cdot$}
\rput(7.34,-1.9913146){$\cdot$}
\rput(7.34,-2.1913147){$\cdot$}
\rput(8.54,0.6086853){$\cdot$}
\rput(8.54,0.4086853){$\cdot$}
\rput(8.54,0.20868531){$\cdot$}
\rput(8.54,0.008685303){$\cdot$}
\rput(8.54,-0.1913147){$\cdot$}
\rput(8.54,-0.3913147){$\cdot$}
\rput(8.54,-0.5913147){$\cdot$}
\rput(10.74,-0.3913147){$\cdot$}
\rput(6.94,-1.7913147){$\left\{\rule{0em}{2em}\right.$}
\rput[r](6.74,-1.7913147){$b_m$}
\rput(6.94,1.8086853){$\left\{\rule{0em}{2em}\right.$}
\rput(11.34,0.008685303){$\left.\rule{0em}{7.2em}\right\}$}
\rput(10.74,-0.5913147){$\cdot$}
\rput(10.74,-0.7913147){$\cdot$}
\psline[linecolor=black, linewidth=0.04](4.14,2.4086852)(4.54,2.4086852)
\psline[linecolor=black, linewidth=0.04](4.14,2.2086854)(4.54,2.2086854)
\psline[linecolor=black, linewidth=0.04](4.14,2.0086854)(4.54,2.0086854)
\psline[linecolor=black, linewidth=0.04](4.14,1.2086853)(4.54,1.2086853)
\psline[linecolor=black, linewidth=0.04](4.14,-1.1913147)(4.54,-1.1913147)
\psline[linecolor=black, linewidth=0.04](4.14,-1.3913147)(4.54,-1.3913147)
\psline[linecolor=black, linewidth=0.04](4.14,-1.5913147)(4.54,-1.5913147)
\psline[linecolor=black, linewidth=0.04](4.14,-2.3913147)(4.54,-2.3913147)
\psline[linecolor=black, linewidth=0.04](2.34,2.6086853)(2.34,-2.5913148)(1.14,-2.5913148)(1.14,2.6086853)(2.34,2.6086853)
\psline[linecolor=black, linewidth=0.04](1.14,2.4086852)(0.74,2.4086852)
\psline[linecolor=black, linewidth=0.04](1.14,2.2086854)(0.74,2.2086854)
\psline[linecolor=black, linewidth=0.04](1.14,2.0086854)(0.74,2.0086854)
\psline[linecolor=black, linewidth=0.04](1.14,1.8086853)(0.74,1.8086853)
\psline[linecolor=black, linewidth=0.04](1.14,1.6086853)(0.74,1.6086853)
\psline[linecolor=black, linewidth=0.04](1.14,1.4086853)(0.74,1.4086853)
\psline[linecolor=black, linewidth=0.04](1.14,-1.1913147)(0.74,-1.1913147)
\psline[linecolor=black, linewidth=0.04](1.14,-1.3913147)(0.74,-1.3913147)
\psline[linecolor=black, linewidth=0.04](1.14,-1.5913147)(0.74,-1.5913147)
\psline[linecolor=black, linewidth=0.04](1.14,-1.7913147)(0.74,-1.7913147)
\psline[linecolor=black, linewidth=0.04](1.14,-1.9913146)(0.74,-1.9913146)
\psline[linecolor=black, linewidth=0.04](1.14,-2.1913147)(0.74,-2.1913147)
\psline[linecolor=black, linewidth=0.04](1.14,-2.3913147)(0.74,-2.3913147)
\psline[linecolor=black, linewidth=0.04](1.14,1.2086853)(0.74,1.2086853)
\rput(3.74,1.8086853){$F_1$}
\rput(3.74,-1.7913147){$F_n$}
\rput(1.74,0.008685303){$T$}
\rput[l](4.94,1.8086853){$a_1$}
\rput(4.14,1.6086853){$\cdot$}
\rput(4.14,1.4086853){$\cdot$}
\rput[r](0.14,0.008685303){$b$}
\rput(0.94,0.8086853){$\cdot$}
\rput(0.94,0.6086853){$\cdot$}
\rput(0.94,0.4086853){$\cdot$}
\rput(4.34,1.8086853){$\cdot$}
\rput(4.34,1.6086853){$\cdot$}
\rput(4.34,1.4086853){$\cdot$}
\rput(0.94,0.20868531){$\cdot$}
\rput(0.94,0.008685303){$\cdot$}
\rput(0.94,-0.1913147){$\cdot$}
\rput(4.34,-1.7913147){$\cdot$}
\rput(4.34,-1.9913146){$\cdot$}
\rput(4.34,-2.1913147){$\cdot$}
\rput(3.34,0.6086853){$\cdot$}
\rput(3.34,0.4086853){$\cdot$}
\rput(3.34,0.20868531){$\cdot$}
\rput(3.34,0.008685303){$\cdot$}
\rput(3.34,-0.1913147){$\cdot$}
\rput(3.34,-0.3913147){$\cdot$}
\rput(3.34,-0.5913147){$\cdot$}
\rput(0.94,-0.3913147){$\cdot$}
\rput[l](4.94,-1.7913147){$a_n$}
\rput{-178.4686}(9.526179,-3.4553142){\rput(4.74,-1.7913147){$\left\{\rule{0em}{2em}\right.$}}
\rput{-178.4686}(9.429971,3.7434){\rput(4.74,1.8086853){$\left\{\rule{0em}{2em}\right.$}}
\rput(0.34,0.008685303){$\left\{\rule{0em}{7.2em}\right.$}
\rput(0.94,-0.5913147){$\cdot$}
\rput(0.94,-0.7913147){$\cdot$}
\psline[linecolor=black, linewidth=0.04](7.54,2.6086853)(7.54,1.0086854)(8.74,1.8086853)(7.54,2.6086853)
\psline[linecolor=black, linewidth=0.04](7.54,-0.9913147)(7.54,-2.5913148)(8.74,-1.7913147)(7.54,-0.9913147)
\psline[linecolor=black, linewidth=0.04](8.74,1.8086853)(9.34,1.8086853)
\psline[linecolor=black, linewidth=0.04](8.74,-1.7913147)(9.34,-1.7913147)
\psline[linecolor=black, linewidth=0.04](4.14,2.6086853)(4.14,1.0086854)(2.94,1.8086853)(4.14,2.6086853)(4.14,2.6086853)
\psline[linecolor=black, linewidth=0.04](4.14,-0.9913147)(4.14,-2.5913148)(2.94,-1.7913147)(4.14,-0.9913147)(4.14,-0.9913147)
\psline[linecolor=black, linewidth=0.04](2.94,1.8086853)(2.34,1.8086853)
\psline[linecolor=black,
linewidth=0.04](2.94,-1.7913147)(2.34,-1.7913147)(2.34,-1.7913147)
\rput[b](6,-.1){\Large$\stackrel{\omega_h}\longrightarrow$}
\rput(5.9,0.38){\rotatebox{-90}{\tt \hskip -2.7em  Axis}} 
\rput(5.9,0.38){\rotatebox{90}{\tt  \hskip- 4em Mundi}}
\end{pspicture}
}
\caption{
The structure of a
bimodule. The rectangles
represent the corresponding pieces of
the bimodule, the
triangles the corresponding pieces of the operad and the cooperad. 
The equalities $a = a_1 + \cdots + a_n$ and $b = b_1 + \cdots + b_m$
following from the axioms of the operadic categories $\ttQ_L$ and
$\ttQ_R$ are assumed.
}
\label{Porad mne pali ty ruce.}
\end{figure}

\begin{remark}
\label{zase}
Assume, as in Example~\ref{Poletim do Izraele?}, that the base monoidal 
category $\ttV$ is the category of vector spaces and try to rewrite
the bimodule action~\eqref{Byly Dusicky.} 
to the form~\eqref{Koupil jsem si dve zpetna zrcatka.},
with $\oR$ the linear dual of $\cooP$. Instead of the diagram in
Figure~\ref{Porad mne pali ty ruce.} we obtain a structure sketched in
Figure~\ref{Za chvili pojedu za Jaruskou.}, where the row of triangles
on the left symbolizes the spaces $\oR(D_1),\ldots,\oR(D_m)$. 
\begin{figure}
  \centering
\psscalebox{1.0 1.0} % Change this value to rescale the drawing.
{
\begin{pspicture}(0,-2.6286852)(11.88,2.6286852)
\psline[linecolor=black, linewidth=0.04](1.14,2.4086852)(0.74,2.4086852)
\psline[linecolor=black, linewidth=0.04](1.14,2.2086854)(0.74,2.2086854)
\psline[linecolor=black, linewidth=0.04](1.14,2.0086854)(0.74,2.0086854)
\psline[linecolor=black, linewidth=0.04](1.14,1.2086853)(0.74,1.2086853)
\psline[linecolor=black, linewidth=0.04](1.14,-1.1913147)(0.74,-1.1913147)
\psline[linecolor=black, linewidth=0.04](1.14,-1.3913147)(0.74,-1.3913147)
\psline[linecolor=black, linewidth=0.04](1.14,-1.5913147)(0.74,-1.5913147)
\psline[linecolor=black, linewidth=0.04](1.14,-2.3913147)(0.74,-2.3913147)
\psline[linecolor=black, linewidth=0.04](9.54,2.6086853)(9.54,-2.5913148)(10.74,-2.5913148)(10.74,2.6086853)(9.54,2.6086853)
\psline[linecolor=black, linewidth=0.04](10.74,2.4086852)(11.14,2.4086852)
\psline[linecolor=black, linewidth=0.04](10.74,2.2086854)(11.14,2.2086854)
\psline[linecolor=black, linewidth=0.04](10.74,2.0086854)(11.14,2.0086854)
\psline[linecolor=black, linewidth=0.04](10.74,1.8086853)(11.14,1.8086853)
\psline[linecolor=black, linewidth=0.04](10.74,1.6086853)(11.14,1.6086853)
\psline[linecolor=black, linewidth=0.04](10.74,1.4086853)(11.14,1.4086853)
\psline[linecolor=black, linewidth=0.04](10.74,-1.1913147)(11.14,-1.1913147)
\psline[linecolor=black, linewidth=0.04](10.74,-1.3913147)(11.14,-1.3913147)
\psline[linecolor=black, linewidth=0.04](10.74,-1.5913147)(11.14,-1.5913147)
\psline[linecolor=black, linewidth=0.04](10.74,-1.7913147)(11.14,-1.7913147)
\psline[linecolor=black, linewidth=0.04](10.74,-1.9913146)(11.14,-1.9913146)
\psline[linecolor=black, linewidth=0.04](10.74,-2.1913147)(11.14,-2.1913147)
\psline[linecolor=black, linewidth=0.04](10.74,-2.3913147)(11.14,-2.3913147)
\psline[linecolor=black, linewidth=0.04](10.74,1.2086853)(11.14,1.2086853)
\rput(1.54,1.8086853){$D_1$}
\rput(1.54,-1.7913147){$D_m$}
\rput(10.14,0.008685303){$S$}
\rput[r](0.34,1.8086853){$b_1$}
\rput(1.14,1.6086853){$\cdot$}
\rput(1.14,1.4086853){$\cdot$}
\rput[l](11.74,0.008685303){$a$}
\rput(10.94,0.8086853){$\cdot$}
\rput(10.94,0.6086853){$\cdot$}
\rput(10.94,0.4086853){$\cdot$}
\rput(0.94,1.8086853){$\cdot$}
\rput(0.94,1.6086853){$\cdot$}
\rput(0.94,1.4086853){$\cdot$}
\rput(10.94,0.20868531){$\cdot$}
\rput(10.94,0.008685303){$\cdot$}
\rput(10.94,-0.1913147){$\cdot$}
\rput(0.94,-1.7913147){$\cdot$}
\rput(0.94,-1.9913146){$\cdot$}
\rput(0.94,-2.1913147){$\cdot$}
\rput(2.14,0.6086853){$\cdot$}
\rput(2.14,0.4086853){$\cdot$}
\rput(2.14,0.20868531){$\cdot$}
\rput(2.14,0.008685303){$\cdot$}
\rput(2.14,-0.1913147){$\cdot$}
\rput(2.14,-0.3913147){$\cdot$}
\rput(2.14,-0.5913147){$\cdot$}
\rput(10.94,-0.3913147){$\cdot$}
\rput(0.54,-1.7913147){$\left\{\rule{0em}{2em}\right.$}
\rput[r](0.34,-1.7913147){$a_n$}
\rput(0.54,1.8086853){$\left\{\rule{0em}{2em}\right.$}
\rput(11.54,0.008685303){$\left.\rule{0em}{7.2em}\right\}$}
\rput(10.94,-0.5913147){$\cdot$}
\rput(10.94,-0.7913147){$\cdot$}
\psline[linecolor=black, linewidth=0.04](5.94,2.4086852)(6.34,2.4086852)
\psline[linecolor=black, linewidth=0.04](5.94,2.2086854)(6.34,2.2086854)
\psline[linecolor=black, linewidth=0.04](5.94,2.0086854)(6.34,2.0086854)
\psline[linecolor=black, linewidth=0.04](5.94,1.2086853)(6.34,1.2086853)
\psline[linecolor=black, linewidth=0.04](5.94,-1.1913147)(6.34,-1.1913147)
\psline[linecolor=black, linewidth=0.04](5.94,-1.3913147)(6.34,-1.3913147)
\psline[linecolor=black, linewidth=0.04](5.94,-1.5913147)(6.34,-1.5913147)
\psline[linecolor=black, linewidth=0.04](5.94,-2.3913147)(6.34,-2.3913147)
\psline[linecolor=black, linewidth=0.04](4.14,2.6086853)(4.14,-2.5913148)(2.94,-2.5913148)(2.94,2.6086853)(4.14,2.6086853)
\psline[linecolor=black, linewidth=0.04](9.54,2.4086852)(9.14,2.4086852)
\psline[linecolor=black, linewidth=0.04](9.54,2.2086854)(9.14,2.2086854)
\psline[linecolor=black, linewidth=0.04](9.54,2.0086854)(9.14,2.0086854)
\psline[linecolor=black, linewidth=0.04](9.54,1.8086853)(9.14,1.8086853)
\psline[linecolor=black, linewidth=0.04](9.54,1.6086853)(9.14,1.6086853)
\psline[linecolor=black, linewidth=0.04](9.54,1.4086853)(9.14,1.4086853)
\psline[linecolor=black, linewidth=0.04](9.54,-1.1913147)(9.14,-1.1913147)
\psline[linecolor=black, linewidth=0.04](9.54,-1.3913147)(9.14,-1.3913147)
\psline[linecolor=black, linewidth=0.04](9.54,-1.5913147)(9.14,-1.5913147)
\psline[linecolor=black, linewidth=0.04](9.54,1.2086853)(9.14,1.2086853)
\rput(5.54,1.8086853){$F_1$}
\rput(5.54,-1.7913147){$F_n$}
\rput(3.54,0.008685303){$T$}
\rput[l](6.74,1.8086853){$a_1$}
\rput(5.94,1.6086853){$\cdot$}
\rput(5.94,1.4086853){$\cdot$}
\rput[r](8.54,0.008685303){$m$}
\rput(9.34,0.8086853){$\cdot$}
\rput(9.34,0.6086853){$\cdot$}
\rput(9.34,0.4086853){$\cdot$}
\rput(6.14,1.8086853){$\cdot$}
\rput(6.14,1.6086853){$\cdot$}
\rput(6.14,1.4086853){$\cdot$}
\rput(9.34,0.20868531){$\cdot$}
\rput(9.34,0.008685303){$\cdot$}
\rput(9.34,-0.1913147){$\cdot$}
\rput(6.14,-1.7913147){$\cdot$}
\rput(6.14,-1.9913146){$\cdot$}
\rput(6.14,-2.1913147){$\cdot$}
\rput(5.14,0.6086853){$\cdot$}
\rput(5.14,0.4086853){$\cdot$}
\rput(5.14,0.20868531){$\cdot$}
\rput(5.14,0.008685303){$\cdot$}
\rput(5.14,-0.1913147){$\cdot$}
\rput(5.14,-0.3913147){$\cdot$}
\rput(5.14,-0.5913147){$\cdot$}
\rput(9.34,-0.3913147){$\cdot$}
\rput[l](6.74,-1.7913147){$b_m$}
\rput{-178.4686}(13.125537,-3.4072094){\rput(6.54,-1.7913147){$\left\{\rule{0em}{2em}\right.$}}
\rput{-178.4686}(13.029327,3.7915046){\rput(6.54,1.8086853){$\left\{\rule{0em}{2em}\right.$}}
\rput(8.74,0.008685303){$\left\{\rule{0em}{7.2em}\right.$}
\rput(9.34,-0.5913147){$\cdot$}
\rput(9.34,-0.7913147){$\cdot$}
\rput[l](7,0.1){\Large$\stackrel \varpi\longrightarrow$}
\psline[linecolor=black, linewidth=0.04](1.14,2.6086853)(1.14,1.0086854)(2.34,1.8086853)(1.14,2.6086853)
\psline[linecolor=black, linewidth=0.04](1.14,-0.9913147)(1.14,-2.5913148)(2.34,-1.7913147)(1.14,-0.9913147)
\psline[linecolor=black, linewidth=0.04](2.34,1.8086853)(2.94,1.8086853)
\psline[linecolor=black, linewidth=0.04](2.34,-1.7913147)(2.94,-1.7913147)
\psline[linecolor=black, linewidth=0.04](5.94,2.6086853)(5.94,1.0086854)(4.74,1.8086853)(5.94,2.6086853)(5.94,2.6086853)
\psline[linecolor=black, linewidth=0.04](5.94,-0.9913147)(5.94,-2.5913148)(4.74,-1.7913147)(5.94,-0.9913147)(5.94,-0.9913147)
\psline[linecolor=black, linewidth=0.04](4.74,1.8086853)(4.14,1.8086853)
\psline[linecolor=black, linewidth=0.04](4.74,-1.7913147)(4.14,-1.7913147)(4.14,-1.7913147)
\psline[linecolor=black, linewidth=0.04](9.54,-1.7913147)(9.14,-1.7913147)
\psline[linecolor=black, linewidth=0.04](9.54,-1.9913146)(9.14,-1.9913146)
\psline[linecolor=black, linewidth=0.04](9.54,-2.1913147)(9.14,-2.1913147)
\psline[linecolor=black, linewidth=0.04](9.54,-2.3913147)(9.14,-2.3913147)
\end{pspicture}
}
\caption{An attempt to rewrite bimodule action~\eqref{Byly Dusicky.}.}
\label{Za chvili pojedu za Jaruskou.}
\end{figure}
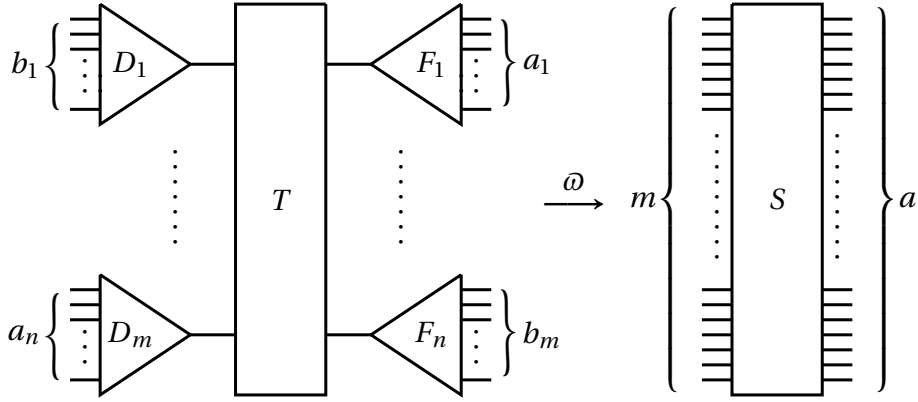  
The bicardinality of the the left figure is $(a_1+
\cdots +a_n,b_1+
\cdots +b_m)$, while the bicardinality of $S$ is $(a,m)$. The
map $\varpi$ may exist only if the respective cardinalities agree, i.e.\ if
\begin{equation}
\label{Neco na mne lezlo ale uz snad odlezlo.}
a_1+ \cdots + a_n = a \ \hbox { and } \
b_1+ \cdots + b_m = m.
\end{equation}
The first equality follows from the axioms of the operadic
category $\ttQ_L$. 
But there is no reason why the second equation should hold, unless we
assume  that $b_j = 1$ for all $j$,  which brings us back to the
unary case. Furthermore, equations~(\ref{Neco na mne lezlo ale uz snad
  odlezlo.}) are not self dual. 
However, the structures like the one in Figure~\ref{Za chvili pojedu za
  Jaruskou.} make sense in the context of the `contramodules' discussed in
the \hbox{Epilogue.} 
\end{remark}

\begin{example}
\label{Zitra jedu pro vetvicky.}
The category $\d$
is the terminal dioperadic category.
Since the fiber subcategory of $\d$ is $\sFSet$, cf.~(\ref{Zitra
  pojedu na Technion.}), 
$\d$-operads are the same as $\Fin$-operads, which are, 
by~\cite[Example~1.15]{duodel}, the most classical $\Sigma$-
(aka symmetric) operads of V.A.~Artamonov~\cite{Artamonov} in
1969, i.e.\ collections $\oP = \Coll \oP n$ of $\Sigma$-modules in~$\ttV$
with structure operations
\[
\oP(n) \ot \oP(k_1) \ot \cdots \ot \oP(k_n) \longrightarrow 
\oP(k_1 + \cdots + k_n),
\]
that are given for any $n, \Rada k1n \in \bbN$ and that satisfy the
standard associativity and equivariance 
axioms~\cite[Definition~1]{markl:handbook}.
Dually, $\d$-cooperads are $\Sigma$-cooperads with structure
operations
\[
\cooP(l_1 + \cdots + l_m) \longrightarrow  
\cooP(l_1) \ot \cdots \ot \cooP(l_m) \ot      \cooP(m),
\]
given for any $m, \Rada l1m \in \bbN$ and satisfying the formal duals
of the operad axioms.

Finally, \CP-bimodules are bicollections $\oM = \{\oM(n,m)\}_{n,m \in \bbN}$
such that each $\oM(n,m)$ is a $\Sigma_m$-$\Sigma_n$-bimodule in $\ttV$,
with structure operations 
\[
\oM(n,l_1 + \cdots + l_m) \ot  \oP(k_1) \ot \cdots \ot \oP(k_n)
\longrightarrow
\cooP(l_1) \ot \cdots \ot \cooP(l_m) \ot \oM(k_1 + \cdots + k_n,m)
\]
defined for any  $n,m, \Rada k1n, \Rada l1m \in \bbN$. We are however
not going to write the bimodule axioms explicitly here since they can be
obtained easily by expanding the definition. Operads, cooperads and
bimodules over the dioperadic category  $\dBq$ of $\frakC$-bibouquets
introduced in
Example~\ref{Zalije mi Jarka jeste jednou kyticky?} are the
$\frakC$-colored extensions of the above objects.

There are important 
particular cases of \CP-bimodules with $\oM(n,m) = 0$ whenever $m\not= 0$.
With $\oM(n) := \oM(n,0)$ for $n \in \bbN$, the structure operations
\[
\oM(n) \ot \oP(k_1) \ot \cdots \ot \oP(k_n) \longrightarrow  
\oM(k_1 + \cdots + k_n)
\]
are that of a {\em right $\oP$-module\/} in the sense
of~\cite[page~1476]{zebrulka},
cf.~also~\cite[Definition~2.2]{cyclo}. In this particular case the
cooperad $\cooP$ does not enter the picture at all, since there are no cofibers. 
\hbox{\CP-bimodules} thus generalize the `classical' right operadic modules
which capture  e.g.\ the structure of the compactification of the moduli space of
points in a manifold~\cite[Proposition~6.4]{markl:cf}.
\end{example}

\begin{example}
The dioperadic category $\d$ has a dioperadic subcategory $\dd:=
\Delta_{\rm alg} \times \Delta_{\rm alg}^\dagger$, where   $\Delta_{\rm alg}$
is the category of finite ordinals, including the empty one. The
related objects, i.e.~operads, cooperads and bimodules, are 
non-$\Sigma$ versions of the corresponding objects in 
Example~\ref{Zitra jedu pro vetvicky.}.
\end{example}

\section{General bioperadic categories}
\label{Jaroslav Vodrazka nahral ty fantasticke veci.}

\lettrine{\color{red} T} {he} 
aim of this section is to generalize the material of
Section~\ref{Posloucham ricercar.} to categories with
objects of arbitrary bicardinality. 
The analysis of the morphism~(\ref{Dnes bude cely den
  prset.}) now consists of the schemes
\begin{equation}
\label{Dnes pojedeme na obed.}
\xymatrix@C=-.4em@R=-.2em{
&&F_i  \ar[rrrrrrrrrrrr]^{\id_{F_i}} &&&&&&&&&&&& F_i  &\cof& \
\rada{C^1_{F_i}}{C^{|F_i|_R}_{F_i}}
\\
&&\raisebox{.7em}{{\rotatebox{270}{$\fib$}}}&&&&&&&&&&&&
\raisebox{.7em}{{\rotatebox{270}{$\fib$}}}
\\
\rada{U^1_S}{U^a_S}\hskip -1em &\fib & \ar[dddddd]_h  S \ar[rrrrrrrrrrrr]^{\id_S}  
&&&&&&&&&&&& S \ar[dddddd]^h     &\cof&  \hskip -1em \rada{C^1_S}{C^m_S}
\\
\\
\\
 &&&&&&\hbox{\hskip 1em \boxed{\tt analysis\rule{0em}{.85em}}}&&&&&&
\\
\\
\\
\rada{U^1_T}{U^n_T}\hskip -1em &\fib &   T \ar[rrrrrrrrrrrr]^{\id_T}
&&&&&&&&&&&& T
&\cof& \hskip -1em  \rada{C^1_T}{C^b_T}
\\
&&\raisebox{.3em}{{\rotatebox{90}{\hskip -.2em$\fib$}}}&&&&&&&&&&&&
\raisebox{.3em}{{\rotatebox{90}{\hskip -.2em$\fib$}}}
\\
\rada{U^1_{D_j}}{U^{|D_j|_L}_{D_j}}     &\fib&D_j \ar[rrrrrrrrrrrr]^{\id_{D_j}}  &&&&&&&&&&&& D_j  &&
}
\end{equation}
given for each $i \in [n]$, $j \in [m]$.

\subsection{Operad algebras and coalgebras revisited}
We follow the
principle formulated after Definition~\ref{Podari se mi ten clanek vubec
  publikovat?} 
and try to define $\oP$-algebras, $\cooP$-coalgebras and
$\oM$-traces for general dioperadic categories. 
Guided by our philosophy, the {\em source\/} $\sou
X$ of an object
$X \in \ttQ$  must be the list
of connected components of the fibers of \hbox{$\id_X : X \to X$} and,
similarly, the {\/\em target \/}  $\tar X$  the list
of connected components of the cofibers. That is,
in the situation
\[
\rada{U_X^1}{U_X^n}  \ \fib X \stackrel {\id_X}\longrightarrow X
\cof \  \rada{C_X^1}{C_X^m}, \ \
\bic X = (n,m)
\]
we shall define
\begin{equation}
\label{Budu muset jet na chalupu.}
\sou X := \big\{\ \pi_0(U^i_X) \ | \ 1 \leq i \leq n\ \big\}
\ \hbox { and } \
\tar X := \big\{\ \pi_0(C^j_X) \ | \ 1 \leq j \leq m\ \big\}.
\end{equation}

For a collection $C = \{C_c\}_{c \in \pi_0(\ttQ)}$ of objects 
of the base monoidal category 
$\ttV$ parametrized by the connected components of
$\ttQ$ we will abbreviate
\begin{equation}
\label{Vubec se mi tam nechce.}
C_{\sou X} := \bigotimes_{c \in\, \sou X} C_c 
\ \hbox { and } \
C_{\tar X} := \bigotimes_{c \in\, \tar X} C_c .
\end{equation}
The non-unary version of the associativity~(\ref{Dnes predsedam Rade}) 
of a $\oP$-operad action
would require, for each
$\Rada F1n \ \fib   S \xrightarrow h T$ in the fiber subcategory
$\ttQ_F \subset \ttQ_L$, isomorphisms
\begin{subequations}
\begin{align}
\label{Pujdu dnes}
A_{\tar T} &\cong A_{\tar S}, \hbox { and}
\\
\label{s Jarkou}
A_{\sou T}  &\cong  A_{\tar {F_1}} \ot \cdots \ot A_{\tar{F_n}}.
\end{align}
\end{subequations}
Likewise, the non-unary version~(\ref{s2}) 
of the associativity for a $\cooP$-cooperad action needs
\begin{subequations}
\begin{align}
\label{k}
B_{\sou S} &\cong B_{\sou T}, \ \hbox { and }
\\
\label{Pakousum?}
B_{\tar S} &\cong   B_{\sou {D_1}} \ot \cdots \ot B_{\sou{D_m}}  
\end{align}
\end{subequations}
for arbitrary $S \xrightarrow h T \cof \ \Rada D1m$ in the cofiber
subcategory 
$\ttQ_C \subset \ttQ_R$.
Since~\eqref{Jarce hrabe z chalupy.} for the bimodule action
involves both  $\oP$-algebras and $\cooP$-coalgebras, it
requires all the isomorphisms above. The isomorphisms~(\ref{s Jarkou})
and~(\ref{Pakousum?}) must moreover hold for arbitrary $h:S \to T$ as
in~\eqref{Dnes bude cely den prset.}, not only for morphisms in
$\ttQ_F$ resp.~$\ttQ_C$. Notice that the isomorphisms
\[
A_{\sou {F_1}} \ot \cdots \ot A_{\sou {F_n}} \cong A_{\sou S}
\ \hbox { and } \
B_{\tar {D_1}} \ot \cdots \ot B_{\tar {D_m}} \cong B_{\tar T}
\]
follow, as in the unary case, from~\cite[Axiom~(iv),
page~1634]{duodel}  which guarantees the equalities of disjoint unions
\[
\sou S = \sou {F_1} \sqcup \cdots \sqcup \sou {F_n} 
\ \hbox { and } \
\tar T = \tar {D_1}  \sqcup \cdots \sqcup  \tar {D_m}.\
\]

\subsection{General bioperadic categories, (co)algebras and traces}
The central definition of this section, formulated below, imposes
on dioperadic categories additional 
conditions  that imply isomorphisms~(\ref{Pujdu dnes})--(\ref{Pakousum?}).
Furthermore, every object of $\ttQ_F$ must
have exactly one `output,' since algebras for  $\ttQ_F$-operads 
must be structures with
many-to-one operations. For the same reasons, 
every object of  $\ttQ_C$ must have
exactly one `input.' This gives rise to an extra
numerological  condition, which is automatically satisfied 
in the unary case.

\begin{definition}
\label{az do dnu antiky}
A {\/\em left bioperadic\/} category is a left dioperadic category $\ttQ$
in Definition~\ref{Musim se objednat na endokrinologii.} such that
all objects of the fiber subcategory $\ttQ_F$
are of  bicardinality  $(a,1)$ with some $a \in
\bbN$ and, for any morphism  $\Rada F1n \ \fib S \stackrel
h\longrightarrow T$ in $\ttQ_F$,
\begin{subequations}
\begin{align}
\label{2a}
(F_i \redukce{$\stackrel {\id_{F_i}}\longrightarrow F_i$}  \cof \ C_{F_i} \ \& \
\rada{U^1_{T}}U^n_{T} \fib T \stackrel {\id_T}\longrightarrow  T) &\Longrightarrow 
(C_{F_i} = U^i_T), \ 1 \leq i \leq n,\ \hbox { and}
\\
  \label{2abis}
(S \Xarrow {\id_S} S\ \cof \ C_S  
\ \& \ T \Xarrow {\id_T} T \cof \ C_T) &\Longrightarrow C_S = C_T.
\end{align}
\end{subequations}
Dually, a right dioperadic category $\ttQ$ is {\/\em right
  bioperadic\/} if
and all objects of the cofiber subcategory $\ttQ_C$ are of bicardinality  
$(1,b)$ with some $b \in
\bbN$ and, 
for every $S\redukce{$\stackrel h\longrightarrow$} T \cof \Rada D1m$ in $\ttQ_C$,
\begin{subequations}
\begin{align}
\label{2b}
(U_{D_j} \fib D_j \stackrel {\id_{D_j}}\longrightarrow  D_j
 \ \& \
S \stackrel {\id_S}\longrightarrow S \cof \ \rada {C^1_S}{C^m_S})
& \Longrightarrow 
(U_{D_j} = C^j_S), \ 1 \leq j \leq m,\ \hbox { and}
\\
\label{2bbis}
(U_T \ \fib T \Xarrow {\id_T} T   
\ \& \ U_S \ \fib S \Xarrow {\id_S} S) &\Longrightarrow U_T = U_S.
\end{align}
Finally, a dioperadic category is {\/\em bioperadic\/} 
if it is both left and right bioperadic,  
and if ~(\ref{2a}) and~(\ref{2b}) are fulfilled for any morphism $h:
S \to T$ of $\ttQ$.
\end{subequations}
\end{definition}

Notice that $m=b=1$ in the analysis~\eqref{Dnes pojedeme na obed.} of a
morphism $h : S\to T$ of the fiber subcategory, so the identity
in~\eqref{2abis} indeed involves only one cofiber. Likewise, the identity
in~\eqref{2bbis} involves only one fiber.
The equalities in~\eqref{2a} imply
$\tar {F_1} \sqcup \cdots \sqcup \tar {F_n} = \sou T$
and the equalities in~\eqref{2b} imply that
$\sou {D_1} \sqcup \cdots \sqcup \sou {D_m} = \tar S$.
The objects in the intersection $\ttQ_F \cap \ttQ_C$ have the
bicardinality $(1,1)$, so the morphisms in $\ttQ_F \cap \ttQ_C$ 
are of the form $F \ \fib  S
\xrightarrow h T  \cof\ D$ and they satisfy 
equalities~(\ref{Posloucham Ceske a morevske barokni varhany II.}) 
in Remark~\ref{Hraji si se cteckou.}.

\begin{proposition}
\label{zasute refremy kterym uz chybeji slova}
The inclusion of the sets
\[
\big\{ U^i_T \ | \ \rada {U^1_T}{U^n_T} \ \fib T \stackrel
  {\id_T}\longrightarrow T,\ T \in \ttQ_F, \ i\in [n]\big\}
\subseteq
\big\{ C^j_T \ | \ T \stackrel
  {\id_T}\longrightarrow T \cof \  \rada {C^1_T}{C^m_T}, \ T \in
  \ttQ_F,\ j \in  [m] \big\}
\]
holds in an arbitrary left bioperadic category $\ttQ$. In a right
bioperadic category,
\[
\big\{ U^i_T \ | \ \rada {U^1_T}{U^n_T} \ \fib T \stackrel
  {\id_T}\longrightarrow T,\ T \in \ttQ_C, \ i\in [n]\big\}
\supseteq
\big\{ C^j_T \ | \ T \stackrel
  {\id_T}\longrightarrow T \cof \  \rada {C^1_T}{C^m_T}, \ T \in
  \ttQ_C,\ j \in  [m] \big\}.
\]
If\/ $\ttQ$ is bioperadic, then
\[
\big\{ U^i_T \ | \ \rada {U^1_T}{U^n_T} \ \fib T \stackrel
  {\id_T}\longrightarrow T,\ T \in \ttQ, \ i\in [n]\big\}
=
\big\{ C^j_T \ | \ T \stackrel
  {\id_T}\longrightarrow T \cof \  \rada {C^1_T}{C^m_T}, \ T \in
  \ttQ,\ j \in  [m] \big\},
\]
so  the sets of elements of the sources and the sets of elements
of the targets are the same. 
\end{proposition}

\begin{proof}
A straightforward modification of the proof of Proposition~\ref{Najdu
  odvahu se dnes v te strasne zime projet na kole?}.   
\end{proof}

\begin{proposition}
\label{Pojedeme do Mercina.}
Let $\ttQ = (\ttQ_L,\ttQ_R)$ be left dioperadic.
Then
\begin{itemize}
\item[(i)]
the subcategory $\ttQ_F \subset \ttQ_L$ is an operadic subcategory,
\item [(ii)]
the fibers of all maps in $\ttQ$ have 
bicardinalities of the form $(a,1)$, with some $a \in
\bbN$,
\item[(iii)]
for each morphism $h : S \to T$ as in~\eqref{Dnes bude cely den
  prset.} and $i \in [n]$, $\bic {\inv {\redukce{$\vec h$}}( i)} = 
\inv{\redukce{$\vec{\bic h}$}}(i)$.
\end{itemize}
Dually, if \/ $\ttQ = (\ttQ_L,\ttQ_R)$ is right dioperadic, then
\begin{itemize}
\item[(\ii)]
the subcategory $\ttQ_C \subset \ttQ_R$ is an
operadic subcategory,
\item [(\iii)]
the cofibers of all maps in $\ttQ$ have  
bicardinalities of the form $(1,b)$, with some $b \in
\bbN$,
\item[(\iiii)]
for each morphism $h : S \to T$ as in~\eqref{Dnes bude cely den
  prset.} and $j \in [m]$, $\bic {\inv {\redukce{$\cev  h$}}( j)} = 
\inv{\redukce{$\cev{\bic h}$}}(j)$.
\end{itemize}
\end{proposition}

In item~(iii),  $\bic {\inv {\redukce{$\vec h$}}( i)}$ is the biarity
of the $i$th fiber of $h$ and $\inv{\redukce{$\vec{\bic h}$}}(i)$ is
the $i$th fiber of the induced map $\bic h : \bic S \to \bic T$.
Recall from~\cite[page~1634]{duodel} that 
the cardinality functor in any operadic category satisfies
\begin{equation}
\label{Za tyden jedu do Berlina.}
|f^{-1}(i)| = |f|^{-1}(i), \ \hbox { for } 
\ f: S \longrightarrow T, \ i \in |T|.  
\end{equation}
The equalities in items~(iii) and~(\rmiiii) are dioperadic versions
of~(\ref{Za tyden jedu do Berlina.}). 

\begin{proof}[Proof of Proposition~\ref{Pojedeme do Mercina.}]
Item~(i) is obvious. Item (ii) follows from the left bioperadicity
and the fact that the fibers of all maps belong to $\ttQ_F$. 
Let us analyze~(iii). We have, by definitions
\begin{eqnarray*}
\bic{
\inv{\redukce{$\vec h$}}(i)
}   = 
\big(\big| 
\inv{\redukce{$\vec h$}}(i)
\big|_L, \,
\big|
\inv{\redukce{$\vec h$}}(i)\big|_R\big) \ \hbox { and\ }\
\inv{\redukce{$\vec{\bic{h}}$}}(i) =
\big(\vec{|h|}^{-1}(i),[1]\big),
\end{eqnarray*}
while $\big|\inv{\redukce{$\vec h$}}(i)\big|_L
= \vec{|h|}^{-1}(i)$  
by property of the cardinality 
functor recalled in~\eqref{Za tyden jedu do Berlina.}.
Thus the equality in item~(iii) is equivalent to 
$\big|\vec h^{-1}(i)\big|_R = [1]$,
meaning that the $i$th fiber of $h$ has the biarity $(a,1)$ with 
some $a \in \bbN$.
This however follows from~(ii).
The proof of the second part of the proposition is similar.
\end{proof}

\begin{definition}
\label{Dnes dam kolecko do opravy.}
Let $\ttQ = (\ttQ_L,\ttQ_R)$ be a left bioperadic category and $\oL$
a $\ttQ_F$-operad. A {\em $\oL$-algebra} is a collection  $A = \{A_c\}_{c
  \in \pi_0(\ttQ_F)}$ together with structure operations
\begin{subequations}
\[
a_F : \oL(F) \ot A_{\sou F} \longrightarrow   A_{\tar F},\ F \in \ttQ_F,
\]
where $\sou F$ and $\tar F$ are as
in~(\ref{Budu muset jet na chalupu.}), and the shorthand of~(\ref{Vubec
  se mi tam nechce.}) is used. Moreover, the diagram
\[
\xymatrix@C=4em{
\oP(T) \ot \oP(F_1) \ot A_{\sou {F_1}} \ot  \cdots \ot
\oP(F_n) \ot A_{\sou {F_n}}
\ar[r]^(.59){\id \ot a_{F_1} \ot \cdots \ot a_{F_n}} 
\ar@{<->}[d]_{\rm symmetry}^\cong
&\ar@{=}[d]^(.6){\eqref{2a}}
   \oP(T) \ot  
A_{\tar {F_1}} \ot \cdots \ot  A_{\tar {F_n}}
\\
\ar[d]_{\gamma_h \ot \id}
\oP(T) \ot \oP(F_1) \ot \cdots \ot  \oP(F_n)
\ot A_{\sou {F_1}} \ot \cdots \ot  A_{\sou {F_n}}
& \oP(T) \ot A_{\sou T}\ar[d]^(.6){a_T} 
\\
\ar@{=}[d]_(.55){\rm axiom\ of\ \ttQ_L\ }
\oP(S) \ot A_{\sou {F_1}} \ot \cdots \ot  A_{\sou {F_n}} 
& A_{\tar T} \ar@{=}[d]^{\eqref{2abis}}
\\
\oP(S) \ot A_{\sou S}   \ar[r]^{a_S} &   A_{\tar S}
}
\]
is required to commute for each $\Rada F1n \ \fib S \Xarrow h T$ in
$\ttQ_F$. 

Dually, suppose that $\ttQ$ is right bioperadic and $\cooP$ a
$\ttQ_C$-cooperad. 
A {\em $\cooP$-coalgebra} is a collection $B = \{B_c\}_{c
  \in \pi_0(\ttQ_C)}$ together with structure operations
\[
b_D : \cooP(D) \ot B_{\sou D} \longrightarrow   B_{\tar D},\ D \in \ttQ_C,
\]
\end{subequations}
such that the diagram
\[
\xymatrix@C=3em{
\cooP(D_1) \ot \cdots \ot \cooP(D_m)\ot \cooP(S) \ot B_{\sou S}
\ar[r]^(.55){\id^{\ot m} \ot b_S}
&\ar@{<->}[d]^{\eqref{2b}}_\cong
\cooP(D_1) \ot \cdots \ot \cooP(D_m)\ot B_{\tar S} 
\\
\ar@{=}[d]_{\eqref{2bbis}}\ar[u]^{\delta_h \ot \id}
\cooP(T)  \ot B_{\sou S}
& \cooP(D_1) \ot \cdots \ot \cooP(D_m) \ot B_{\sou {D_1}} \ot \cdots
\ot  B_{\sou {D_m}} \ar@{<->}[d]^{\rm \ symmetry}_\cong
\\
\cooP(T)  \ot B_{\sou T}
\ar[d]_{b_T}
& \cooP(D_1) \ot B_{\sou {D_1}} \ot \cdots \ot \cooP(D_m) \ot B_{\sou {D_m}}
\ar[d]^{b_{D_1} \ot \cdots \ot b_{D_m}}
\\
B_{\tar T} \ar@{=}[r]^(.45){\rm axiom \ of \ \ttQ_R}  &    B_{\tar {D_1}} \ot \cdots
\ot  B_{\tar {D_m}}
}
\]
commutes for an arbitrary $S \Xarrow h T \cof \ \Rada D1m$ in $\ttQ_C$.
\end{definition}

\begin{proposition}
\label{Dnes jsem byl s Jarkou na hrbitove.}
The `bivariant envelope' \/ $\ttO \mapsto \tO$ constructed in
Example~\ref{Pobezim po strese?} defines a faithful embedding of the
category of unital operadic categories to the category of bioperadic
categories $\ttQ = (\ttQ_L,\ttQ_R)$ with unital
$\ttQ_L$.  The categories of\/ $\ttO$- and\/ $\tO$-operads are
isomorphic and so are also the categories of the associated
algebras.
\end{proposition}

\begin{proof}
\label{Pujdeme dnes do toho mlyna?}
We claim that the dioperadic category $\tO = (\tO_L,\tO_R)$ is
bioperadic if $\ttO$ is unital. 
Let us start by verifying~\eqref{2a} for $h: S \to T\in \tO= \ttO$ as in 
\begin{equation}
\label{Uz jsem hodne unaveny.}
\Rada F1n \ \fib S \stackrel h\longrightarrow T \cof \  D,\  D :=U_c,
\ c := \pi_0(T),
\end{equation}
where $\Rada F1n$ are the fibers of $h$ in $\ttO$ and $U_c$ belongs to
the family~\eqref{Jsem posledni den v Cuernavace.}. 
By definition, $C_{F_i}$ is the chosen local
terminal object $U^i_c$ 
in the connected component of $F_i$, while $U^i_T$ is the $i$th
fiber of the identity $\id_T : T \to T$, $1 \leq i \leq n$, that is
\[
\rada {U^1_T}{U^n_T} \ \fib T \stackrel {\id_T}\longrightarrow T .
\]
By~the left unitality of $\ttO$, all $\rada {U^1_T}{U^n_T}$
are chosen local terminal. The functoriality of the fiber functors
gives the map $h_i :F_i \to U^i_T$ in the diagram
\[
\xymatrix@R=-0.3em@C=.8em{F_i
  \ar[rr]^(.55){h_i}   
&&U^i_T \hskip -.5em
\\
\hskip -2em
\hskip 1.9em\raisebox{.7em}{{\rotatebox{270}{$\fib$}}} && 
\hskip -.1em\raisebox{.7em}{{\rotatebox{270}{$\fib$}}} \hskip -.5em
\\
S \ar[rr]^{h}  \ar@/_.9em/[ddddddr]_(.4){h}   &&
T \ar@/^.9em/[ddddddl]^(.4){\id_T} \hskip -.5em
\\&& 
\\&& 
\\&& 
\\&&
\\ &&
\\
&T &
}
\]
thus $F_i$ and $U^i_T$ are in the same connected component, therefore
$U^i_c = U^i_T$ for each $1 \leq i \leq n$. This proves~\eqref{2a}. 
By definition,  $C_S = U_{a}$ and
$C_T  =  U_{b}$, where 
$a :=\pi_0(S)$ and $b := \pi_0(T)$. 
If~$S$  is connected with $T$ by a morphism, $a = b$, so $C_S =
C_T$. This verifies~\eqref{2abis}. 

For $h :S \to T \cof \ D$, $D = U_b$ with $b := \pi_0(T)$, so $U_D$ is
the fiber of $U_b \stackrel {\id_{U_b}} \longrightarrow U_b$,
therefore $U_D = U_b$ by the right unitality of $\ttO$. By definition,
$C_S$ is the cofiber of $S \stackrel {\id_S} \longrightarrow S$, i.e.\
the chosen local terminal object in the connected component of
$S$. Since $S$ and $T$ are connected by a morphism, we conclude that
$C_S = U_b = U_D$, which is~(\ref{2b}). The equalities in~\eqref{2b}
are immediate because~$\tO_C$ is discrete, cf.~Example~\ref{Pobezim po
  strese?}. The numerological assumptions about the biarities are
clearly satisfied, so $\tO$ is bioperadic as claimed.  The rest of the
the proposition is obvious.
\end{proof}

Unlike the Example~\ref{Dva dny za sebou.}, the embedding
given by the bivariant envelope is not full, since there may exist
functors between the envelopes that do not preserve the chosen local
terminals.

\begin{definition}
\label{Navlekl jsem si gumu do prdlavek.}
Let $\oM$ be a \CP-bimodule. An {\/\em $\oM$-trace\/} consists 
of an $\oL$-algebra 
$A$  and a~$\cooP$-coalgebra $B$ as in Definition~\ref{Dnes dam kolecko
  do opravy.}, 
together with structure operations
\[
c_T : \oM(T) \ot A_{\sou T} \longrightarrow   B_{\tar T},\ T \in \ttQ,
\]
such that the diagram
\[
\xymatrix@C=-.5em{
\oM(T) \ot \! A_{\sou T} \ar[r]+/l 3.4em/^(.67){c_T}  \ar@{=}[d]_{\eqref{2a}}  &
 \hskip -2em B_{\tar T}  \hskip 2em  & 
\ar@{=}[l]-/r .7em/_(.72){\rm axiom \ of \ \ttQ_R} B_{\tar {D_1}} \ot \cdots \ot
B_{\tar {D_m}} 
\\
\oM(T) \ot  \! A_{\tar{F_1}} 
\ot \cdots \ot \! A_{\tar {F_n}} && \cooP(D_1) \ot B_{\sou {D_1}} \ot \cdots \ot
\cooP(D_m) \ot B_{\sou {D_m}} \ar[u]_{b_{D_1} \ot \cdots \ot b_{D_m}}
\ar@{<->}[d]^{\ \rm symmetry}_\cong
\\
\ar[u]^(.57){\id \ot a_{F_1} \ot \cdots \ot a_{F_n}}
\oM(T) \ot \oP(F_1) \ot \! A_{\sou{F_1}} \ot \cdots \ot 
\oP(F_n)  \ot \! A_{\sou {F_n}}  && 
\cooP(D_1) \ot \cdots \ot \cooP(D_m)  \ot B_{\sou {D_1}} \ot \cdots
\ot
 B_{\sou {D_m}}
\\
\ar[d]_{\omega_h \ot
  \id^{\ot n}}
\oM(T) \ot \oP(F_1) \ot \cdots \ot \oP(F_n) \ot \! A_{\sou{F_1}} \ot
\cdots \ot \! A_{\sou {F_n}} \ar@{<->}[u]^{\rule{.7em}{0em} \rm symmetry}_\cong && 
\cooP(D_1) \ot \cdots \ot \cooP(D_m) \ot B_{\tar S}
\ar@{=}[u]_{\eqref{2b}}
\\
\cooP(D_1) \ot \cdots \ot \cooP(D_m) \ot \oM(S) 
\ot \!\! A_{\sou{F_1}} \ot \cdots \ot \!\! A_{\sou {F_n}} 
\ar@{=}[rr]^(.55){\ \rm  axiom\ of\ \ttQ_L}&&
\cooP(D_1) \ot \cdots \ot \cooP(D_m) \ot \oM(S) 
\ot \!\! A_{\sou{S}}\ar[u]_{\id^{\ot m} \ot c_S}
}
\]
commutes for each morphism 
$\Rada F1n \ \fib S \stackrel h\to T \cof \ \Rada D1m$ in $\ttQ$.
\end{definition}

Notice that the algebra operations $a_{F_i} : \oP(F_i) \ot A_{\sou
  {F_i}} \to  A_{\tar {F_i}}$ in the diagram in
Definition~\ref{Navlekl jsem si gumu do prdlavek.} are defined, since
$F_i \in \ttQ_F$ for each $i \in [n]$.
The coalgebra operations are defined for the similar reasons.

\begin{example}
\label{Vcera jsem byl na schuzi Neuronu.}
Let us return to  the context of Example~\ref{Zitra jedu pro
  vetvicky.}. 
The traces for the `traditional' right operadic modules were
introduced in~\cite[Definition~2.6]{cyclo}. 
Denote $\Cycl( n) := \Span{\Sigma_{n-1}}$, the $\bbk$-linear
span of the symmetric group on \hbox{$n\!-\!1$} elements. As argued
in~\cite[Example~2.8]{cyclo}, the collection $\Cycl = \{\Cycl (n)\}_{n
  \geq 1}$ is a right module over the operad  \hbox{$\Ass$} for
associative algebras. A~$\Cycl$-trace is given by a map
$T:A \to W$ from an associative algebra $A$ to a vector space $W$ such
that 
\[
T(ab) = T(ba), \ a,b \in A,
\] 
i.e.\ $T$ is a trace of an associative algebra in the traditional sense. 
\end{example}

\begin{example}
\label{Necham rozmrznout angrest z Jarciny zahradky.}
The augmentation ideal  $\overline \Ass$ \/ in the operad $\Ass$ 
for associative algebras bears an obvious right
$\Ass$-action.
An \/ $\overline \Ass$-trace is given by an associative algebra $A$, a
vector space $B$ and a bilinear map \hbox{$T: A \ot A \to
B$} satisfying
\[
T(ab,c) = T(a,bc), \  \hbox { for all }  a,b,c \in A.
\]
More interesting examples of traces involving Stasheff's associahedra
acting on the cyclohedra can be found in~\cite[Section~2]{cyclo}.
\end{example}

\section*{Epilogue}

\lettrine{\color{red} T} {he} 
structures studied so far have alternatives in
the form of objects $\ttQ = (\ttQ',\ttQ'')$, where $\ttQ'$ and
$\ttQ''$ are operadic categories sharing the same underlying
category $\ttQ$. Morphisms of $\ttO$ would then possess two sets of fibers, 
so instead of~\eqref{Dnes bude cely den prset.} we would have
\begin{equation}
\label{Uz zacina kruta zima.}
\genfrac{}{}{0pt}{3}{\hbox{$\Rada {F'}1n\ \fib$}}{\hbox{$\Rada
    {F''}1m\ \fib$}}\big\} \ S \stackrel h \longrightarrow T .
\end{equation}
Given a $\ttQ'$-operad $\oP'$ and a $\ttQ''$-operad $\oP''$, the
bimodules in Definition~\ref{Dnes jsem byl s Jarkou na obede.} would be
replaced by `contramodules' $\oM = \{\oM(S)\}_{S \in \ttQ}$ with structure operations
\begin{equation}
\label{Uz mam docela hlad.}
\omega_h:
\oP'(F'_1) \ot \cdots \ot \oP'(F'_n) \ot \oM(T) \ot 
\oP''(F''_1) \ot \cdots \ot \oP''(F''_m) \longrightarrow 
\oM(S)
\end{equation}
given for any $h : S \to T$ in~\eqref{Uz zacina kruta zima.},
visualized in Figure~\ref{vylet do valecne zony}.  
\begin{figure}
  \centering
\psscalebox{1.0 1.0} % Change this value to rescale the drawing.
{
\begin{pspicture}(0,-2.6286852)(11.88,2.6286852)
\psline[linecolor=black, linewidth=0.04](1.14,2.4086852)(0.74,2.4086852)
\psline[linecolor=black, linewidth=0.04](1.14,2.2086854)(0.74,2.2086854)
\psline[linecolor=black, linewidth=0.04](1.14,2.0086854)(0.74,2.0086854)
\psline[linecolor=black, linewidth=0.04](1.14,1.2086853)(0.74,1.2086853)
\psline[linecolor=black, linewidth=0.04](1.14,-1.1913147)(0.74,-1.1913147)
\psline[linecolor=black, linewidth=0.04](1.14,-1.3913147)(0.74,-1.3913147)
\psline[linecolor=black, linewidth=0.04](1.14,-1.5913147)(0.74,-1.5913147)
\psline[linecolor=black, linewidth=0.04](1.14,-2.3913147)(0.74,-2.3913147)
\psline[linecolor=black, linewidth=0.04](9.54,2.6086853)(9.54,-2.5913148)(10.74,-2.5913148)(10.74,2.6086853)(9.54,2.6086853)
\psline[linecolor=black, linewidth=0.04](10.74,2.4086852)(11.14,2.4086852)
\psline[linecolor=black, linewidth=0.04](10.74,2.2086854)(11.14,2.2086854)
\psline[linecolor=black, linewidth=0.04](10.74,2.0086854)(11.14,2.0086854)
\psline[linecolor=black, linewidth=0.04](10.74,1.8086853)(11.14,1.8086853)
\psline[linecolor=black, linewidth=0.04](10.74,1.6086853)(11.14,1.6086853)
\psline[linecolor=black, linewidth=0.04](10.74,1.4086853)(11.14,1.4086853)
\psline[linecolor=black, linewidth=0.04](10.74,-1.1913147)(11.14,-1.1913147)
\psline[linecolor=black, linewidth=0.04](10.74,-1.3913147)(11.14,-1.3913147)
\psline[linecolor=black, linewidth=0.04](10.74,-1.5913147)(11.14,-1.5913147)
\psline[linecolor=black, linewidth=0.04](10.74,-1.7913147)(11.14,-1.7913147)
\psline[linecolor=black, linewidth=0.04](10.74,-1.9913146)(11.14,-1.9913146)
\psline[linecolor=black, linewidth=0.04](10.74,-2.1913147)(11.14,-2.1913147)
\psline[linecolor=black, linewidth=0.04](10.74,-2.3913147)(11.14,-2.3913147)
\psline[linecolor=black, linewidth=0.04](10.74,1.2086853)(11.14,1.2086853)
\rput(1.54,1.8086853){$F'_1$}
\rput(1.54,-1.7913147){$F'_n$}
\rput(10.14,0.008685303){$S$}
\rput[r](0.34,1.8086853){$a_1$}
\rput(1.14,1.6086853){$\cdot$}
\rput(1.14,1.4086853){$\cdot$}
\rput[l](11.74,0.008685303){$b$}
\rput(10.94,0.8086853){$\cdot$}
\rput(10.94,0.6086853){$\cdot$}
\rput(10.94,0.4086853){$\cdot$}
\rput(0.94,1.8086853){$\cdot$}
\rput(0.94,1.6086853){$\cdot$}
\rput(0.94,1.4086853){$\cdot$}
\rput(10.94,0.20868531){$\cdot$}
\rput(10.94,0.008685303){$\cdot$}
\rput(10.94,-0.1913147){$\cdot$}
\rput(0.94,-1.7913147){$\cdot$}
\rput(0.94,-1.9913146){$\cdot$}
\rput(0.94,-2.1913147){$\cdot$}
\rput(2.14,0.6086853){$\cdot$}
\rput(2.14,0.4086853){$\cdot$}
\rput(2.14,0.20868531){$\cdot$}
\rput(2.14,0.008685303){$\cdot$}
\rput(2.14,-0.1913147){$\cdot$}
\rput(2.14,-0.3913147){$\cdot$}
\rput(2.14,-0.5913147){$\cdot$}
\rput(10.94,-0.3913147){$\cdot$}
\rput(0.54,-1.7913147){$\left\{\rule{0em}{2em}\right.$}
\rput[r](0.34,-1.7913147){$a_n$}
\rput(0.54,1.8086853){$\left\{\rule{0em}{2em}\right.$}
\rput(11.54,0.008685303){$\left.\rule{0em}{7.2em}\right\}$}
\rput(10.94,-0.5913147){$\cdot$}
\rput(10.94,-0.7913147){$\cdot$}
\psline[linecolor=black, linewidth=0.04](5.94,2.4086852)(6.34,2.4086852)
\psline[linecolor=black, linewidth=0.04](5.94,2.2086854)(6.34,2.2086854)
\psline[linecolor=black, linewidth=0.04](5.94,2.0086854)(6.34,2.0086854)
\psline[linecolor=black, linewidth=0.04](5.94,1.2086853)(6.34,1.2086853)
\psline[linecolor=black, linewidth=0.04](5.94,-1.1913147)(6.34,-1.1913147)
\psline[linecolor=black, linewidth=0.04](5.94,-1.3913147)(6.34,-1.3913147)
\psline[linecolor=black, linewidth=0.04](5.94,-1.5913147)(6.34,-1.5913147)
\psline[linecolor=black, linewidth=0.04](5.94,-2.3913147)(6.34,-2.3913147)
\psline[linecolor=black, linewidth=0.04](4.14,2.6086853)(4.14,-2.5913148)(2.94,-2.5913148)(2.94,2.6086853)(4.14,2.6086853)
\psline[linecolor=black, linewidth=0.04](9.54,2.4086852)(9.14,2.4086852)
\psline[linecolor=black, linewidth=0.04](9.54,2.2086854)(9.14,2.2086854)
\psline[linecolor=black, linewidth=0.04](9.54,2.0086854)(9.14,2.0086854)
\psline[linecolor=black, linewidth=0.04](9.54,1.8086853)(9.14,1.8086853)
\psline[linecolor=black, linewidth=0.04](9.54,1.6086853)(9.14,1.6086853)
\psline[linecolor=black, linewidth=0.04](9.54,1.4086853)(9.14,1.4086853)
\psline[linecolor=black, linewidth=0.04](9.54,-1.1913147)(9.14,-1.1913147)
\psline[linecolor=black, linewidth=0.04](9.54,-1.3913147)(9.14,-1.3913147)
\psline[linecolor=black, linewidth=0.04](9.54,-1.5913147)(9.14,-1.5913147)
\psline[linecolor=black, linewidth=0.04](9.54,1.2086853)(9.14,1.2086853)
\rput(5.54,1.8086853){$F''_1$}
\rput(5.54,-1.7913147){$F''_m$}
\rput(3.54,0.008685303){$T$}
\rput[l](6.74,1.8086853){$b_1$}
\rput(5.94,1.6086853){$\cdot$}
\rput(5.94,1.4086853){$\cdot$}
\rput[r](8.54,0.008685303){$a$}
\rput(9.34,0.8086853){$\cdot$}
\rput(9.34,0.6086853){$\cdot$}
\rput(9.34,0.4086853){$\cdot$}
\rput(6.14,1.8086853){$\cdot$}
\rput(6.14,1.6086853){$\cdot$}
\rput(6.14,1.4086853){$\cdot$}
\rput(9.34,0.20868531){$\cdot$}
\rput(9.34,0.008685303){$\cdot$}
\rput(9.34,-0.1913147){$\cdot$}
\rput(6.14,-1.7913147){$\cdot$}
\rput(6.14,-1.9913146){$\cdot$}
\rput(6.14,-2.1913147){$\cdot$}
\rput(5.14,0.6086853){$\cdot$}
\rput(5.14,0.4086853){$\cdot$}
\rput(5.14,0.20868531){$\cdot$}
\rput(5.14,0.008685303){$\cdot$}
\rput(5.14,-0.1913147){$\cdot$}
\rput(5.14,-0.3913147){$\cdot$}
\rput(5.14,-0.5913147){$\cdot$}
\rput(9.34,-0.3913147){$\cdot$}
\rput[l](6.74,-1.7913147){$b_m$}
\rput{-178.4686}(13.125537,-3.4072094){\rput(6.54,-1.7913147){$\left\{\rule{0em}{2em}\right.$}}
\rput{-178.4686}(13.029327,3.7915046){\rput(6.54,1.8086853){$\left\{\rule{0em}{2em}\right.$}}
\rput(8.74,0.008685303){$\left\{\rule{0em}{7.2em}\right.$}
\rput(9.34,-0.5913147){$\cdot$}
\rput(9.34,-0.7913147){$\cdot$}
\rput[l](7,0.1){\Large$\stackrel \omega\longrightarrow$}
\psline[linecolor=black, linewidth=0.04](1.14,2.6086853)(1.14,1.0086854)(2.34,1.8086853)(1.14,2.6086853)
\psline[linecolor=black, linewidth=0.04](1.14,-0.9913147)(1.14,-2.5913148)(2.34,-1.7913147)(1.14,-0.9913147)
\psline[linecolor=black, linewidth=0.04](2.34,1.8086853)(2.94,1.8086853)
\psline[linecolor=black, linewidth=0.04](2.34,-1.7913147)(2.94,-1.7913147)
\psline[linecolor=black, linewidth=0.04](5.94,2.6086853)(5.94,1.0086854)(4.74,1.8086853)(5.94,2.6086853)(5.94,2.6086853)
\psline[linecolor=black, linewidth=0.04](5.94,-0.9913147)(5.94,-2.5913148)(4.74,-1.7913147)(5.94,-0.9913147)(5.94,-0.9913147)
\psline[linecolor=black, linewidth=0.04](4.74,1.8086853)(4.14,1.8086853)
\psline[linecolor=black, linewidth=0.04](4.74,-1.7913147)(4.14,-1.7913147)(4.14,-1.7913147)
\psline[linecolor=black, linewidth=0.04](9.54,-1.7913147)(9.14,-1.7913147)
\psline[linecolor=black, linewidth=0.04](9.54,-1.9913146)(9.14,-1.9913146)
\psline[linecolor=black, linewidth=0.04](9.54,-2.1913147)(9.14,-2.1913147)
\psline[linecolor=black, linewidth=0.04](9.54,-2.3913147)(9.14,-2.3913147)
\end{pspicture}
}
\caption{A $\frac12$PROPic version of Figure~\ref{Porad mne pali ty
    ruce.}. The rectangles
represent the  corresponding pieces of
$\oM$, the
triangles the pieces of $\oP'$ and $\oP''$, respectively.
The equalities $a = a_1 + \cdots + a_n$ and $b = b_1 + \cdots + b$ are assumed.}
\label{vylet do valecne zony}
\end{figure}
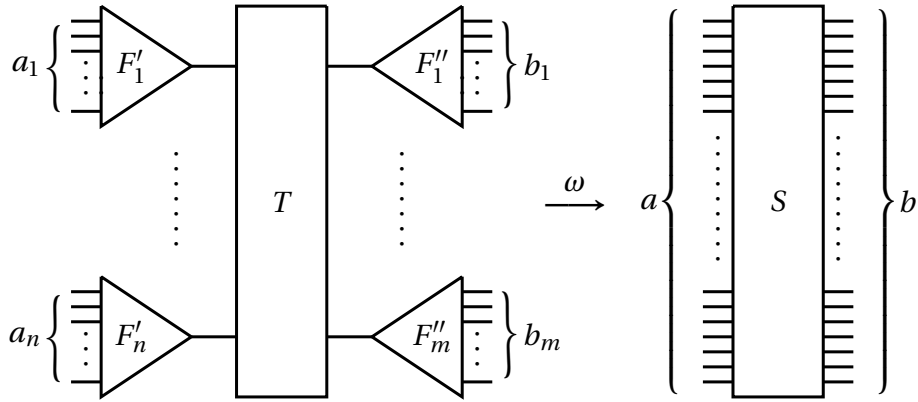

A typical example of $\ttQ = (\ttQ',\ttQ'')$ is the category 
$\sFSet^2 := \sFSet \times \sFSet$. The structure
operations~(\ref{Uz mam docela hlad.}) are of the form
\[
\oP'(a_1) \ot \cdots \ot \oP'(a_n) \ot \oM(n,m) \ot 
\oP''(b_1) \ot \cdots \ot \oP''(b_n) \longrightarrow 
\oM(a_1 + \cdots + a_n,b_1 + \cdots + b_m)
\]
and defined for all $n,m,\Rada a1n,\Rada b1m \in \bbN$. 

\begin{example}
Each $\frac12$PROP  $\prop =
\{\prop(n.m)\}_{n,m \in \bbN}$, cf.~\cite[Definition~1]{mv}, contains
the `lower' and the `upper' suboperads $\prop^\uparrow =  
\{\prop^\uparrow(n)\}_{n \in \bbN}$ \ and \ 
$\prop^\downarrow = \{\prop^\downarrow(m)\}_{m \in \bbN}$
given by
\[
\prop^\uparrow(n) := \prop(n,1) \ \hbox { and } \
\prop^\downarrow(n) := \prop(1,m) \ \hbox { for }  n,m \in \bbN.
\]
The triple $(\prop^\uparrow,\prop,\prop^\downarrow)$ is a natural 
$\sFSet^2$-contramodule.
\end{example}

One may wish to develop a 
theory parallel to that for di- and bioperadic categories
also for the above structures generalizing 
$\frac12$PROPs, but we won't do it here because in our opinion 
the time is not yet ripe for that.
A natural question is whether also fully-fledged  PROPs or at least
properads 
can be described as operad-type objects over an operatic-like
category. We don't have the slightest idea if or how to do it. 
We welcome any suggestion in this direction.

%\listoftodos

\end{document}